\documentclass[reqno]{amsart}
\usepackage[utf8]{inputenc}
\usepackage{amsmath,amssymb,amsthm,amsfonts}
\usepackage[dvipsnames]{xcolor}
\usepackage{mathrsfs}
\usepackage{verbatim}
\usepackage{mathtools}
\usepackage{float}
\usepackage{bbm}
\usepackage{enumerate}
\usepackage[colorlinks=true, citecolor=blue]{hyperref}
\usepackage[german,english]{babel}
\usepackage[notref,notcite,final]{showkeys}
\usepackage{tikz-cd}
\usepackage[normalem]{ulem}

\numberwithin{equation}{section}
\numberwithin{table}{section}
\numberwithin{figure}{section}

\newtheorem{theorem}{Theorem}[section]
\newtheorem{corollary}[theorem]{Corollary}\newtheorem{proposition}[theorem]{Proposition}
\newtheorem{lemma}[theorem]{Lemma}
\theoremstyle{definition}
\newtheorem{definition}[theorem]{Definition}
\newtheorem{remark}[theorem]{Remark}
\newtheorem{example}[theorem]{Example}

\newcommand{\id}{\mathop{\textrm{\upshape{id}}}}

\renewcommand{\Re}{\mathop{\textrm{\upshape{Re}}}}
\renewcommand{\Im}{\mathop{\textrm{\upshape{Im}}}}

\newcommand{\orange}[1]{{\color{Orange}#1}}

\newcommand{\supp}{\mathop{\textrm{\upshape{supp}}}}

\renewcommand{\phi}{\varphi}
\newcommand{\eps}{\varepsilon}

\newcommand{\forget}[1]{}

\newcommand{\spk}[1]{\langle #1 \rangle}

\newcommand{\R}{\mathbb{R}}
\newcommand{\rz}{\mathbb{R}}
\newcommand{\C}{\mathbb{C}}
\newcommand{\N}{\mathbb{N}}

\newcommand{\E}{\mathbb{E}_\lambda}
\newcommand{\F}{\mathbb{F}_\lambda}
\newcommand{\BUC}{\mathop{\textrm{\upshape{BUC}}}}
\renewcommand{\epsilon}{\varepsilon}

\newcommand{\wt}[1]{\widetilde{#1}}

\newcommand{\y}{y} % variable half space
\newcommand{\V}{V} % covering half space
\newcommand{\A}{A} % local operators
\newcommand{\B}{B}
\newcommand{\acf}{a} % local coefficients
\newcommand{\bcf}{b}
\newcommand{\diffeo}{\vartheta} % diffeomorphism from domain to half space
\newcommand{\vhs}{v} % solution transformed to half space, that is v=\diffeo^{-1}\circ u
\newcommand{\EE}{\mathcal E_\lambda} % solution space with standard Sobolev spaces
\newcommand{\FF}{\mathcal F_\lambda} % right-hand side space with standard Sobolev spaces
\newcommand{\Lhs}{L} % solution operator in the half-space
\newcommand{\Lom}{L^\Omega} % solution operator in the domain
\newcommand{\phiom}{\varphi^\Omega} % partition of unity in domain
\newcommand{\psihs}{\psi} % partition of unity in the half-space
\newcommand{\psiom}{\psi^\Omega} % partition of unity in the domain
\allowdisplaybreaks

\begin{document}
\title[Boundary value problems with rough boundary data]{Boundary value problems \\ with rough boundary data}

\author{Robert Denk}
\address{R.\ Denk, Universit\"at Konstanz, Fachbereich f\"ur Mathematik und Statistik, Konstanz, Germany}
\email{robert.denk@uni-konstanz.de}

\author{David Plo\ss}
\address{D.\ Plo\ss, Universit\"at Konstanz, Fachbereich f\"ur Mathematik und Statistik, Konstanz, Germany}
\email{david.ploss@uni-konstanz.de}

\author{Sophia Rau}
\address{S.\ Rau, Universit\"at Konstanz, Fachbereich f\"ur Mathematik und Statistik, Konstanz, Germany}
\email{sophia.rau@uni-konstanz.de}

\author{J{\"o}rg Seiler}
\address{J. Seiler, Università di Torino, Dipartimento di Matematica, V. Carlo Alberto 10, 10123 Torino, Italy}
\email{joerg.seiler@unito.it}

\begin{abstract}
We consider linear boundary value problems for higher-order para\-meter-elliptic equations, where the boundary data do  not belong to the classical trace spaces.
We employ a class of Sobolev spaces of mixed smoothness that admits a generalized boundary trace with values in Besov spaces of negative order.
%\sout{To show unique solvability, we consider Sobolev spaces of mixed smoothness. This class of anisotropic Sobolev spaces admits a generalized boundary trace  with values in Besov spaces of negative order.}
We prove unique solvability for rough boundary data in the half-space and in sufficiently smooth domains. As an application, we show that the operator related to the linearized Cahn--Hilliard equation with dynamic boundary conditions generates a holomorphic semigroup in $L^p(\R^n_+)\times L^p(\R^{n-1})$. %\sout{, with the basic space being the product of $L^p$-spaces.}
\end{abstract}

\date{September 17, 2020}

\date{\today}

\maketitle

\addtocounter{section}{0}

\section{Introduction}
%\rot{Korrektur lesen!}
In the present paper, we study linear differential boundary value problems of the form
\begin{equation}
  \label{0-1}
  \begin{aligned}
    (\lambda - A) u & = f \quad \text{ in } \Omega,\\
    B_j u & = g_j \quad (j=1,\dots,m)\; \text{ on }\Gamma,
  \end{aligned}
\end{equation}
where $\Omega$ is either the half-space $\R^n_+ := \{x\in\R^n: x_n>0\}$ or a domain in $\R^n$
with compact and sufficiently smooth boundary $\Gamma$. %The operators $A$ and $B_j$ are linear differential operators of order $2m$ and linear boundary operators of order $m_j<2m$.
Moreover, $A$ is a differential operator of order $2m$ and $B_j$ is a boundary operator of order $m_j<2m$ for $j=1,\ldots, m$.
Whereas for sufficiently smooth $f$ and $g_j$ this problem can be solved by classical theory, we focus on rough boundary data $g_1,\dots,g_m$. In particular, we want to solve \eqref{0-1}
for $f\in L^p(\Omega)$ but $g_j\in B_{pp}^{s_j}(\Gamma)$, where $s_j$ may be zero or even negative.
For such rough boundary data, even the formulation of the boundary conditions needs justification: It is
known that the classical trace $u\mapsto u|_\Gamma$, first defined for smooth functions, has a
continuous extension to an operator $\gamma_0\colon H^{s}_p(\Omega)\to B_{pp}^{s-1/p}(\Gamma)$
\emph{if and only if} $s>\frac 1p$ (\cite{Johnsen-Sickel08}). Nevertheless, it is possible to define a continuous trace
on subspaces of $H^s_p(\Omega)$ for $s\le \frac1p$, see, e.g.,
Lions--Magenes (\cite{Lions-Magenes62}, \cite{Lions-Magenes72}) and Roitberg (\cite{Roitberg96}, \cite{Roitberg99}).
In the present paper, we will introduce a class of Sobolev spaces $H_p^{s,\sigma}(\R^n)$
of anisotropic type, for which the trace exists as a continuous operator, following
the ideas from Grubb (\cite{Grubb96}, \cite{Grubb-Kokholm93}).

The motivation to study problem \eqref{0-1} with rough boundary data is two-fold: The first motivation arises
in the study of
stochastic partial differential equations (SPDEs) with boundary noise. Exemplarily, we
mention here \cite{Schnaubelt-Veraar11} and  \cite{Mohammed-Bloemker16} for parabolic equations and reaction-diffusion systems with Neumann boundary conditions,
\cite{Brzezniak-Goldys-Peszat-Russo15} for the heat equation with Dirichlet boundary
conditions,
\cite{Agresti-Hieber-Hussein-Saal21} and \cite{Binz-Hieber-Hussein-Saal20}
for a free boundary value problem in fluid mechanics, and
\cite{Chueshov-Schmalfuss07} for dynamical boundary conditions. The key step in the
analysis of these problems is to understand the properties of the solution operator
to the boundary value problem (formulated for Neumann boundary conditions)
\begin{equation}
  \label{0-2}
  \begin{alignedat}{4}
  \partial_t u - A u & = 0 &&\text{ in }(0,\infty) \times \Omega,\\
  \partial_\nu u & = \xi &&\text{ on } (0, \infty) \times \Gamma,
  \end{alignedat}
\end{equation}
where $\xi$ stands for the boundary noise and $\partial_\nu$ denotes the derivative in the direction of the outward pointing unit normal vector of the boundary $\Gamma$. As it is known that the paths of Gaussian white noise belong with probability one to some Besov space with negative regularity
(see, e.g., \cite{Aziznejad-Fageot20}, \cite{Hummel21}, \cite{Veraar11}), this fits into the setting of \eqref{0-1}
with $f=0$. In the context of SPDEs, the solution operator is often denoted as the Neumann (or Dirichlet) map.

The second motivation for studying \eqref{0-1} arises from boundary value problems with
 Wentzell or dynamical boundary conditions.  As a prototype example, we consider
 the heat equation with Wentzell boundary conditions
\begin{equation}
  \label{0-3}
  \begin{alignedat}{4}
  \partial_t u - \Delta u & = 0 &&\text{ in }(0,\infty) \times \Omega,\\
  \Delta u + \partial_\nu u & = 0 &&\text{ on }(0, \infty) \times \Gamma,\\
  u|_{t=0} & = u_0 &&\text{ in }\Omega.
  \end{alignedat}
\end{equation}
Replacing $\Delta u=\partial_t u$ in the boundary condition, we obtain the dynamic
boundary condition $\partial_t u +\partial_\nu u =0$. In a standard approach, one decouples $u=:u_1$ and $u|_\Gamma =:u_2$ and
obtains a resolvent problem of the form
\begin{equation}
  \label{0-4}
  \begin{alignedat}{4}
  \lambda u_1 - \Delta u_1 & = f &&\text{ in }\Omega,\\
  \lambda u_2 + \partial_\nu u_1 & = g &&\text{ on }\Gamma
 \end{alignedat}
\end{equation}
with the additional condition $u_1|_{\Gamma} = u_2$. The corresponding operator
acts on the tuple $u=(u_1,u_2)$ as $Au = (\Delta u_1, -\partial_\nu u_1)$.
From the point of view of maximal regularity for \eqref{0-3}, the basic space for
this operator would be $L^p(\Omega)\times B_{pp}^{1-1/p}(\Gamma)$, where the
second component is the trace space of $H_p^2(\Omega)$ for the Neumann boundary operator.
In fact, for boundary value problems with dynamic boundary conditions, the generation
of a holomorphic semigroup in trace spaces was shown in \cite{Pruess-Racke-Zheng06} for
the Cahn--Hilliard equation and in \cite{Denk-Pruess-Zacher08} for a general class of problems. However, a more natural basic space for the operator $A$ is $L^p(\Omega)\times L^p(\Gamma)$. At least for $p=2$, form methods can easily lead to the proof of
the generation of a holomorphic semigroup. This was elaborated, e.g., for second-order
equations in \cite{Arendt-Metafune-Pallara-Romanelli03} and in
\cite{Warma13}, for the Bi-Laplacian
in \cite{Denk-Kunze-Ploss21},  and in an abstract setting in  \cite{Engel-Fragnelli11}.
For the analysis in the basic space $L^p(\Omega)\times L^p(\Gamma)$, one has to deal
with boundary values in $L^p$-spaces, which again is not covered by classical theory.
In the present paper, we will apply our solution theory to  the Cahn--Hilliard equation
with dynamic boundary conditions.

Our analysis of \eqref{0-1} starts with the observation that (at least in the smooth
situation) this problem fits
into the framework of Boutet de Monvel's calculus of pseudodifferential boundary value problems. In this calculus, the solution operator for $f=0$ is called a Poisson operator,
and such operators have good mapping properties in the complete scale of Sobolev spaces.
This follows, e.g., from the work of Grubb (\cite{Grubb95}, \cite{Grubb96}) and
Grubb and Kokholm \cite{Grubb-Kokholm93}.
However, the classical trace only exists for sufficiently smooth functions. Therefore, one has to define an appropriate generalization of the trace on the boundary.
Versions of generalized traces were considered by Lions and Magenes (\cite{Lions-Magenes62},
\cite{Lions-Magenes72}), using duality, and by Roitberg (\cite{Roitberg96},
\cite{Roitberg99}), using completion of smooth functions, see also Remark~\ref{1.9}
below.

In this paper, we use another approach to a generalized trace by considering a new class of Sobolev
spaces with mixed smoothness (see also \cite{Grubb96} for $p=2$). These spaces
differ from anisotropic Sobolev spaces in the sense of  \cite{Johnsen-Sickel08} and  \cite{Triebel06} and from spaces with dominating mixed smoothness in the sense
of \cite{Schmeisser07} and \cite{Triebel19}. For this class of Sobolev spaces,
both the existence of a continuous trace and the unique solvability of
parameter-elliptic model problems in the whole space and in the half-space follow
immediately from known results. However, the passage from model problems (i.e.,
constant coefficients and no lower-order terms) to variable coefficients is
not standard. It requires the application of an elaborate localization procedure, even for problems on the half-space.
% \sout{, due to the non-classical spaces considered. Therefore, an elaborate
% analysis of the localization procedure is necessary -- even in the half-space.}
% \rot{Hier vielleicht noch was, dass wir aufpassen müssen, wann man $\delta$ und wann man $\lambda_0$ wählt.} \grun{Ist glaube ich doch hier zu kompliziert. Robert}
In domains, the definition of the Sobolev spaces with mixed smoothness is not canonical. Therefore, we work with classical Sobolev spaces in domains but employ local embeddings  into our spaces of mixed smoothness. The necessity to estimate certain commutators leads to restrictions on the orders of the involved spaces (see Lemma \ref{3.10}); however these restrictions still allow to deal with boundary values
in $L^p(\Gamma)$, for example.

The paper is structured as follows.
In Section~\ref{sec2}, we define and analyze Sobolev spaces of mixed smoothness,
including parameter-dependent norms. We show trace results and typical embeddings.
In Lemma~\ref{1.5}, interpolation properties are shown which seem not to follow
immediately from known results. Section~\ref{sec3} deals with boundary value
problems in the half-space. %\sout{Whereas the model problems can be handled easily, the analysis for variable coefficients is much more elaborate.}
The main result
(Theorem~\ref{2.10}) gives unique solvability of parameter-elliptic boundary
value problems under appropriate smoothness assumptions on the coefficients. Note that
we do not consider the infinitely smooth setting and thus pseudodifferential theory
cannot be applied. Here, the boundary data may belong to Besov spaces with arbitrary
low order. As a corollary, one obtains unique solvability in classical Sobolev spaces
(Corollary~\ref{2.11}). The situation in domains is studied in Section~\ref{sec4}. The main result
(Theorem~\ref{3.12}) yields unique solvability in classical Sobolev spaces
for rough boundary data. %\sout{As we avoid Sobolev spaces with mixed smoothness in domains, we obtain some restrictions on the orders of the spaces.}
Finally, in Section~\ref{sec5}, we apply the above results to the linearized
Cahn--Hilliard equation with dynamic boundary conditions. We show that the related
operator $A$ generates a holomorphic semigroup in $L^p(\R^n_+)\times L^p(\R^{n-1})$,
see Theorem~\ref{4.5}. In fact, we even show that, for every $\lambda_0>0$, the
operator $A-\lambda_0$ generates a \emph{bounded} holomorphic semigroup of angle $\frac\pi2$. In the proof, we use the bounded $H^\infty$-calculus for
the Laplacian with explicit symbol estimates, see Lemma~\ref{4.3}. The same method
can be applied to the (much easier) boundary value problem \eqref{0-4},
and we obtain unique solvability of \eqref{0-4} and the generation of a holomorphic semigroup for the related operator.

\section{Sobolev spaces of mixed smoothness and traces}
\label{sec2}

Let us fix some notation used throughout the paper. We consider the Euclidean space $\R^n$ with variable $x=(x',x_n)$ and corresponding co-variable $\xi=(\xi',\xi_n)$. We fix $m\in\N$ and define
\begin{align*}
    \langle \xi,\lambda\rangle := (1+|\xi|^2+|\lambda|^{1/m})^{1/2} \textnormal{ and } \langle \xi',\lambda\rangle := (1+|\xi'|^2+|\lambda|^{1/m})^{1/2}
\end{align*}
for $\xi\in\R^n$ and $\lambda\in\C$. Moreover, we write $\langle \xi\rangle:= \spk{\xi,0}$,
$\langle \xi'\rangle:= \spk{\xi',0}$ and $\spk{\lambda}:=\spk{0,\lambda}$.
For (suitable) functions $\phi(\xi)$ defined on $\R^n$ we denote by $\phi(D)$ its associated Fourier multiplier, which is defined by $\phi(D)u := \mathscr F^{-1} (\phi \mathscr F u)$, where $\mathscr F$ denotes the Fourier transform acting in the space $\mathscr S'(\R^n)$ of tempered distributions.
In particular, we have $ D^\alpha = (-i)^{|\alpha|}\partial^\alpha $ for $\alpha\in\N_0^n$. In case $\phi(\xi)=\phi(\xi')$ is independent of $\xi_n$, the associated Fourier multiplier will also be denoted by $\phi(D')$.

For two Banach spaces $X$ and $Y$ let $L(X,Y)$ be the space of bounded linear operators $X\to Y$ and $L(X):=L(X,X)$. We shall write $X=Y$ if both spaces have the same elements and equivalent norms, and we write $X\subset Y$ if $X$ is a subset of $Y$ and the inclusion map $X\to Y$ is bounded.

% \forget{
% We fix $m\in\N$ and define
% $\langle \xi\rangle := (1+|\xi|^2)^{1/2}$ for $\xi\in\R^n$ and
% $\langle \xi,\lambda\rangle := (1+|\xi|^2+|\lambda|^{1/m})^{1/2}$ for $\xi\in\R^n,\,\lambda\in\C$.
% We also set $\langle \xi'\rangle := \langle (\xi',0)\rangle$ and $\langle \xi',\lambda\rangle :=
% \langle (\xi',0),\lambda\rangle$ for $\xi'=(\xi_1,\dots,\xi_{n-1})^\top \in \R^{n-1}$ and $\lambda\in\C$.
% For (suitable) functions $\phi(\xi)$ defined on $\R^n$ we denote by $\phi(D)$ its associated Fourier multiplier, which is defined by $\phi(D) := \mathscr F^{-1} \phi \mathscr F$, where $\mathscr F$ denotes the $n$-dimensional Fourier transform acting in the space $\mathscr S'(\R^n)$ of tempered distributions. Similarly, with functions $\psi(\xi')$ we associate $\psi(D'):= (\mathscr F')^{-1} \psi \mathscr F'$ with the $(n-1)$-dimensional Fourier transform $\mathscr F'$. In particular, we will consider $\spk{D'}^\sigma = (\mathscr F')^{-1} (1+|\cdot|^2)^{\sigma/2}\mathscr F'$ for $\sigma\in\R$.
% }

\subsection{Some function spaces}
\label{subsec1.1}

In the following, let $H_p^s(\R^n)$ and $B_{pp}^s(\R^n)$ denote the standard Bessel potential and Besov spaces for $s\in\R$. Throughout the paper, we assume $p \in (1, \infty)$. For $p=2$, the following definition can also be found in \cite{Grubb96}, Appendix~A.3.

\begin{definition}\label{1.1}
  For $s,\sigma\in\R$ and $p\in (1,\infty)$ we define the Bessel potential space of mixed smoothness
  $H_p^{s,\sigma}(\R^n)$ as
  \begin{align*}
    H_p^{s,\sigma}(\R^n)
    &= \{ u\in\mathscr S'(\R^n)\,:\; \spk{D'}^\sigma u \in H_p^s(\R^n)\}\\
    &= \{ u\in\mathscr S'(\R^n)\,:\; \spk{D}^s\spk{D'}^\sigma u \in L^p(\R^n)\}
  \end{align*}
  with canonical norm $\|u\|_{H_p^{s,\sigma}(\R^n)} := \|\spk{D'}^\sigma u\|_{H_p^s(\R^n)}=\|\spk{D}^s\spk{D'}^\sigma u\|_{L^p(\R^n)}$.
  %Analogously, we define the Besov space $B_{pp}^{s,\sigma}(\R^n)$. \rot{Gibt es bei uns gemischte Besovräume?}
\end{definition}
In the previous definition, $\spk{D'}^\sigma$ acts only on the $x'$-variable. Therefore, the above spaces have different smoothness in $x'$-direction and in $x_n$-direction, which is the reason of the notion of mixed smoothness.

Clearly, $H_p^{s,\sigma}(\R^n)$ is a Banach space. Since $\spk{D'}^\sigma$ and $\spk{D}^s$ leave $\mathscr S(\R^n)$ invariant, the rapidly decreasing functions are a dense subset of  $H_p^{s,\sigma}(\R^n)$. For every $t,\tau\in\R$ the map
  \[ \spk{D}^{t}\spk{D'}^{\tau}\colon H_p^{s,\sigma}(\R^n)\to H_p^{s-t,\sigma-\tau}(\R^n)\]
is an isometric isomorphism with inverse $\spk{D}^{-t}\spk{D'}^{-\tau}$.
%The same observations apply to the scale of Besov spaces.

We remark that the scale  $H_p^{s,\sigma}(\R^n)$ for $s,\sigma\in\R$ is different from the scale of anisotropic spaces $H_p^{s,\vec{a}}(\R^n)$ in the sense of \cite{Johnsen-Sickel08}, Proposition~2.10 (see also \cite{Triebel06}, Section~5.1.3). In particular, for $s>0$ and $\sigma<-s$, we have positive smoothness with respect to $x_n$ but negative smoothness in $x'$, which is not allowed for the anisotropic spaces. $H_p^{s,\sigma}(\R^n)$ is also different from the space of dominating mixed smoothness (see, e.g.,  \cite{Triebel19}, Section~1.1.2). Spaces of dominating mixed smoothness are defined similarly as above, but with $\spk{\xi}$ being replaced by $\prod_{j=1}^n (1+\xi_j^2)^{1/2}$. We refer to \cite{Schmeisser07}, Section~1, and \cite{Hummel21}, Subsection~2.2, for further information on spaces of dominating mixed smoothness and applications to boundary value problems.

We state some elementary properties of the Bessel potential spaces with mixed smoothness which we shall use frequently later on.

% \forget{
% \begin{remark}
%   \label{1.3}
%   a) By definition, for all $s,s_1,\sigma,\sigma_1\in\R$, the map
%   \[ \spk{D}^{s}\spk{D'}^{\sigma}\colon H_p^{s_1,\sigma_1}(\R^n)\to H_p^{s_1-s,\sigma_1-\sigma}(\R^n)\]
%   is an isometric isomorphism.

%   b) For all $s,\sigma\in\R$, we have $H_p^{s,0}(\R^n) = H_p^s(\R^n)$ and $H_p^{ 0,\sigma }(\R^n)=L^p(\rz,H_p^\sigma(\R^{n-1}))$. For $s_1\le s_2$ and $\sigma_1 \le \sigma_2$, we have
% $H_p^{s_2,\sigma_2}(\R^n)\subset H_p^{s_1,\sigma_1}(\R^n)$.

%    c) As $\mathscr S(\R^n)$ is dense in $H^s_p(\R^n)$ for all $s\in\R$ and $\spk{D'}^\sigma$ leaves
%    $\mathscr S(\R^n)$ invariant, we immediately see that $\mathscr S(\R^n)$ is dense in $H_p^{s,\sigma}
%    (\R^n)$ for all $s,\sigma\in\R$.

% d) For all $\alpha\in\N_0^n$ and $s,\sigma\in\R$, the derivative
%   $\partial^\alpha\colon H_p^{s,\sigma}(\R^n)\to H_p^{s-|\alpha_n|, \sigma-|\alpha'|}(\R^n)$ is continuous.
% \end{remark}
% }

\begin{proposition}\label{1.4}
Let $s,\sigma\in\R$.
\begin{enumerate}[a)]
\item $H_p^{s,0}(\R^n) = H_p^s(\R^n)$ and $H_p^{0,\sigma}(\R^n)=L^p(\rz,H_p^\sigma(\R^{n-1}))$.
\item $H_p^{t,\tau}(\R^n)\subset H_p^{s,\sigma}(\R^n)$ whenever $s\le t$ and $\sigma\le\tau$.
\item For $\sigma\ge0$ we have
 \[H_p^{s+\sigma}(\R^n)\subset H_p^{s,\sigma}(\R^n)\subset H_p^s(\R^n)
 \subset H_p^{s,-\sigma}(\R^n)\subset H_p^{s-\sigma}(\R^n).
 \]
\item If $q$ is the dual coefficient to $p$, i.e. $\frac{1}{p}+\frac{1}{q}=1$, then the standard bilinear pairing $L^p(\rz^n)\times L^q(\rz^n)\to\C$ induces an identification of the dual space of $H^{s,\sigma}_{p}(\R^{n})$ with $H^{-s,-\sigma}_{q}(\R^{n})$.
\item In case of $s\ge 0$ we have
 \[H_p^{s,\sigma}(\R^n)=L^p(\R,H_p^{s+\sigma}(\R^{n-1}))\cap H_p^s (\R, H_p^\sigma(\R^{n-1})).\]
\item For $\alpha\in\N_0^n$, the derivative
  $\partial^\alpha\colon H_p^{s,\sigma}(\R^n)\to H_p^{s-|\alpha_n|, \sigma-|\alpha'|}(\R^n)$ is continuous.
\item If $s\in\N_{0}$ is an integer, $u\in\mathscr S'(\R^n)$ belongs to $H_p^{s,\sigma}(\R^n)$ if and only if $\partial_{n}^j u \in H_p^{0,s+\sigma-j}(\R^{n})$ for all $j=0,\dots,s$. Moreover, $\|u\|:=\sum_{j=0}^s \|\partial_{n}^j u\|_{H_p^{0,s+\sigma-j}(\R^{n})}$ defines an equivalent norm on $H_p^{s,\sigma}(\R^n)$.
\end{enumerate}
\end{proposition}

\begin{proof}
a) is clear. b) is true, since $\spk{D}^{s-t}\spk{D'}^{\sigma-\tau}$
is a bounded operator in $L^p(\R^n)$ due to Mikhlin's theorem.

c) By Mikhlin's theorem, both $\spk{D}^{-\sigma}\spk{D'}^{\sigma}$ and
$\spk{D'}^{-\sigma}$ are bounded operators in $L^p(\R^n)$ for $\sigma\ge 0$. This yields the first and the second inclusion, respectively. The other two are verified analogously.

d) As $\spk{D^{\prime}}^{\sigma}\colon H^{s,\sigma}_{p}(\R^{n})\to H^{s}_{p}(\R^{n})$ is an isometric isomorphism,
the adjoint operator $(\spk{D^{\prime}}^{\sigma})'\colon (H^{s}_{p}(\R^{n}))' \to (H^{s,\sigma}_{p}(\R^{n}))'$ is an isometric isomorphism in the dual spaces.
In the bilinear pairing,
this adjoint operator is again $\spk{D^{\prime}}^{\sigma}$, and the dual space of  $H^s_p(\R^n) $ is
given by  $H^{-s}_q(\rz^n)$.
Hence, the dual space of $H^{s,\sigma}_{p}(\R^{n})$ is identified with
$\spk{D^{\prime}}^{\sigma}(H^{-s}_q(\rz^n))=H^{-s,-\sigma}_{q}(\R^{n})$.

e) The equality is known to be true in case $\sigma=0$. Applying $\spk{D^{\prime}}^{-\sigma}$ to both sides of this equality yields the claim.

f) holds because $\partial^\alpha\spk{D'}^{-|\alpha'|}\spk{D}^{-\alpha_n}$ is
bounded in $L^p(\R^n)$ due to Mikhlin's theorem.

g) Again this is known in case $\sigma=0$. Then, the general case holds true because $\spk{D^{\prime}}^{\sigma}$ commutes with $\partial_n$.
\end{proof}

\subsection{Interpolation spaces}

Let us briefly recall the complex interpolation method, following \cite{Hytoenen-vanNeerven-Veraar-Weis16}, Section~C.2 (see also
\cite{Triebel78}, Section~1.9). Let $X_0$ and $X_1$ be an interpolation couple of
complex Banach spaces and $S:=\{z\in\C: 0<\Re z<1\}$. Denote by $\mathcal F(X_0,X_1)$ the space of all continuous functions $f\colon \overline S\to X_0+X_1$ such that $f|_S$ is holomorphic as an $(X_0+X_1)$-valued
function on $S$ and, for $j\in\{0,1\}$,
the function $b \mapsto f(j+ib):\R\to X_j$ is bounded and continuous.
$\mathcal F(X_0,X_1)$ becomes a Banach space with the norm
 \[ \|f\|_{\mathcal F(X_0,X_1)} := \max_{j=0,1}\,
    \sup_{b \in\R}   \|f(j+ib)\|_{X_j}.\]
For $\theta\in (0,1)$, the complex interpolation space
$[X_0,X_1]_\theta$ is defined as the space of all $x\in X_0+X_1$ for which $x=f(\theta)$ for some
$f\in\mathcal F(X_0,X_1)$, endowed with the norm
\[ \|x\|_{[X_0,X_1]_\theta} := \inf\{ \|f\|_{\mathcal F(X_0,X_1)} \,:\; f(\theta)=x\}.\]
With this norm, the complex interpolation space becomes a Banach space satisfying $X_0\cap X_1\subset
[X_0,X_1]_\theta\subset X_0+X_1$. In the definition of the interpolation space and the norm,
the space $\mathcal F(X_0,X_1)$ can be replaced by the subspace $\mathcal F_0(X_0,X_1)$ which consists of all
$f\in \mathcal F(X_0,X_1)$ such that $b\mapsto\|f(j+ib)\|_{X_j}$ vanishes for $|b|\to\infty$ and $j=0,1$.
%(see \cite{Hytoenen-vanNeerven-Veraar-Weis16}, Proposition~C.2.4).
We will also consider the space $\mathcal F_{0}(X_0,X_1; X_0\cap X_1)$ consisting of all $f\in \mathcal F_0(X_0,X_1)$ for which $f(z)\in X_0\cap X_1$ for all $z\in\overline S$ and where $f$ is continuous on $\overline S$ and holomorphic in $S$ as a function with values in $X_0\cap X_1$ $($note that the definition of this space differs from the one in \cite{Hytoenen-vanNeerven-Veraar-Weis16}).
By \cite{Triebel78}, Theorem~1.9.1, $\mathcal F_{0}(X_0,X_1; X_0\cap X_1)$ is dense in $\mathcal F_{0}(X_0,X_1)$.

\begin{lemma}\label{1.5}
Let $s_0,\sigma_0,s_1,\sigma_1\in\R$ and $\theta\in (0,1)$. Then
\begin{equation}\label{1-4}
 [H^{s_0,\sigma_0}_p(\R^n),H^{s_1,\sigma_1}_p(\R^n)]_\theta = H^{s_\theta,\sigma_\theta}_p(\R^n)
\end{equation}
with $s_\theta=(1-\theta)s_0+\theta s_1$ and $\sigma_\theta=(1-\theta)\sigma_0+\theta \sigma_1$.
\end{lemma}

\begin{proof}
For simplicity, we write $H_p^{s_1,\sigma_1} := H_p^{s_1,\sigma_1}(\R^n)$ etc. in the proof.
Due to $[X_0,X_1]_\theta = [X_1,X_0]_{1-\theta}$, we may assume $s_0\le s_1$.
Applying the operator $\spk{D}^{s_1}\spk{D'}^{\sigma_0}$ and setting $s:=s_0-s_1\le0 $ and $\sigma:= \sigma_1-\sigma_0\in\R$,
by the retraction argument from \cite{Amann95}, Proposition~2.3.2,
it remains to show that
\begin{equation}\label{1-1}
 [H_p^{s,0} , H_p^{0,\sigma} ]_\theta = H_p^{(1-\theta)s, \theta \sigma} .
\end{equation}
Note that $H_p^{s,0}=H_p^s$ and $H_p^{0,\sigma} = L^p(\R, H_p^\sigma(\R^{n-1}))$
by Proposition~\ref{1.4}~a).

We first show ``$\subset$'' in \eqref{1-1}. For this, let $u\in [H_p^s, H_p^{0,\sigma}]_\theta$,
and choose $g\in \mathcal F_0(H_p^s , H_p^{0,\sigma} )$
with $g(\theta)=u$.
By density,
%As $\mathcal F_0(H_p^s, H_p^{0,\sigma}; H_p^s\cap H_p^{0,\sigma})$
%is dense in $\mathcal F_0(H_p^s , H_p^{0,\sigma} )$,
there exists a sequence
$(g_k)_{k\in\N}\subset  \mathcal F_0(H_p^s, H_p^{0,\sigma}; H_p^s\cap H_p^{0,\sigma})$ such that
$g_k\to g$ in $\mathcal F (H_p^s, H_p^{0,\sigma})$.
%Due to the definition of $ [H^{s}_{p}, H^{0,\sigma}_{p}]_{\theta}$
It follows that
$g_k(\theta)\to g(\theta)$ in $[H^{s}_{p}, H^{0,\sigma}_{p}]_{\theta}$.
For $k\in\N$ let us define
\[ f_k(z) := e^{z^2-\theta^2} \spk{D'}^{\sigma z} g_k(z)\quad (z\in \overline S).\]
First we show that $f_k(z)\in H^s_p$ with continuous and holomorphic dependence on $z\in\overline S$ and $z\in S$, respectively. This is equivalent to showing that
  $$h_k(z):=\spk{D}^s\spk{D'}^{\sigma z} g_k(z) \quad (z \in \overline S)$$ defines a function $h_k \colon \overline{S}\to L^p$ which
depends on $z$ as requested. In case of $\sigma\le 0$ note that $g_k\colon\overline S\to H_p^s$, hence $\spk{D}^sg_k\colon\overline S\to L^p$, with the requested dependence on~$z$. Then the claim for $h_k$ follows from Lemma~5.6.8 in \cite{Hytoenen-vanNeerven-Veraar-Weis16}.
By (5.53) of \cite{Hytoenen-vanNeerven-Veraar-Weis16} we also find the estimate
 $$\|h_k(z)\|_{L^p}\le C (1+|\sigma\Im z|)\|g_k(z)\|_{H^s_p}\quad (z\in \overline S).$$
If $\sigma>0$, note that $\spk{D}^s\in L(L^p)$ since $s\le 0$. Then write
\begin{align*}
    \spk{D'}^{\sigma z} g_k(z)=\spk{D'}^{\sigma(z-1)} \spk{D'}^{\sigma}g_k(z).
\end{align*}
Since
$g_k\colon \overline S\to H^{0,\sigma}_p$, hence $\spk{D'}^{\sigma}g_k\colon \overline S\to L^p$,
with the requested dependence on $z$, the claim again follows by Lemma~5.6.8 of \cite{Hytoenen-vanNeerven-Veraar-Weis16}. Also
 $$\|h_k(z)\|_{L^p}\le C (1+|\sigma\Im z|)\|g_k(z)\|_{H^{0,\sigma}_p}
   \quad (z\in \overline S).$$
This yields
 $$\|f_k(z)\|_{H_p^{s}} \le \widetilde C \|g_k(z)\|_{H_p^s\cap H_{p}^{0,\sigma} }\quad (z\in \overline S)$$
with $\widetilde C := C\sup_{b\in\R} (1+|\sigma b |) e^{1-\theta^2-b ^2}$.

Arguing similarly, one finds $f_k(ib)\in H^s_p$ and
$f_k(1+ib)\in L^p$ with continuous dependence on $b\in\R$ and
\begin{align*}
  \sup_{b\in\R} \|f_k(ib)\|_{H_p^s }
  & \le C\sup_{b\in\R} \|g_k(ib)\|_{H_p^s },\\
  \sup_{b\in\R} \|f_k(1+ib)\|_{L^p } &
  \le C\sup_{b\in\R} \|g_k(1+ib)\|_{H_p^{0,\sigma} }.
\end{align*}
Summing up, we have shown that $f_k\in \mathcal F(H_p^s , L^p)$ with
\[ \|f_k\|_{\mathcal F(H_p^s , L^p)}\le C \|g_k\|_{\mathcal F (H_p^s, H_p^{0,\sigma})}.\]

% \rot{
% We first remark that $f_k(z)\in H_p^s$ for all $z\in\overline S$. In fact, if $z=a\in [0,1]$,
% this follows for $\sigma\le 0$ from $g_k(z)\in H_p^s$ and the fact that $\spk{D'}^{\sigma a}$
% is a bounded operator in $H_p^s$, while for $\sigma>0$ we use $g_k(z)\in H_p^{0,\sigma}$ which
% gives $f_k(z)\in H_p^{0,\sigma(1-a)}\subset L^p \subset H_p^s$ (note here that $s\le 0$).
% For general   $z=a +ib$ with $a \in [0,1]$ and
% $b \in\R$, the operator $\spk{D'}^{ i\sigma b }$
% is   bounded   in $L^p$ and therefore in $H_p^s$
% for all $b \in\R$ with operator norm not greater than
% $C (1+|\sigma| |b |)$ by \cite{Hytoenen-vanNeerven-Veraar-Weis16}, Section~5.6.b.
% This yields $\|f_k(a +ib)\|_{H_p^{s}} \le \widetilde C \|g_k(a+ib )\|_{H_p^s\cap H_{p}^{0,\sigma} }$
% with $\widetilde C := C\sup_{b\in\R} (1+|\sigma|\, |b |) e^{1-\theta^2-b ^2}$.

% By \cite{Hytoenen-vanNeerven-Veraar-Weis16}, Lemma~5.6.8, the function $f_k$ is holomorphic in $S$, and we have
% \begin{align*}
%   \sup_{b\in\R} \|f_k(ib)\|_{H_p^s } & \le C\sup_{b\in\R} \|g_k(ib)\|_{H_p^s },\\
%   \sup_{b\in\R} \|f_k(1+ib)\|_{L^p } &
%   \le C\sup_{b\in\R} \|g_k(1+ib)\|_{H_p^{0,\sigma} }
% \end{align*}
% for all $k\in\N$.
% Therefore, $f_k\in \mathcal F(H_p^s , L^p)$, and
% \[ \|f_k\|_{\mathcal F(H_p^s , L^p)}\le C \|g_k\|_{\mathcal F (H_p^s, H_p^{0,\sigma})}.\]
% }

As $(g_k)_{k\in\N}$ is convergent and therefore a Cauchy sequence, the above estimate,
applied to $f_k-f_\ell$, yields that also $(f_k)_{k\in\N}\subset \mathcal F(H_p^s , L^p)$
is a Cauchy sequence. Again by the definition of the interpolation space $[H_p^s,L^p]_\theta $, we obtain that
$(f_k(\theta))_{k\in\N}$ is a Cauchy sequence in $[H_p^s,L^p]_\theta = H_p^{(1-\theta)s}$
(for the last equality, see  \cite{Hytoenen-vanNeerven-Veraar-Weis16}, Theorem~5.6.9).
By completeness, there exists $v\in H_p^{(1-\theta)s}$ with $f_k(\theta)\to v$.
On the other hand, $f_k(\theta) = \spk{D'}^{ \theta\sigma} g_k(\theta)$ and
$g_k(\theta)\to g(\theta)=u$ in $[H_p^s,H_p^{0,\sigma}]_\theta$, hence
$f_k(\theta)\to \spk{D'}^{\theta\sigma}u$ in
$\mathscr S'(\R^n)$.
% \forget{
% On the other hand, we know  $f_k(\theta) = \spk{D'}^{ \theta\sigma} g_k(\theta)$ and
% $g_k(\theta)\to g(\theta)=u$ in $[H_p^s,H_p^{0,\sigma}]_\theta$. By the embeddings $ H^{(1-\theta)s}_{p}\subset H^{s,-2|\sigma|}_{p} $ and $ [H^{s}_{p}, H^{0,\sigma}_{p}]_{\theta} \subset H^{s,-2|\sigma|+\theta\sigma} $, we obtain
% \begin{align*}
% 	&\|v-\spk{D'}^{\theta\sigma} u\|_{H^{s,-2|\sigma|}_{p}} \\
% 	\leq \,&C\left( \|v-f_{k}(\theta)\|_{H^{s,-2|\sigma|}_{p}} + \|\spk{D'}^{\theta \sigma}\|_{L\left(H^{s,-2|\sigma|+\theta\sigma}_{p}, H^{s,-2|\sigma|}_{p}\right)}\|  g_{k}(\theta)-u\|_{H^{s,-2|\sigma|+\theta\sigma}_{p}}\right)\\
% 	\to&\, 0
% \end{align*}
% for $ k\to\infty $.
% }
%\[ v=\lim_{k\to\infty} f_k(\theta) = \spk{D'}^{\theta\sigma} \lim_{k\to\infty} g_k(\theta)
%= \spk{D'}^{\theta\sigma} u.\]
Therefore, $u = \spk{D'}^{-\theta\sigma} v \in H_p^{(1-\theta)s,\theta\sigma}$ with norm
\begin{align*}
  \|u\|_{H_p^{(1-\theta)s,\theta\sigma}} & \le C \|v\|_{[H_p^s,L^p]_\theta}
  \le C\lim_{k\to\infty} \|f_k\|_{\mathcal F(H_p^s , L^p)} \le C \lim_{k\to \infty}
  \|g_k\|_{\mathcal F(H_p^s , H_p^{0,\sigma})} \\
  & = C \|g\|_{\mathcal F(H_p^s , H_p^{0,\sigma})} .
\end{align*}
As $g\in \mathcal F_0(H_p^s , H_p^{0,\sigma} )$ was arbitrary with $g(\theta)=u$, we obtain
\begin{align*}
    \|u\|_{H_p^{(1-\theta)s,\theta\sigma}} \le C \|u\|_{[H_p^{s,0} , H_p^{0,\sigma} ]_\theta}
\end{align*}
which finishes the proof of ``$\subset$''.

The proof of the embedding ``$\supset$'' follows in exactly the same way. Let $u
\in H_p^{(1-\theta)s,\theta\sigma}$ and $v:= \spk{D'}^{\theta \sigma} u\in H_p^{(1-\theta)s}$,
and let  $f\in \mathcal F_0(H^s,L^p)$ with $f(\theta)=v$. We approximate $f$ by
a sequence $(f_k)_{k\in\N}\subset \mathcal F_0(H_p^s,L^p;L^p)$ and set
$g_k(z) := e^{z^2-\theta^2} \spk{D'}^{-\sigma z} f_k(z)$ for $z\in \overline S$ and
$k\in\N$. %\rot{Stimmt der Exponent in der Definition von $g_k$?}
If $\sigma\ge 0$, we see that $g_k(z)\in L^p \subset H_p^s$, and for
$\sigma<0$ we have $g_k(z)\in H_p^{0,\sigma}$, so we have $g_k(z)\in H_p^s+H_p^{0,\sigma}$
 for all $z\in\overline S$. Therefore, we can argue as above
to show the embedding ``$\supset$'' in \eqref{1-1}.
\end{proof}
In addition to the spaces above, we will also consider the standard Besov spaces
$B_{pp}^s(\R^n)$ for $p\in (1,\infty)$ and $s\in\R$.
For $X\in\{ \mathscr S, \mathscr{S}', H_p^s, B_{pp}^s, H_p^{s,\sigma}\}$
and a domain $\Omega\subset\R^n$, we define
\begin{align*}
    X(\Omega) := \{ u|_\Omega: u\in X(\R^n)\}
\end{align*}
(where restriction is understood in the distributional sense) with the canonical norm
\[ \|v\|_{X(\Omega)} := \inf\{ \|u\|_{X(\R^n)} : u\in X(\R^n),\, u|_\Omega = v\} \]
(see, e.g., \cite{Triebel06}, Definition~4.1).
{Moreover, for a closed subset $A \subset \R^n$ we define
\begin{align*}
    \dot X(A) :=
\{ u\in X(\R^n): \supp u\subset A\}.
\end{align*}
% \rot{
% Das einfachste waere so: $\Omega$ offen und $X(\Omega)$ wie gehabt. Fuer $A\subset\R^n$ abgeschlossen
% definiere
%  $$\dot X(A)=\{u\in X(\R^n): \supp u\subseteq A\}.$$
% (abgeschlossener Unterraum von $X(\R^n)$) und dann ist
% $X(\Omega)=X(\R^n)/\dot X(\R^n\setminus\Omega)$.

% \grun{Ich finde diese Variante gut. Dann würde man oben "smooth" streichen und $ \dot X(A)$ für abgeschlossene Mengen definieren. Wir sollten das hier kurz lassen.Robert.}

% Will man $\dot X$ nur fuer Abschluesse offener Mengen definieren, dann bekommt man als Bedingung, dass
% $\R^n\setminus\Omega$ der Abschluss seines Inneren sein muss, d.h. der Abschluss von
% $\R^n\setminus\overline\Omega$. Genau genommen muesste man dann aber auch
%  $$X(\Omega)=X(\R^n)/\dot X(\overline{\R^n\setminus\overline\Omega})$$
% schreiben. Sieht nicht gerade schoen aus.

% Dass $\R^n\setminus\Omega$ der Abschluss seines Inneren ist, ist aequivalent dazu, dass jeder Punkt
% $x\in\partial(\R^n\setminus\Omega)=\partial\Omega$ Grenzwert einer Folge
% $(x_n)\subset \R^n\setminus\overline\Omega$ ist. Das geht z.B., falls es offene Umgebungen
% $U,V\subset\rz^n$ von $x$ bzw. $0$ und einen Homeomorphismus $\chi:U\to V$ gibt, derart,
% dass $\chi(U\cap\partial\Omega)=V\cap\rz^{n-1}$. Das heisst, $\partial\Omega$ ist eine stetige Hyperflaeche
% im $\rz^n$ bzw. $\Omega$ hat einen stetigen Rand.
% }
\noindent Then by definition we see that $X(\Omega)$ can be identified with the quotient space
$X(\R^n)/\dot X(\R^n\setminus \Omega).$} Note that we consider $H^{s,\sigma}_p(\Omega)$  only for $\Omega=\R^n_+$.
% \forget{
% Setting $\dot X(\overline{\R^n_\pm}) :=
% \{ u\in X(\R^n): \supp u\subset \overline{\R^n_\pm}\}$,
% }

%\blau{Unten heisst das Gebiet $\Omega$. Vielleicht vereinheitlichen. (Joerg)}

\begin{remark}
  \label{1.6}
a) Let $\Omega=\R^n_+$. Then the restriction $r_{\R^n_+}\colon X(\R^n)\to X(\R^n_+),\,
u\mapsto u|_{\R^n_+}$ is a retraction. This follows from the fact that there exists
a restriction-extension pair $(R,E)$ for $(\mathscr S'(\R^n), \mathscr S'(\R^n_+))$
in the sense of \cite{Amann19}, Theorem~VI.1.2.3, which yields the restriction-extension pair $(r_{\R^n_+}, e_{\R^n_+})$ on $(X(\R^n), X(\R^n_+))$ by \cite{Amann19}, Lemma~VII.2.8.1. In particular, the extension operator $e_{\R^n_+}$ is universal for all considered spaces. Later, we will also consider $e_\Omega^0$, the canonical extension from $\Omega$ to $\R^n$ by zero.

b) Similarly, if $\Omega\subset\R^n$ has smooth boundary, the
map $u\mapsto u|_\Omega$ is a retraction from $H_p^s(\R^n)$ to $H_p^s(\Omega)$ and from
$B_{pp}^s(\R^n)$ to $B_{pp}^s(\Omega)$, and for all $N\in\N$ there exists a common co-retraction (i.e., a continuous right-inverse) for all $|s|<N$ (see \cite{Triebel10}, Theorem~3.3.4).
%In this paper, we will consider the spaces $H_p^{s,\sigma}(\Omega)$ only for $\Omega=\R^n_+$ (and $\Omega=\R^n$) as the specific role of the $x_n$-variable in the definition of $H^{s,\sigma}$ is motivated by the normal direction.

c) Due to a) and standard retraction-coretraction arguments (see \cite{Amann95}, Section I.2.3),
all statements of Proposition~\ref{1.4} and Lemma~\ref{1.5} remain valid
if we replace $\R^n$ by $\R^n_+$, with the exception of Proposition~\ref{1.4}~d) which has to be modified
in the following form: For all $s,\sigma\in\R$ the dual space of $H_p^{s,\sigma}(\R^n_+)$ with respect
to the standard pairing is given by $(H_p^{s,\sigma}(\R^n_+))' = \dot H_q^{-s,-\sigma}(\overline{\R^n_+})$.
This follows from $(H_p^s(\R^n_+))' = \dot H_q^{-s}(\overline{\R^n_+})$ (see \cite{Amann19},
Theorem~VII.4.4.2) in the same way as in the proof of Proposition~\ref{1.4}~d), noting that
$\spk{D'}^\sigma u$ has support in $\overline{\R^n_+}$ if $u$ does.
\end{remark}

\subsection{Boundary traces}

To define the trace space of $H_p^{s,\sigma}(\R^n_+)$ on the boundary $\partial\R^n_+=\R^{n-1}$, we
first note that for the standard space $H^s_p(\R^n_+)$  the trace
\begin{equation}\label{1-5}
 \gamma_0\colon H^s_p(\R^n_+)\to B_{pp}^{s- 1/p}(\R^{n-1}), \; u\mapsto \gamma_0 u :=
u|_{\R^{n-1}}
\end{equation}
 exists and is continuous if and only if $s> \frac 1p$. In fact, it was
shown in \cite{Johnsen-Sickel08}, Theorem~2.4, that for $s\le \frac1p$ the map $u\mapsto \gamma_0 u$
is not even continuous from $H^s_p(\R^n_+)$ to $\mathscr D'(\R^{n-1})$. If $s>\frac 1p$, then
\eqref{1-5} is a retraction, and $\gamma_0$ is the unique extension of the classical boundary trace
$u\mapsto u|_{x_n=0}$  for smooth functions  $u\in\mathscr S(\R^n_+)$. We will also
consider the higher-order traces $\gamma_j\colon H^s_p(\R^n_+)\to B_{pp}^{s-j-1/p}(\R^n),\;
u\mapsto \gamma_0\partial_n^{j}u$ for $j\in\N_0$ and $s>j+\frac 1p$.

\begin{definition}
  \label{1.7}
  Let $j\in\N_0$, $s\in (j+\frac1p,\infty)$, and $\sigma\in\R$. Then we define the $j$-th order trace
  $\wt\gamma_j$ on $H^{s,\sigma}_p(\R^n_+)$ as
  \begin{equation}\label{1-6}
   \wt\gamma_j\colon H^{s,\sigma}_p(\R^n_+)\to B_{pp}^{s+\sigma-j-1/p}(\R^{n-1}),\quad
   u\mapsto \langle D'\rangle^{-\sigma} \gamma_j \langle D'\rangle^{\sigma} u.
   \end{equation}
\end{definition}

\begin{remark}
  \label{1.8}
  Note that $\widetilde\gamma_j$ is well-defined as $\langle D'\rangle^{\sigma} u\in H_p^{s}(\R^n_+)$
  and  the unique continuous extension of the classical trace which is defined on the dense subspace
$\mathscr S(\R^n_+)$. The fact that $\gamma_j$ is a retraction on the classical
  space $H_p^s(\R^n_+)$ immediately implies that \eqref{1-6} is a retraction, too. In fact,
  if $e_j$ is a co-retraction to $\gamma_j$, then $\widetilde e_j := \langle D'\rangle^{-\sigma}
  e_j\langle D'\rangle^\sigma$ is a co-retraction to $\widetilde \gamma_j$. As $\widetilde \gamma_j$
  is (for $\sigma\le 0$) the unique continuous extension of $\gamma_j$ to $H_p^{s,\sigma}(\R^n_+)$,
  we will write $\gamma_j$ instead of $\widetilde \gamma_j$ again.
\end{remark}

\begin{remark}[Roitberg spaces]\label{1.9}
There is a theory of generalized boundary value problems in spaces of negative regularity due to
Roitberg \cite{Roitberg96}. In this theory, for $s\in\R$ and $\ell\in\N_0$, the space
$\widetilde H_p^{s,(\ell)}(\R^n_+)$ is defined as the set of all tuples $(u,g_0,\dots,g_{\ell-1})$ such that
there exists a sequence $(u_k)_{k\in\N}\subset \mathscr S(\R^n_+)$ satisfying
$(u_k,\gamma_0 u_k,\dots, \gamma_{\ell-1}u_k)\to (u,g_0,\dots,g_{\ell-1})$, where the convergence takes place
in the space
\begin{align*}
  H_p^s(\R^n_+)\times \prod_{j=0}^{\ell-1} B_{pp}^{s-j-1/p}(\R^{n-1}) & \text{ if } s\ge 0,\\
  \dot H_p^s(\overline{\R^n_+})\times \prod_{j=0}^{\ell-1} B_{pp}^{s-j-1/p}(\R^{n-1}) & \text{ if } s < 0.
\end{align*}
For $s>\ell-1+1/p$, the space $\widetilde H_p^{s,(\ell)}(\R^n_+)$ can be identified with the standard
Sobolev space $H_p^s(\R^n_+)$, and we have $g_j=\gamma_j u$ for $j=0,\ldots, \ell-1$ in this case, see~\cite{Roitberg96},
Section~2.1.

Let $\ell\in\N_0$, $s>\ell-1+1/p$, and $\sigma\le 0$, and let $u\in H_p^{s,\sigma}(\R^n_+)$. By
density, there exists a sequence $(u_k)_{k\in\N}\subset\mathscr S(
\R^n_+)$ with $\|u_k-u\|_{H_p^{s,\sigma}(\R^n_+)} \to 0\;(k\to\infty)$. The continuity
of \eqref{1-6} yields  $\gamma_j u_k\to \gamma_j u\in B_{pp}^{s+\sigma-j-1/p}(\R^{n-1})$ for $j=0,\ldots, \ell-1$. \pagebreak[3]

From this and Lemma~\ref{1.4}~a), we obtain the continuous embeddings
\begin{alignat*}{5}
  & H_p^{s,\sigma}(\R^n_+) & \subset \widetilde H_p^{s+\sigma,(\ell)}(\R^n_+)&\quad & \text{ if } s+\sigma\ge 0,\\
  &  H_p^{s,\sigma}(\R^n_+)\cap \dot H_p^{s+\sigma}({\overline{\R^n_+}})\, &
  \subset \widetilde H_p^{s+\sigma,(\ell)}(\R^n_+)&& \text{ if } s+\sigma < 0,
\end{alignat*}
where we identify $u\in H_p^{s,\sigma}(\R^n_+)$ with the tuple $(u,\gamma_0 u,\dots,\gamma_{\ell-1}u)$.
\end{remark}

\subsection{Parameter-dependent spaces}

We will also need parameter-dependent versions of the above spaces. For this, we follow the approach of Grubb--Kokholm~\cite{Grubb-Kokholm93}.

\begin{definition}
\label{1.10.0}
If $X_\lambda$ and $Y_\lambda$ are families of Banach spaces (parametrized by $\lambda$ from some index set), a family of linear operators $T(\lambda)\colon X_\lambda\to Y_\lambda$ is said to be continuous if  $T(\lambda)\in L(X_\lambda,Y_\lambda)$  for every fixed $\lambda$
and the operator norm  $\|T(\lambda)\|_{L(X_\lambda,Y_\lambda)}$ is uniformly bounded in $\lambda$. A continuous family
%$T(\lambda)\colon X_\lambda\to Y_\lambda$
is called an isomorphism if each $T(\lambda)$ is invertible and $T(\lambda)^{-1}\colon
Y_\lambda\to X_\lambda$ is a continuous family, too.
\end{definition}
In our context, the occurring families of spaces $X_\lambda$ will consist of a fixed vector space $X$ equipped with a norm depending on the parameter $\lambda$.
\begin{definition}
  \label{1.10}
  For $\lambda\in\C$ %, let $\langle\lambda\rangle_{m} = (1+ |\lambda|^{1/m})^{1/2}$, and
  let $\kappa_\lambda$ denote the homeomorphism of
  $\mathscr{S}'(\R^n)$ given on $\mathscr{S}(\R^n)$ by $(\kappa_\lambda
  u)(x)=u(\langle\lambda\rangle x)$. Then,
  we define the parameter-dependent norms by
    \begin{alignat*}{4}
     % \|u\|_{H_{p,\lambda}^s(\R^n_+)} & := \langle\lambda\rangle^{s-n/p} \|
      %\kappa_\lambda^{-1}  u\|_{H_p^s(\R^n_+)}&\quad& (s\in\R),\\
      \|u\|_{H_{p,\lambda}^{s,\sigma}(\R^n)} & := \langle\lambda\rangle^{s+\sigma-n/p} \|
      \kappa_\lambda^{-1}  u\|_{H_p^{s,\sigma}(\R^n)}&& (s,\sigma\in\R),\\
      \|u\|_{B_{pp,\lambda}^{s}(\R^{n-1})} & := \langle\lambda\rangle^{s-(n-1)/p} \|
      \kappa_\lambda^{-1}  u\|_{B_{pp}^s(\R^{n-1})}&\quad& (s \in\R).
    \end{alignat*}
    Additionally, we set $H_{p,\lambda}^s(\R^n):=H_{p,\lambda}^{s,0}(\R^n)$ for $s\in\R$. Analogously, we define $H_{p,\lambda}^{s,\sigma}(\R^n_+)$ and $H_{p,\lambda}^{s}(\R^n_+)$.% \rot{Vielleicht wäre es intuitiver die Räume auf $\R^n$ zu definieren und dann zu schreiben, dass sie auf $\R^n_+$ analog definiert sind? Außerdem könnte man evtl. die erste Formelzeile weglassen und stattdessen $H_{p,\lambda}^s(\R^n):=H_{p,\lambda}^{s,0}(\R^n)$ setzen.}
% \forget{
%    b) Let $X_\lambda$ and $Y_\lambda$ stand for one of the parameter-dependent spaces from a).
%     We say that a family of linear operators $T(\lambda)\colon X_\lambda\to Y_\lambda$ is
%     continuous if  $T(\lambda)\in L(X_\lambda,Y_\lambda)$  for every fixed $\lambda$
%     and its
%     operator norm  $\|T(\lambda)\|_{L(X_\lambda,Y_\lambda)}$ can be estimated by a constant
%     independent of $\lambda$. A continuous family $T(\lambda)\colon X_\lambda\to Y_\lambda$
%     is called an isomorphism if each $T(\lambda)$ is bijective and $\|T(\lambda)^{-1}\|_{L(
%     Y_\lambda,X_\lambda)}$ can be estimated by a constant independent of $\lambda$.
% }
\end{definition}

\begin{lemma} \label{1.11}\quad
\begin{enumerate}[a)]
\item
The statements from Proposition~\ref{1.4} a)--f), Lemma~\ref{1.5}, Remark~\ref{1.6} and Remark~\ref{1.8} remain valid in
the spaces $H_{p,\lambda}^{s,\sigma}$ for $\lambda\in\C$ with respect to the parameter-dependent norms.

\item For all $s,\sigma\in\R$ and $\lambda\in\C$, we have
  \[ \|u\|_{H_{p,\lambda}^{s,\sigma}(\R^n )} =
   \| \langle D',\lambda\rangle^{\sigma} u\|_{H_{p,\lambda}^s(\R^n )} = \| \langle D,\lambda
  \rangle^{s}\langle D',\lambda\rangle^{\sigma} u\|_{L^p(\R^n )}.\]
%   \rot{Im Halbraum muesste man eine andere spitze Klammer nehmen (siehe Grubb). Checken ob das Resultat nicht doch benutzt wird.}

\item (Interpolation inequality)
Let $s_0<s<s_1$ and $\sigma\in\R$. For every $\eps>0$ there exists a constant $C(\eps)>0$
such that, for every $\lambda\in\C$ and $u\in H^{ s_1,\sigma}_{p,\lambda}(\R^n)$,
\[\|u\|_{H^{s,\sigma}_{p,\lambda}(\R^n)}\le \eps \|u\|_{H^{s_1,\sigma}_{p,\lambda}(\R^n)}+
   C(\eps)\spk{\lambda}^{s-s_0}\|u\|_{H^{s_0,\sigma}_{p,\lambda}(\R^n)}.\]
   The analog statement holds for $\R^n_+$ instead of $\R^n$.
\end{enumerate}
\end{lemma}

\begin{proof}
a) We can apply the above results in the parameter-independent norms to the function $\kappa_\lambda^{-1}u$
and obtain constants independent of $\lambda$, noting also that $\kappa_\lambda$ commutes with taking
the trace on the boundary $\R^{n-1}$.

b) For $\sigma=0$, the statement follows from \cite{Grubb-Kokholm93}, Formula (1.9).
For general $\sigma$, we use the identity
\[ \spk{\lambda}^\sigma\kappa_\lambda\spk{D'}^\sigma \kappa_\lambda^{-1}
= \kappa_\lambda \big\langle \spk{\lambda} D',\lambda\big\rangle^\sigma\kappa_\lambda^{-1}
= \spk{D',\lambda}^\sigma, \]
which is obtained by straightforward calculation. This yields
\begin{align*}
 \|u\|_{H^{s,\sigma}_{p,\lambda} (\R^n)} & =
 \spk{\lambda}^{s+\sigma-n/p} \big\| \spk{D'}^{\sigma} \kappa_\lambda^{-1} u\big\|_{H_p^s(\R^n)} = \spk{\lambda}^{\sigma} \big\| \kappa_\lambda \spk{D'}^{\sigma} \kappa_\lambda^{-1} u\big\|_{H_{p,\lambda}^s(\R^n)}\\
& = \big\| \langle D',\lambda\rangle^{\sigma} u\big\|_{H_{p,\lambda}^s(\R^n)}
= \big\| \langle D,\lambda
  \rangle^{s} \langle D',\lambda\rangle^{\sigma} u\big\|_{L^p(\R^n)}.
\end{align*}

c) An application of  the standard interpolation inequality gives
\begin{align*}
\|u\|_{H^{ s,\sigma }_{p,\lambda} (\R^n)} &=
 \spk{\lambda}^{s+\sigma-n/p} \big\| \spk{D'}^{\sigma} \kappa_\lambda^{-1} u\big\|_{H_p^s(\R^n)}\\
&    \le \spk{\lambda}^{s+\sigma-n/p}\Big(
\eps \|\spk{D'}^\sigma \kappa_\lambda^{-1} u \|_{H^{s_1}_p(\R^n)}+
 C(\eps)\|\spk{D'}^\sigma \kappa_\lambda^{-1} u \|_{H^{s_0}_p(\R^n)}\Big) \\
& \le \eps \|u\|_{H^{s_1,\sigma}_{p,\lambda}(\R^n)}+
   C(\eps)\spk{\lambda}^{s-s_0}\|u\|_{H^{s_0,\sigma}_{p,\lambda}(\R^n)}.
\end{align*}

\vspace*{-1em}
\end{proof}

Let us remark that a statement analogous to Lemma~\ref{1.11}~b)  does not hold for the parameter-dependent Besov spaces (with $L^p(\R^n)$ being replaced by $B_{pp}^0(\R^{n-1})$).
Although $\langle D',\lambda\rangle^s\colon B_{pp,\lambda}^s(\R^{n-1})\to B_{pp,\lambda}^0(\R^{n-1})$
is an isomorphism, the norm in $B_{pp,\lambda}^0(\R^{n-1})$ still depends on $\lambda$, in
contrast to $\|\cdot\|_{H_{p,\lambda}^0(\R^n)} = \|\cdot\|_{L^p(\R^n)}$.
This was observed in \cite{Grubb-Kokholm93}, Section~1.1.

\subsection{Multiplication operators}

We finish this section with some considerations concerning multiplication operators. For a sufficiently smooth function $a\colon \R^n\to\C$, we define the multiplication operator $M_a$ by $M_a u \coloneqq au$ whenever the function $u$ belongs to some Sobolev space of positive order. For negative order spaces, we define the multiplication operator by duality with respect to the canonical pairing $L^p(\R^n)\times L^{q}(\R^n)$, where $\frac1p+\frac1q=1$.
In the following, $\BUC^r(\Omega)$ denotes the space of all functions which are $r$-times continuously differentiable
in $\Omega$ and for which all derivatives up to order $r$ are bounded and uniformly continuous.

%\orange{The following Lemma concerns mapping properties of the operator $M_a$ of multiplication by a function $a$, i.e., $M_au=au$ if $u$ is sufficiently regular.}
\begin{lemma}\label{1.12}
Let $s, \sigma \in \R$, and define $r'=r'(s,\sigma):= \max\{|s|,|\sigma|,|s+\sigma|\}$ and
\begin{align}\label{def:r}
  r=r(s,\sigma)%:=\lfloor r' \rfloor +1
  :=\lfloor r' %\max\{|s|,|\sigma|,|s+\sigma|\}
  \rfloor+1.
\end{align}
Let $a \in \BUC^r(\R^n)$ and let $M_a$ denote the operator of multiplication by $a$.
\begin{enumerate}[a)]
    \item There are constants $C=C(r)=C(s,\sigma)>0$ and $\gamma=\gamma(s,\sigma)> 0$ such that for all $\lambda \in \C$
  \begin{align}  \label{eq:int}
    \|M_a\|_{L(H_{p,\lambda}^{s,\sigma}(\R^n))}\le
    C(r)\|a\|^{1-\gamma}_{\BUC^{r}(\R^{n})}\|a\|_\infty^\gamma.
  \end{align}
  \item If we only have $a \in\BUC^{\lceil r' \rceil }(\R^n)$, $M_a$ is still a multiplier in $H^{s,\sigma}_{p,\lambda}(\R^n)$ and \eqref{eq:int} holds with $\gamma=0.$
      \item For every $\eps>0$ there exists a $\delta=\delta(\eps,s,\sigma)>0$ and a $\lambda_0=\lambda_0\big(\|a\|_{\BUC^{r}(\R^{n})}\big) > 0$ such that
    \[ \|M_a\|_{L(H_{p,\lambda}^{s,\sigma}(\R^n))} < \epsilon \]
  whenever $\|a\|_\infty<\delta$ and $\lambda\in\C$ with $|\lambda|\geq \lambda_0$. \item  The results in a$)$, b$)$ and c$)$ hold analogously
  for $\R^n_+$ instead of $\R^n$ with
  $a\in \BUC^{r}({\R^n_+})$.
  \item The results in a), b) and c) also hold if we replace $H_{p,\lambda}^{s,\sigma}(\R^n)$ by  $B_{pp,\lambda}^{s}(\R^{n-1})$, taking $\sigma=0$, i.e. $r'=|s|$, and $a \in \BUC^r(\R^{n-1})$ or $a \in \BUC^{\lceil r' \rceil}(\R^{n-1})$, respectively.
\end{enumerate}
\end{lemma}
\begin{proof}~\
a) Consider the hexagon which is the convex hull of the vertex set
    \[\mathcal{H}:=\{(r,0), (0,r), (-r,r),(-r,0),(0,-r),(r,-r)\}\]
%Given $s,\sigma\in\R$, the smallest possible choice for $r$ in order that
%$(s,\sigma)\in\mathcal{H}$ is $r=\lceil r'\rceil$ with $r':=\max\{|s|,|\sigma|,|s+\sigma|\}$. Taking $r$ as in \eqref{def:r} assures that $(s,\sigma)$ is an interior point of $\mathcal{H}$.
(see Figure~\ref{fig01}).
In a first step we are going to show that for all $P \in \mathcal{H}$ we can deduce the bound \[\|M_a\|_{L(H^{P}_{p}(\R^n))} \leq C(r)\|a\|_{\BUC^{r}(\R^n)}.\]
For the first two vertices this follows from the fact that $H_{p}^{r,0}(\R^n)=H_{p}^{r}(\R^n)$ and $H_{p}^{0,r}(\R^n)=L^p(\R,H_{p}^{r}(\R^{n-1}))$ due to Lemma~\ref{1.4}~a) as well as the product rule. Their counterparts $(0,-r)$ and $(-r,0)$ can be treated by a duality argument. For the space $H_p^{r,-r}(\R^n)$ we use Lemma~\ref{1.4}~g) to obtain
    \begin{align*}
\|au\|_{H_p^{r,-r}(\R^n)}&\leq C(r) \sum_{j=0}^{r} \|\partial_n^j(au)\|_{H_p^{0,-j}(\R^n)}\\
&\leq C(r)\sum_{j=0}^{r}\sum_{l=0}^j \|(\partial_n^{j-l} a) (\partial_n^l u)\|_{H_p^{0,-j}(\R^n)} \\
&\leq C(r) \sum_{j=0}^{r}\sum_{l=0}^j \|(\partial_n^{j-l} a) (\partial_n^l u)\|_{H_p^{0,-l}(\R^n)} .
\end{align*}
Here we used $l\leq j$ and hence $H_p^{0,-l}(\R^n)$ is continuously embedded in $H_p^{0,-j}(\R^n).$ Furthermore we have $\partial_n^{j-l}a \in \BUC^{l}(\R^n)$ as $j \leq r$. So we may apply the boundedness of $M_{\partial^{j-l}_n a}$ on $H_p^{0,-l}(\R^n)$ and find
%. As also $j\leq r$ we have
\begin{align*}
\|au\|_{H_p^{r,-r}(\R^n)}
%&\leq  C(r) \sum_{j=0}^{r}\sum_{l=0}^j \|(\partial_n^{j-l} a) (\partial_n^l u)\|_{H_p^{0,-l}(\R^n)}\\
&\leq C(r) \sum_{j=0}^{r}\sum_{l=0}^j C(l)\|\partial_n^{j-l}a\|_{\BUC^l(\R^n)}\| \partial_n^l u\|_{H_p^{0,-l}(\R^n)}  \\
&\leq C(r) \|a\|_{\BUC^{r}(\R^n)}\sum_{j=0}^{r}\sum_{l=0}^j  \| \partial_n^l u\|_{H_p^{0,-l}(\R^n)} \\
&= C(r) \|a\|_{\BUC^{r}(\R^n)}\sum_{l=0}^{r} (r-l+1) \| \partial_n^l u\|_{H_p^{0,-l}(\R^n)}\\
&\leq C(r) \|a\|_{\BUC^{r}(\R^n)}\sum_{l=0}^{r}\| \partial_n^l u\|_{H_p^{0,-l}(\R^n)}\\
&\leq C(r)\|a\|_{\BUC^{r}(\R^n)} \|u\|_{H_p^{r,-r}(\R^n)}.
\end{align*}
The last vertex follows by duality again.

In a second step we interpolate along the edges of the hexagon, which is precisely the domain $\{(t,\tau): \max\{|t|,|\tau|,|t+\tau|\}=r\}$.
For any point $P_\theta=(1-\theta) P_0 + \theta P_1$, where $0<\theta<1$ and
$P_0, P_1 \in \mathcal{H}$, we obtain by interpolation $H^{P_\theta}_p(\R^n)=[H^{P_0}_p(\R^n),H^{P_1}_p(\R^n)]_{\theta}$ and thus
\begin{equation}\label{eq:N}
    \|M_a\|_{L(H^{P_\theta}_p(\R^n))}\leq C \|M_a\|^{1-\theta}_{L(H^{P_0}_p(\R^n))}\cdot \|M_a\|^{\theta}_{L(H^{P_1}_p(\R^n))}\leq C(r) \|a\|_{\BUC^{r}(\R^n)}.
\end{equation}
Moreover, we observe that in $L^p(\R^n)=H^{0,0}_p(\R^n)$ we have $\|M_a\|_{L(L^p(\R^n))}=\|a\|_\infty.$
Finally, we interpolate along a straight line that starts in the origin, passes through $(s, \sigma)$ and hits the boundary of the hexagon in the point  $\Big(\frac{rs}{r'},\frac{r\sigma}{r'}\Big)$. More precisely, we use the interpolation
\[H^{s,\sigma}_p(\R^n)
=[L^p(\R^n),H^{rs/r',r\sigma/r'}_p(\R^n)]_{r'/r}\]
to find that
\[\|M_a\|_{L(H^{s,\sigma}_p(\R^n))}\leq C \|M_a\|^{1-r'/r}_{L(L^p(\R^n))}\cdot \|M_a\|^{r'/r}_{L(H^{rs/r',r\sigma/r'}_p(\R^n))}\leq  C(r)\|a\|^{1-\gamma}_{\BUC^{r}(\R^{n})}\|a\|_\infty^\gamma \] for $\gamma:=1-\frac{r'}{r}=1-\frac{r' }{\lfloor r' \rfloor+1}>0$.

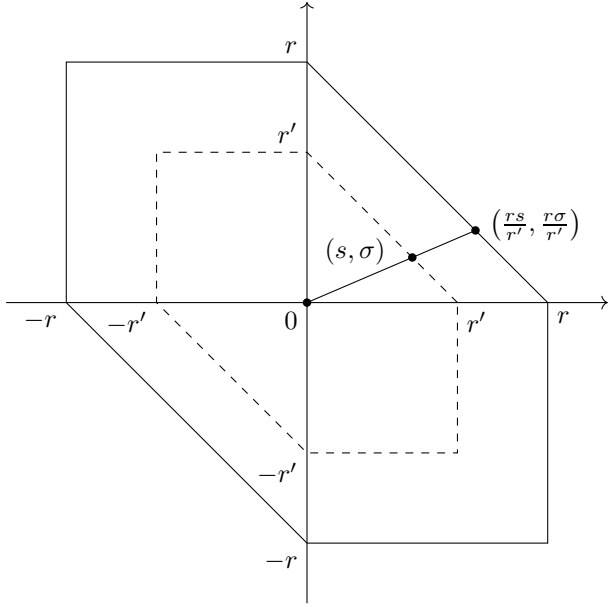
\begin{figure}[H]
	\begin{center}
		\begin{tikzpicture}[scale=0.8]
		%äußeres Sechseck
		\draw (4,0)node[below right] {$r$} -- (0,4)node[above left] {$r$} -- (-4,4) -- (-4,0)node[below left] {$ -r$} --(0,-4)node[below left] {$ -r$} -- (4,-4) -- (4,0) ;
		
		%inneres Sechseck
		\draw[dashed] (2.5,0)node[below right] {$r'$} -- (0,2.5)node[above left] {$r'$} -- (-2.5,2.5) -- (-2.5,0)node[below left] {$ -r'$} --(0,-2.5)node[below left] {$ -r'$} -- (2.5,-2.5) -- (2.5,0) ;
		
		%Interpolationslinie durch den Ursprung
		\draw (0,0) -- (2.8,1.2) ;
		\fill(0,0)circle(2pt);
		\fill(1.75,0.75)circle(2pt);
		\fill(2.8,1.2)circle(2pt);
		
		%Beschriftung der Punkte
		\node at  (0.8,0.85) {$ (s,\sigma) $ };
		\node at  (3.8,1.3) {$ \left(\frac{rs}{r'},\frac{r\sigma}{r'}\right) $};
		
% 		\draw (0,0) -- (2.5,-1)node[below left]{$ (s,\sigma) $} -- (4, -1.6)node[right]{$ \left(\frac{rs}{r'},\frac{r\sigma}{r'}\right) $};
% 		\fill(0,0)circle(2pt);
% 		\fill(2.5,-1)circle(2pt);
% 		\fill(4,-1.6)circle(2pt);

		% Achsen
		\node[below left]{$0$};
		\draw[->] (-5,0) -- (5,0);
		\draw[->] (0,-5) -- (0,5);
		\end{tikzpicture}
	\end{center}
	\caption{\label{fig01}In a first step, we see that the operator $ M_{a} $ is continuous on $ H^{P}_{p}(\R^{n}) $ for every vertex $ P\in \mathcal{H} $ of the outer hexagon and therefore by interpolation continuous on $ H^{P_{\theta}}_{p}(\R^{n}) $ for every $ P_{\theta} $ on its boundary. Finally, we interpolate between that boundary and the origin to get the continuity on $ H^{s,\sigma}_{p}(\R^{n}) $ for every $ (s,\sigma) $ on the boundary of the dashed hexagon. In the origin, we have $ \gamma=1 $, on the boundary of the outer hexagon, we have $\gamma=0$.
	}
\end{figure}

In order to carry the result over to the parameter-dependent norms, we observe the following for any bounded operator $T$ in $H^{s,\sigma}_{p}(\R^{n})$: By Definition~\ref{1.10} we have
\begin{align*}
 \|Tu\|_{H^{s,\sigma}_{p,\lambda}(\R^{n})}
 &=\spk{\lambda}^{s+\sigma-n/p}
   \|\kappa_{\lambda}^{-1}(Tu)\|_{H^{s,\sigma}_p(\R^n)}\\
 &=\spk{\lambda}^{s+\sigma-n/p}
   \|(\kappa_{\lambda}^{-1}T\kappa_{\lambda})(\kappa_{\lambda}^{-1}
   u)\|_{H^{s,\sigma}_p(\R^n)}.
\end{align*}
Dividing by $\|u\|_{H^{s,\sigma}_{p,\lambda}(\R^{n})}
=\spk{\lambda}^{s+\sigma-n/p}
\|\kappa_{\lambda}^{-1}u\|_{L(H^{s,\sigma}_p(\R^n))}$ and
passing to the supremum over all $0\not=u\in H^{s,\sigma}_{p,\lambda}(\R^{n})$ we conclude that
 \[\|T\|_{L(H^{s,\sigma}_{p,\lambda}(\R^{n}))}=
   \|\kappa_{\lambda}^{-1}T\kappa_{\lambda}\|_{L(H^{s,\sigma}_{p}(\R^{n}))}.\]
%   \qquad\forall\;\lambda\not=0.\]
Since we have
\begin{align*}
    \kappa^{-1}_\lambda M_a \kappa_\lambda u(x)=\kappa_\lambda^{-1}[a(x)u(\spk{\lambda}x)]
   =a(\spk{\lambda}^{-1}x)u(x)=(\kappa^{-1}_\lambda a)(x)u(x)
\end{align*}
it holds
   \begin{align*}
       \|M_a\|_{L(H^{s,\sigma}_{p,\lambda}(\R^{n}))}=
   \|M_{\kappa_{\lambda}^{-1}a}\|_{L(H^{s,\sigma}_{p}(\R^{n})).}
%   \qquad\forall\;\lambda\not=0.
   \end{align*}
Thus
\begin{align}\label{1-7}
 \|M_a\|_{L(H^{s,\sigma}_{p,\lambda}(\R^{n}))}
 &\leq C(r)  \|\kappa_{\lambda}^{-1} a\|_{\BUC^{r}(\R^n)}^{1-\gamma}
   \|\kappa_{\lambda}^{-1} a\|_{\infty}^{\gamma}
 \leq C(r) \|a\|_{\BUC^{r}(\R^n)}^{1-\gamma}
    \|a\|_{\infty}^{\gamma},
\end{align}
since $\|\partial^\alpha(\kappa_{\lambda}^{-1} a)\|_{\infty}
=\spk{\lambda}^{-|\alpha|}\|\partial^\alpha a\|_\infty\le
\|\partial^\alpha a\|_\infty$
for every $\alpha\in\N_0^n$ with $|\alpha|\le r$.

b) Obviously $\lceil r' \rceil \neq r$ only for $r' \in \N_0.$ So for $r' \not\in \N_0,$ the proof from a) remains unchanged. For $r' \in \N_0,$ we proceed analogously as in a), just replacing $r$ by $\lceil r' \rceil=r'$ but stop at \eqref {eq:N}. Carrying over this result to the parameter-dependent norms as before yields \eqref{eq:int} with $\gamma=0$.

c) Let $\epsilon>0$ and choose $\delta\in (0,1)$ with $\delta^\gamma < \frac{\epsilon}{2C(r)}$. Let $a\in \BUC^{r}(\R^n)$ with $\|a\|_\infty<\delta$.
As $\|\partial^\alpha (\kappa_\lambda^{-1} a)\|_\infty = \langle\lambda\rangle^{-|\alpha|}\|\partial^\alpha a\|_\infty$, there
is a $\lambda_0= \lambda_0(\|a\|_{\BUC^{r}(\R^n)})>0$ such that
\[ \sum_{|\alpha|=1}^{r} \|\partial^\alpha (\kappa_\lambda^{-1} a)\|_\infty \le  1\]
for all $\lambda\in\C$ with $|\lambda|\ge \lambda_0$. For all such $\lambda$ we obtain
\[ \|\kappa_\lambda^{-1} a\|_{\BUC^{r}(\R^n)} \le \|\kappa_\lambda^{-1} a\|_\infty + 1  = \|a\|_\infty + 1 < \delta +1 <2 \]
and therefore, using the analog of \eqref{1-7},
$\|M_a\|_{L(H^{s,\sigma}_{p,\lambda}(\R^{n}))}
\le C(r)2^{1-\gamma} \delta^\gamma <\eps$.

d) There exists a bounded extension operator $E_{\R^n_+}\colon \BUC^r({\R^n_+})\to \BUC^r(\R^n)$ for any $r \in \N_0$ (see, e.g., the construction in \cite{Adams-Fournier03}, Theorem~5.19). Then for~$\R^n_+$ part a) follows from
%$a\in \BUC^{r}(\overline{\R^n_+})$ from
\begin{align*}
    \|M_au\|_{H_{p,\lambda}^{s,\sigma}(\R^n_+)} & \le \|M_{E_{\R^n_+}a}(e_{\R^n_+}u)\| _{H_{p,\lambda}^{s,\sigma}(\R^n)}\le C \|E_{\R^n_+}a\|_{\BUC^{r}(\R^n)} \|e_{\R^n_+}u\|_{H_{p,\lambda}^{s,\sigma}(\R^n)} \\
    & \le C \|a\|_{\BUC^{r}({\R^n_+})}
    \|u\|_{H_{p,\lambda}^{s,\sigma}(\R^n_+)}.
\end{align*}

e) Taking $\sigma=0$, we use the result from a) and b) for the spaces $H^{s+\rho}_{p,\lambda}(\R^{n-1})$ and $H^{s-\rho}_{p,\lambda}(\R^{n-1})$ for a sufficiently small $\rho>0$ such that $|s \pm \rho|\leq \lfloor |s| \rfloor +1$ still holds. Then the result follows by real interpolation of the $\lambda$-dependent, but classical Sobolev spaces, which was established in \cite[(1.16)]{Grubb-Kokholm93}.
\end{proof}

\begin{remark}
  In the case of $r' \in \N_0$ in Lemma~\ref{1.12}~b) and $\sigma=\lambda=0$, we directly get back the classical results for the usual Sobolev spaces $H^s_{p}(\R^n)$.

  Furthermore, we would like to note that pointwise multipliers in Besov spaces with $\lambda=0$ are described, e.g., in \cite{Mazya-Shaposhnikova09} and
\cite{Yuan-Sickel-Yang10}. In particular, it is known that functions which are H\"older continuous with
H\"older index larger than $|s|$ are multipliers in $B_{pp}^s(\R^{n-1})$ (see \cite{Runst-Sickel96}, Theorem~4.7.1 (ii)). For our purposes, however, Lemma~\ref{1.12}~e) is sufficient.

\end{remark}

\section{Boundary value problems in the half-space}
\label{sec3}

We now deal with boundary value problems in domains and in the half-space. In the following, let
$\Omega\subset\R^n$ be a  domain with compact and sufficiently smooth boundary $\Gamma$, or
let $\Omega=\R^n_+$ with boundary $\Gamma=\R^{n-1}$. %\rot{Ist hier $\Omega = \R^n$, bzw. $\Gamma = \emptyset $ zulässig?}
We  consider the boundary value problem
\begin{equation}
  \label{2-1}
  \begin{aligned}
    (\lambda - A) u & = f \quad \text{ in } \Omega,\\
    B_j u & = g_j \quad (j=1,\dots,m)\; \text{ on }\Gamma,
  \end{aligned}
\end{equation}
where $A$ and $B_j$ are linear differential operators of order $2m$ and linear boundary operators
of order $m_j<2m$, respectively, of the form
\begin{align}
    A & = A(x,D) = \sum\nolimits_{|\alpha|\le 2m} a_\alpha(x) D^\alpha,\label{2-6}\\
    B_j & =  B_j(x,D) = \sum\nolimits_{|\beta|\le m_j} b_{j\beta}(x)\gamma_0 D^\beta.\label{2-7}
\end{align}
We also  write $B=(B_1,\ldots,B_m)$. Here, $a_\alpha\colon \overline \Omega\to\C$ and $b_{j\beta}\colon \Gamma\to\C$
are sufficiently smooth functions. More precisely, we will consider the following smoothness assumptions, depending
on $(s,\sigma)\in\R^2$.
\begin{itemize}
  \item[\textbf{(S1)}]
%   Let $r=r(s-2m,\sigma):=\lfloor \max\{|s-2m|,|\sigma|,|s+\sigma-2m|\} \rfloor +1$ and $a_\alpha\in \BUC^r(
%   \Omega)$
%   for all $|\alpha|\le 2m$.
%   \rot{Wahrscheinlich könnte man die Bedingung für $|\alpha| <2m$ etwas abschwächen, ähnlich in (S3).}
 Let $r'=r'(s-2m,\sigma):= \max\{|s-2m|,|\sigma|,|s+\sigma-2m|\}$ and $r:= \lfloor r'\rfloor +1$. We assume $a_\alpha\in \BUC^r(\Omega)$ for all $|\alpha| = 2m$ and $a_\alpha \in \BUC^{\lceil r' \rceil}(\Omega) $ for all $|\alpha|<2m$.
  \item[\textbf{(S2)}]
  %If $\Omega$ is unbounded, then $a_\alpha(\infty) := \lim_{x\in \Omega,\, |x|\to\infty} a_\alpha(x)$ exists for all $|\alpha|=2m$.
  %\rot{Hier muss man evtl. fordern, dass es für hinreichend großen Radius tatsächlich konstant wird, Problem siehe unten in Bemerkung~\ref{rem_ext}, ähnlich in (S4).}
%   \blau{If $\Omega$ is unbounded, then $x\mapsto a_\alpha \big(\frac{x}{|x|^2}\big)$ possesses an extension to a function in $\BUC^r(U)$ for all $|\alpha|=2m$ and an extension to a function in $\BUC^{\lceil r'\rceil}(U)$ in some neighborhood $U$ of $0$ for all $|\alpha|<2m$, respectively. In particular, $a_\alpha(\infty) := \lim_{x\in \Omega,\, |x|\to\infty} a_\alpha(x)$
%   exists for all $|\alpha|\leq2m$.}
   If $\Omega$ is unbounded, then $a_\alpha(\infty) := \lim_{x\in \Omega,\, |x|\to\infty} a_\alpha(x)$ exists for all $|\alpha|\leq2m$.
  In addition, all derivatives of the function
  \[ x\mapsto a_\alpha \bigg(\frac{x}{|x|^2}\bigg) \quad  (x\neq 0)\]
  % \begin{align*}
  %     x\mapsto \begin{cases}
  %     a_\alpha \big(\frac{x}{|x|^2}\big) &\textnormal{if } x\neq 0,\\
  %     a_\alpha(\infty)&\textnormal{if } x=0
  %     \end{cases}
  % \end{align*}
  up to order $r$ for $|\alpha|=2m$ (and up to order $\lceil r' \rceil$ for $|\alpha| < 2m$) possess a continuous extension to $x=0$.
    %In particular, this holds true, if $a_\alpha$ is constant for sufficiently large $|x|$.
 % \rot{Man muss hier evtl. bei der Wahl von $r_0$ aufpassen.}

  \item[\textbf{(S3)}]
  %For each $j\in\{1,\dots,m\}$, let $k_j:=\lfloor |s+\sigma-m_j-1/p| \rfloor +1$  and
  %$b_{j\beta}\in \BUC^{k_j}(\Gamma)$ for all $|\beta|\le m_j$.
  For each $j\in\{1,\dots,m\}$, let $k_j':= |s+\sigma-m_j-\frac1p|$  and $k_j:= \lfloor k_j'\rfloor +1$. We assume
  $b_{j\beta}\in \BUC^{k_j}(\Gamma)$ for all $|\beta|= m_j$ and $b_{j\beta}\in \BUC^{\lceil k_j'\rceil}(\Gamma)$ for all $|\beta_j|<m_j$.
  \item [\textbf{(S4)}] If $\Omega=\R^n_+$, then
  $b_{j\beta}(\infty) := \lim_{x\in\R^{n-1},\, |x|\to\infty} b_{j\beta}(x)$ exists for all $j\in\{1,\dots,m\}$ and
  $|\beta|\leq m_j$. In addition, all derivatives of the function
 \[
       x\mapsto
      b_{j\beta} \bigg(\frac{x}{|x|^2}\bigg) \quad (x \neq 0) \]
  up to order $k_j$ for $|\beta|=m_j$ (and up to order $\lceil k_j' \rceil$ for $|\beta|<m_j$) possess a continuous extension to $x=0$.
  \item [\textbf{(S5)}] The domain $\Omega$ is of class $C^{2m+\lceil r' \rceil}$.
  % with $r>2m+\max\{|s-2m|,|\sigma|,|s+\sigma-2m| \}$ and $r>m_j+|s+\sigma-m_j-1/p|$ for any $j=1,\dots, m$.
  % \rot{Wenn ich mich nicht verrechnet habe, ist die Bedingung mit $j$ automatisch erf\"ullt \blau{(Ja, stimme zu (Joerg).)}, und wir haben $r:= 2m+k$ mit $k$ aus (S1). Wir sollten dann direkt schreiben: of class $C^{2m+k}$. Robert}
\end{itemize}

In the following, let $\Lambda\subset \C$ be a closed sector in the complex plane with vertex at the origin.
 Then the family $\lambda-A(x,D)$ is called parameter-elliptic in $\Lambda$ if  the principal symbol $A_0(x,\xi) := \sum_{|\alpha|=2m} a_\alpha(x)  \xi^\alpha$ %\blau{(Wenn man $D$ verwendet (und nicht $\partial$) so muss man im Symbol $\xi^\alpha$ schreiben, nicht $(i\xi)^\alpha$ (Joerg))}
satisfies
\begin{equation}\label{2-5}
 |\lambda-A_0(x,\xi) | \ge C \big( |\lambda|+ |\xi|^{2m}\big)
\quad (x\in\overline \Omega,\, \lambda\in \Lambda,\, \xi\in\R^n,\, (\lambda,\xi)\not=0)
\end{equation}
for some constant $C>0$.
Similarly, we define the  principal symbols $B_{0,j}(x,\xi) := \sum_{|\beta|=m_j} b_{j\beta}(x)\xi^\beta$.
The boundary value problem is called parameter-elliptic in $\Lambda$
if $\lambda-A(x,D)$ is parameter-elliptic in $\Lambda$ and  if the following Shapiro--Lopatinskii condition holds:

Let $x_0\in\partial\Omega$ be an arbitrary point of the boundary, and rewrite the boundary value problem $(\lambda-A_0(x_0,D),$ $B_{0,1}(x_0,D),\dots, $ $B_{0,m}(x_0,D))$ in  the coordinate system associated with $x_0$, which is obtained from the original one by a rotation after which the positive $x_n$-axis has the direction of the interior normal to $\partial \Omega$ at~$x_0$. Then the trivial solution $w=0$ is the only stable solution of the ordinary differential equation on the half-line
\begin{align*}
   ( \lambda -  A_0(x_0,\xi', D_n ))  w(x_n) & = 0 \quad (x_n\in (0,\infty)),\\
  B_{0,j}(x_0,\xi',D_n) w(0) & = 0 \quad (j=1,\dots,m)
\end{align*}
 for all $\xi'\in\R^{n-1}$ and $\lambda\in\Lambda$ with $(\xi',\lambda)\not=0$.
%\blau{Hier sollte $D_n$ und $\xi'$ verwendet werden (Joerg).}

In this section we show that parameter-elliptic problems induce an isomorphism between parameter-dependent spaces $($in the sense of Definition~\ref{1.10.0}$)$. We focus on the case of the half-space.

\subsection{Model problems and small perturbations}

\begin{lemma}[Model problem in $\R^n$]\label{2.1}
Let $A_0(D)=\sum_{|\alpha|=2m} a_\alpha D^\alpha$
have constant coefficients $a_\alpha\in\C$, and let
$\lambda-A_0(D)$ be parameter-elliptic in
$\Lambda$. Then, for every $s,\sigma\in\R$ and every $\lambda_0>0$, the operator family
\begin{align}\label{2-2}
 \lambda-A_0\colon H_{p,\lambda}^{s,\sigma}(\R^n)\to H_{p,\lambda}^{s-2m,\sigma}(\R^n)
\end{align}
is an isomorphism for $\lambda\in\Lambda$ with $|\lambda|\ge \lambda_0$.
\end{lemma}

\begin{proof}
The result is well known in case $\sigma=0$, see \cite{Grubb95}, Theorem~1.7, or
can be obtained immediately from Mikhlin's theorem. Let us denote by $(\lambda-A_0)^{-1}_{(s,0)}$ \
the corresponding inverse operator. We use the description of the norm in $H_{p,\lambda}^{s,\sigma}(\R^n)$ from Lemma~\ref{1.11}~b).

Since $A_0$ commutes with $\spk{D^{\prime},\lambda}^{\sigma}$ and
$\spk{D^{\prime},\lambda}^{\sigma}\colon H_{p, \lambda}^{s,\sigma}(\R^n)\to H_{p, \lambda}^{s,0}(\R^n)$
is an isometric isomorphism, the inverse to \eqref{2-2} is
\[(\lambda-A_0)^{-1}_{(s,\sigma)}
     =\spk{D^{\prime},\lambda}^{-\sigma}(\lambda-A_0)^{-1}_{(s,0)}\spk{D^{\prime},\lambda}^{\sigma},\]
which then has the same uniform bound as $(\lambda-A_0)^{-1}_{(s,0)}$.
\end{proof}
Let us now pass to the situation in the half-space, where we consider the following boundary problem.
% For $j=1,\dots,m$, let
% the boundary operators be of the form
% $ B_{0,j}(D) = \sum_{|\beta|= m_j} b_{j\beta}\gamma_0 D^\beta$
% with constant coefficients $b_{j\beta}\in\C$. We set
% $B_0:=( B_{0,1}(D),\ldots, B_{0,m}(D))$. \rot{Gehört das nicht eigentlich in die Voraussetzungen von Theorem~\ref{2.2}?}

\begin{theorem}[Model problem in $\R^n_+$]\label{2.2}
Let $(\lambda-A_0,B_0)$ be parameter-elliptic in~$\Lambda$. Here again, we have $A_0(D)=\sum_{|\alpha|=2m} a_\alpha D^\alpha$ with constant coefficients $a_\alpha \in \C$, as well as $B_0:=( B_{0,1}(D),\ldots, B_{0,m}(D))$ where $ B_{0,j}(D) = \sum_{|\beta|= m_j} b_{j\beta}\gamma_0 D^\beta$
with constant coefficients $b_{j\beta}\in\C$ for $j=1,\ldots,m.$
  Then, for every
  $s>\max_j m_j +\frac{1}{p}$, $\sigma\in\R$, and $\lambda_0>0$, the family of operators
\begin{equation}\label{2-3}
 \begin{pmatrix}\lambda-A_0\\ B_0\end{pmatrix}:
 H_{p,\lambda}^{ s,\sigma}(\R^n_+)\to
  H_{p,\lambda}^{s-2m,\sigma}(\R^n_+)\times
  \prod_{j=1}^m  B_{pp,\lambda}^{s+\sigma-m_j-1/p}(\R^{n-1})
\end{equation}
is an isomorphism for $\lambda\in\Lambda$ with $|\lambda|\ge \lambda_0$.
%We denote its inverse by $L(\lambda)$ or also $L_{(s,\sigma)}(\lambda)$
%\blau{Satz geloescht (Joerg)}
\end{theorem}

\begin{proof}
The proof is similar to that of Lemma~\ref{2.1}. The result is known for $\sigma =0$, see
\cite{Grubb95}, Theorem~1.9;
let $L_{(s,0)}(\lambda)$ be the inverse. All involved operators commute with $\spk{D',\lambda}^\sigma $.
Hence the inverse operator $L_{(s,\sigma )}(\lambda)$ for general $\sigma $ is given by
 $$\spk{D',\lambda}^{-\sigma} L_{(s,0)(\lambda)}(\lambda)
    \mathrm{diag}(\spk{D',\lambda}^\sigma ,\ldots,\spk{D',\lambda}^\sigma ),$$
where the diagonal matrix acts as $\spk{D',\lambda}^\sigma $ on each space on the right-hand side of~\eqref{2-3}. Hence $L_{(s,\sigma )}(\lambda)$ has the same uniform norm-bound as
$L_{(s,0)}(\lambda)$.
\end{proof}

Motivated by the last two results, we define the parameter-dependent spaces
\begin{align}\label{DefE}
  \E^{ s,\sigma } (\Omega) & := H_{p,\lambda}^{s,\sigma}(\Omega) \qquad
  (\text{with $\Omega=\R^n$ or $\Omega=\R^n_+$}),
\end{align}
as well as
\begin{equation}\label{DefF}
\begin{aligned}
  \F^{ s,\sigma }(\R^n) & := H_{p,\lambda}^{s-2m,\sigma}(\R^n),\\
  \F^{s,\sigma}(\R^n_+) & := H_{p,\lambda}^{s-2m,\sigma}(\R^n_+)\times
  \prod_{j=1}^m B_{pp,\lambda}^{s+\sigma-m_j-1/p}(\R^{n-1})
\end{aligned}
\end{equation}
for $s,\sigma\in\R$.
\begin{align*}
  \lambda- A_0\colon \E^{s,\sigma}(\R^n) \to \F^{s,\sigma}(\R^n)\text{ and }
  \binom{\lambda-A_0}{B_0}\colon \E^{s,\sigma}(\R^n_+) \to \F^{s,\sigma}(\R^n_+)
\end{align*}
are  both isomorphisms.
Below, we will consider the case of variable coefficients which are close to constant coefficients in an appropriate sense. As a preparation, we show some auxiliary continuity results.

\begin{lemma}\label{2.3}
Let $(s,\sigma)\in\R^2$.
\begin{enumerate}[a)]
  \item Let $ A$ be a differential operator in $\R^n$ as in \eqref{2-6} and assume
\textnormal{(S1)} to hold. Let $M_A := \max_{|\alpha|= 2m} \|a_\alpha\|_{\BUC^r(\R^n)} + \max_{|\alpha| < 2m} \|a_\alpha\|_{\BUC^{\lceil r'\rceil}(\R^n)} $.
Then for every $\eps>0$ there exist constants $\delta=\delta(\epsilon,s,\sigma)>0$ and
$\lambda_0=\lambda_0(M_A) > 0$ such that
\[  \| A\|_{L(\E^{s,\sigma}(\R^n),\F^{s,\sigma}(\R^n))} < \eps \]
holds for all $\lambda\in\C$ with $|\lambda|\ge \lambda_0$ provided
$\max_{|\alpha|=2m}\|a_\alpha\|_{\infty}<\delta$.
\item Let %\blau{$B:=(B_1, \ldots, B_m)$} \rot{Evtl. außerhalb des Theorems irgendwo formulieren.} and
$(A, B)$ be a boundary
value problem of the form \eqref{2-6}--\eqref{2-7} in $\R^n_+$
and assume \textnormal{(S1)} and \textnormal{(S3)} to hold with $s>\max_j m_j + \frac{1}{p}$.
Let
\begin{align*}
     M_{A,B} :=\, \max_{|\alpha|= 2m}&  \|a_\alpha\|_{\BUC^r(\R^n_+)} + \max_{|\alpha|< 2m} \|a_\alpha\|_{\BUC^{\lceil r' \rceil}(\R^n_+)} + \\
&+
\max_{\substack{j=1,\dots,m\\|\beta|= m_j}}
\|b_{j\beta}\|_{\BUC^{k_j}(\R^{n-1})} + \max_{\substack{j=1,\dots,m\\|\beta|< m_j}}
\|b_{j\beta}\|_{\BUC^{\lceil k'_j \rceil }(\R^{n-1})} .
\end{align*}
Then for every $\eps>0$ there exist constants $\delta=\delta(\epsilon,s,\sigma)>0$ and
$\lambda_0=\lambda_0( M_{A,B})> 0$ such that
\[  \Big\|\binom{ A}{ B}\Big\|_{L(\E^{s,\sigma}(\R^n_+),\F^{s,\sigma}(\R^n_+))}
  < \eps\]
holds for all $\lambda\in\C$ with $|\lambda|\ge \lambda_0$ provided
\[ \max_{|\alpha|=2m}\|a_\alpha\|_{\infty} + \max_{\substack{j=1,\dots,m\\|\beta|= m_j}} \|b_{j\beta}\|_\infty
<\delta.\]
\end{enumerate}
\end{lemma}

\begin{proof}
a) Let $A_0=\sum_{|\alpha|= 2m}a_\alpha(x)D^\alpha$ be the principal part of $A$
and set $\wt{A}:=A-A_0$. Let $\eps>0$ be fixed and $u\in H^{s,\sigma }_{p,\lambda}(\rz^n)$ arbitrary.
Then, due to Lemma~\ref{1.12}~c), for appropriate $\epsilon'>0$ there exist $\delta(\eps',s,\sigma)>0$ and $\lambda_0(M_A)>0$ such that for $|\lambda| \geq \lambda_0$ we have
\begin{align*}
 \|A_0u\|_{H_{p,\lambda}^{s-2m,\sigma}(\R^n)}
  &\leq  \sum\limits_{|\alpha|=2m}\| a_\alpha(\cdot) D^\alpha u\|_{H_{p,\lambda}^{s-2m,\sigma} (\R^n)}\\
  &\le \epsilon' \sum_{|\alpha|=2m} \|D^\alpha u\|_{H_{p,\lambda}^{s-2m,\sigma} (\R^n)} \le
  \frac\eps2\, \|u\|_{H_{p,\lambda}^{s,\sigma}(\R^n)},
\end{align*}
given $\max_{|\alpha|=2m}\|a_\alpha\|_{\infty}<\delta.$
For $\wt{A}u$ we use Lemma \ref{1.12}~b), as we only need the fact that the coefficients are multipliers, which justifies the weaker regularity assumptions for the lower order terms. Thus, we obtain the estimate
\begin{align*}
 \|\wt{A}u\|_{H_{p,\lambda}^{s-2m,\sigma }(\R^n)}
 &\leq \sum_{|\alpha|<2m}\|a_\alpha(\cdot) D^{\alpha}u\|_{H_{p,\lambda}^{s-2m,\sigma}(\R^n)}\\
 &\leq C M_A\sum_{|\alpha|<2m}\|D^{\alpha}u\|_{H_{p,\lambda}^{s-2m,\sigma}(\R^n)}\\
 &\leq C M_A   \|u\|_{H_{p,\lambda}^{s-1,\sigma }(\R^n)}
 \leq C  M_A\spk{\lambda}^{-1} \|u\|_{H_{p,\lambda}^{s,\sigma}(\R^n)}.
\end{align*}
The last inequality holds true because we have
\begin{align*}
 \|u\|_{H_{p,\lambda}^{s-1,\sigma}(\R^n)}
 =\|\spk{D,\lambda}^{s-1}\spk{D',\lambda}^{\sigma }u\|_{L^p(\R^n)}
  \leq C \spk{\lambda}^{-1} \|\spk{D,\lambda}^{s}\spk{D',\lambda}^{\sigma }u\|_{L^p(\R^n)}
\end{align*}
uniformly in $\lambda$, since
$\spk{\lambda}\spk{\xi,\lambda}^{-1}$ is a Mikhlin multiplier with symbol estimates that are
uniform in $\lambda$. As $\spk{\lambda}^{-1}$ vanishes for $|\lambda|\to\infty$, we can choose $\lambda_0$ so large that
$C M_A \spk{\lambda}^{-1}<\frac{\eps}{2}$ whenever
$|\lambda|\ge\lambda_0$.

b) The calculations for $A$ are similar to a), we just replace the whole space estimates by the half-space estimates. For the boundary operators $B_j$ we use Lemma~\ref{1.12}~e) instead, noting that  (S3)  yields the required smoothness. Hence for $B_j$ we split off the lower order terms again. Then for a given $\eps>0$, again, for appropriate $\eps'>0$ there exist $\delta(\eps',s,\sigma)>0$ and $\lambda_0(M_{A,B})>0$ such that for $|\lambda|\geq\lambda_0$ we obtain
\begin{align*}
 \|B_{0,j}u\|_{B_{pp,\lambda}^{s+\sigma-m_j-1/p}(\R^{n-1})}
  &\leq  \sum\limits_{|\beta|=m_j}\| b_{j\beta}(\cdot)\gamma_0 D^\beta u\|_{B_{pp,\lambda}^{s+\sigma-m_j-1/p}(\R^{n-1})}\\
  &\leq \epsilon' \sum\limits_{|\beta|=m_j} \|\gamma_0  D^\beta u \|_{B_{pp,\lambda}^{s+\sigma-m_j-1/p}(\R^{n-1})}\\
  & \leq C \epsilon' \sum\limits_{|\beta|=m_j} \|D^\beta u \|_{H_{p,\lambda}^{s-m_j,\sigma}(\R^{n}_+)}
 \leq  \frac\epsilon 2\,  \|u \|_{H_{p,\lambda}^{s,\sigma}(\R^{n}_+)},
  \end{align*} given $\max_{|\alpha|=2m}\|a_\alpha\|_{\infty} + \max_{j=1,\dots,m}\max_{|\beta|=m_j} \|b_{j\beta}\|_\infty
<\delta$. Here we also used the continuity of the trace from Definition~\ref{1.7}. The lower order terms can be handled  as in a),  applying Lemma~\ref{1.12}~e) once more.
\end{proof}

\orange{
% ******************

% Bei den beiden folgenden Lemmas m\"ussen wir nochmal aufpassen: Das $\delta$ darf nicht von $\|a_\alpha\|_{\BUC^r}$ abhängen. Es h\"angt allerdings
% von der Norm $\rho$ aus dem Beweis ab. Mit einem Kompaktheitsschluss sollten wir zeigen k\"onnen, dass das $\rho$ sp\"ater (Beweis von Lemma~\ref{2.7})
% für alle $x$ eine obere Schranke hat. Daher w\"ahlen wir zuerst eine obere Schranke f\"ur das $\rho$, damit das $\delta$, damit die $x_k$,
% setzen dann die Operatoren fort und w\"ahlen dann das $\lambda_0$ in Abh\"angigkeit der Normen der fortgesetzten Operatoren.

% Vielleicht sollten wir das alles in den Beweis von Lemma~\ref{2.7} bringen? Die Abh\"angigkeiten sind doch etwas verschr\"ankt.

% *********************

}
	
\begin{lemma}[Small perturbation in $\R^n$]\label{2.4}
Let $\lambda-A_0$ be as in Lemma~$\ref{2.1}$, and let
$A=A_{0}+\wt{A}$, where
 $$\wt{A}=\wt{A}(x,D)=\sum\limits_{|\alpha|\leq 2m} \wt{a}_\alpha(x) D^\alpha.$$
Moreover, let $s,\sigma \in\rz$ and assume \textnormal{(S1)} to hold. Then there exist $\delta>0$ and $\lambda_0 \geq 1$ such that if $\max_{|\alpha|=2m}\|\wt{a}_\alpha\|_\infty<\delta$, the operator family
 $$\lambda-A\colon \E^{s,\sigma}(\R^n)\to \F^{s,\sigma}(\R^n)$$
is an isomorphism for $\lambda\in\Lambda$ with $|\lambda|\ge\lambda_0$.
\end{lemma}

\begin{proof}
Using Lemma~\ref{2.1}, we can write
 $$\lambda-A=(\lambda-A_0)\big(I-(\lambda-A_{0})^{-1}\wt{A}\big)\quad (\lambda\not=0).$$
Now let
$\rho:=\sup_{|\lambda|\ge1}\|(\lambda-A_0)^{-1}\|_{{L}(\F^{s,\sigma}(\R^n),\E^{s,\sigma }(\R^n)})$.
Applying Lemma~\ref{2.3}~a) to $\wt{A}$, we can choose $\delta$ and $\lambda_0$ such that
$\|\wt{A}\|_{{L}(\E^{s,\sigma}(\R^n),\F^{s,\sigma}(\R^n)})<1/(2\rho)$
whenever $|\lambda|\ge\lambda_0$ and $\max_{|\alpha|=2m}\|\wt{a}_\alpha\|_\infty<\delta$. Therefore,
 $$\|(\lambda-A_0)^{-1}\wt{A}\|_{{L}(\E^{s,\sigma}(\R^n)})
     <\frac{1}{2}.$$
By the usual Neumann series argument, $I-(\lambda-A_{0})^{-1}\wt{A}$ is invertible with
 $$\|(I-(\lambda-A_{0})^{-1}\wt{A})^{-1}\|_{{L}(\E^{s,\sigma}(\R^n)})< 2$$
for every $|\lambda|\ge\lambda_0$. We conclude that, for such $\lambda$,
 $$(\lambda-A)^{-1} =\big(I-(\lambda-A_{0})^{-1}\wt{A}\big)^{-1}
   (\lambda-A_0)^{-1}$$
with
\begin{align*}
 \|(\lambda-A)^{-1}&\|_{{L}(\F^{s,\sigma}(\R^n),\E^{s,\sigma }(\R^n)}) < 2
    \|(\lambda-A_0)^{-1}\|_{{L}(\F^{s,\sigma}(\R^n),\E^{s,\sigma}(\R^n)}).
\end{align*}
Using Lemma~\ref{2.1} once more completes the proof.
\end{proof}
	
\begin{theorem}[Small perturbation in $\R^n_+$]\label{2.5}
Consider the boundary value problem $(\lambda-A,B)$ with $A=A_0+\wt{A}$ and $B=B_0+\wt{B}$,
where $(\lambda-A_0,B_0)$ is as in Theorem~$\ref{2.2}$,
 $$\wt{A}=\wt{A}(x,D)=\sum\limits_{|\alpha|\leq 2m} \wt{a}_\alpha(x) D^\alpha$$
and $\wt{B}=(  \wt{B}_{1},\ldots,  \wt{B}_{m})$ with
\begin{align*}
 \wt{B}_j=\wt{B}_{j}(x,D) = \sum_{|\beta|\le m_j} \wt{b}_{j\beta}(x)\gamma_0 D^\beta
\end{align*}
and $m_j<2m$. Moreover, let $s,\sigma \in\rz$ with $s>\max_j m_j+\frac{1}{p}$, and
assume \textnormal{(S1) and (S3)} to hold.
Then there exist $\delta>0$ and $\lambda_0\geq 1$ such that if
\[ \max_{|\alpha|=2m}\|\wt{a}_\alpha\|_\infty +
 \max_{\substack{j=1,\dots,m\\|\beta|=m_j}}\|\wt{b}_{j\beta}\|_\infty<\delta,\]
then
\begin{equation}\label{2-4}
 \begin{pmatrix}\lambda-A\\ B\end{pmatrix}\colon
 \E^{ s,\sigma}(\R^n_+)\to \F^{ s,\sigma}(\R^n_+)
\end{equation}
is an isomorphism for $\lambda\in\Lambda$ with $|\lambda|\ge \lambda_0$.
\end{theorem}

\begin{proof}
We use Theorem~\ref{2.2}. If $L(\lambda)$ denotes the inverse of the map in
\eqref{2-3}, we can write
\[\binom{\lambda-A}{B}=\binom{\lambda-A_0}{B_0}+\binom{-\wt A}{\wt B}=\binom{\lambda-A_0}{B_0}\bigg(I+L(\lambda)\binom{-\wt A}{\wt B}\bigg). \]
In analogy to the proof of Lemma~\ref{2.4} above, we take \[\rho:=\sup_{|\lambda|\geq 1} \| L(\lambda)\|_{L(\F^{s,\sigma}(\R^n_+),\E^{s,\sigma}(\R^n_+))}.\]
Now, due to Lemma~\ref{2.3}~b), we can choose a $\delta>0$ and a $\lambda_0\ge 1$ such that
\[  \Big\|\binom{- \wt A}{ \wt B}\Big\|_{L(\E^{s,\sigma}(\R^n_+),\F^{s,\sigma}(\R^n_+))}
  < \frac{1}{2\rho}\]
for all $\lambda\in\C$ with $|\lambda|\ge \lambda_0$ provided
\[\max_{|\alpha|=2m}\|a_\alpha\|_{\infty} + \max_{\substack{j=1,\dots,m\\|\beta|=m_j}} \|b_{j\beta}\|_\infty<\delta.\]
Therefore \[\bigg\|L(\lambda)\binom{-\wt A}{\wt B}\bigg\|_{L(\E^{s,\sigma}(\R^n_+))}<\frac{1}{2},\]
which allows us to use the Neumann series just as above, yielding the desired isomorphism. %\blau{Satz geloescht (Joerg)}
%For simplicity, we denote the solution operator again by $L(\lambda)$.
\end{proof}

\subsection{General boundary value problems}%\blau{Subsection eingebaut (Joerg).}

The analysis of the general case of variable coefficients is based on the classical method of freezing the coefficients.

In the following, let $(\lambda-A,B)$ be a boundary value problem in $\R^n_+$
of the form \eqref{2-6}--\eqref{2-7} which is parameter-elliptic in $\Lambda$
for all $x\in \overline{\R^n_+}\cup\{\infty\}$. Let $(s,\sigma)\in\R^2$, and assume
the validity of (S1)--(S4).

For every $x_0\in\R^n_+$, we consider the model problem
$ \lambda-A_0(x_0,D)$ with frozen coefficients $a_\alpha(x_0)\in \C$ and
without lower-order terms. By the assumption of parameter-ellipticity,
we can apply Lemma~\ref{2.1} and obtain the existence of the inverse operator
 $$(\lambda-A_0(x_0,D))^{-1}\in L(\F^{s,\sigma}(\R^n), \E^{s,\sigma}(\R^n))$$
for $\lambda \in \Lambda$. In the same way, for $x_0\in\R^{n-1}\cup\{\infty\}$ and $s>\max_j m_j+\frac{1}{p}$, we obtain from
Theorem~\ref{2.2} the existence of the inverse operator
 $$L_{x_0}(\lambda):=\binom{\lambda-A_0(x_0,D)}{B_0(x_0,D)}^{-1}\in L(\F^{s,\sigma}(\R^n_+), \E^{s,\sigma}(\R^n_+)).$$
%\rot{Irgendwo erwähnen, dass $\lambda\in\Lambda$ gilt.}
\begin{lemma}
\label{2.5a}
With the above notation,
 \begin{align*}
  \rho_{A,B} := & \sup_{\substack{x_0\in\R^n_+ \\ \lambda\in\Lambda,\,|\lambda|\ge \lambda_0}}
  \|(\lambda-A_0(x_0,D))^{-1}\|_{L(\F^{s,\sigma}(\R^n), \E^{s,\sigma}(\R^n))} +\\
  & + \sup_{\substack{x_0\in\R^{n-1}\\ \lambda\in\Lambda,\,|\lambda|\ge\lambda_0}}
  \|L_{x_0}(\lambda)\|_{L(\F^{s,\sigma}(\R^n_+), \E^{s,\sigma}(\R^n_+))}
  <\infty
 \end{align*}
for every choice of $\lambda_0>0$.
\end{lemma}
\begin{proof}
Let us consider the first supremum with some fixed $\lambda_0>0$.
Assume this supremum to be infinite. Then there exist sequences $(x_k)_{k\in\N}\subset\R^{n}_+$ and $(\lambda_k)_{k\in\N}\subset\Lambda$
with $|\lambda_k|\ge \lambda_0$ such that $\|(\lambda_k-A_0(x_k,D))^{-1}\|\to\infty$ for $k\to\infty$. By passing to a subsequence we may assume that $x_k\to x^*$ for $k\to\infty$ where either $x^*\in \overline{\R^n_+}$ or $x^*=\infty$.
Now write
 $$\lambda_k-A_0(x_k,D)=\lambda_k-A_0(x^*,D)-\wt{A}^k(D),\qquad
   \wt{A}^k(D):=A_0(x_k,D)-A_0(x^*,D).$$
Since $A_0(x^*,D)$ satisfies the assumptions of Lemma~\ref{2.1}, we get
 $$\lambda_k-A_0(x_k,D)=(\lambda_k-A_0(x^*,D))\big[1-(\lambda_k-A_0(x^*,D))^{-1}
   \wt{A}^k(D)\big].$$
Now let
 $$\rho^*=\sup_{\lambda\in\Lambda,\,|\lambda|\ge \lambda_0}
  \|(\lambda-A_0(x^*,D))^{-1}\|_{L(\F^{s,\sigma}(\R^n), \E^{s,\sigma}(\R^n))}$$
which is finite due to Lemma~\ref{2.1}. Moreover, observe that
\begin{align*}
 \|\wt{A}^k(D)u\|_{\F^{s,\sigma}(\R^n)}
 \le & \, \|\spk{D,\lambda}^{s-2m}\spk{D',\lambda}^{\sigma}
 \wt{A}^k(D)\spk{D,\lambda}^{-s}\spk{D',\lambda}^{-\sigma}\|_{L(L^p(\R^n))}
 \times\\
 &\times  \|\spk{D,\lambda}^{s}\spk{D',\lambda}^{\sigma}u\|_{L^p(\R^n)}\\
 =&\|\spk{D,\lambda}^{-2m}\wt{A}^k(D)\|_{L(L^p(\R^n))}
 \|u\|_{\E^{s,\sigma}(\R^n)}.
\end{align*}
It is a straightforward consequence of Mikhlin's Theorem that
 $$\sup_{\lambda\in\C}\|\spk{D,\lambda}^{-2m}\wt{A}^k(D)\|_{L(L^p(\R^n))}
   \xrightarrow{k\to\infty}0,$$
since the (constant) coefficients of $\wt{A}^k(D)$ tend to zero with $k\to\infty$. It follows that
 $$\sup_{\lambda\in\Lambda,\,|\lambda|\ge \lambda_0}
  \|(\lambda-A_0(x^*,D))^{-1}\wt{A}^k(D)\|_{L(\E^{s,\sigma}(\R^n))}
  \le \frac{1}{2}$$
for all sufficiently large $k$. As above, using the Neumann series, we conclude that
\begin{align*}
 \|(\lambda_k-&A_0(x_k,D))^{-1}\|_{L(\mathbb{F}^{s,\sigma}_{\lambda_k}(\R^n), \mathbb{E}^{s,\sigma}_{\lambda_k}(\R^n))}\\
 &\le
 2\|(\lambda_k-A_0(x^*,D))^{-1}\|_{L(\mathbb{F}^{s,\sigma}_{\lambda_k}(\R^n), \mathbb{E}^{s,\sigma}_{\lambda_k}(\R^n))}\le 2\rho^*
\end{align*}
for all sufficiently large $k$. This is a contradiction.
% \forget{
% \rot{Sollten hier $\lambda_k$ statt $\lambda$ stehen?} Hence Lemma~\ref{2.4} together with the

% Let us consider the first supremum. Assume that this supremum would be infinite.
% Then there exist sequences $(x_k)_{k\in\N}\subset\blau{\R^{n}_+}$ and $(\lambda_k)_{k\in\N}\subset\Lambda$ with
% $\|(\lambda_k-A_0(\blau{x_k},D))^{-1}\|\to\infty$
% for $k\to\infty$. By passing to a subsequence we may assume that $x_k\to x^*$
% for $k\to\infty$ where either $x^*\in\blau{\overline{\R^n_+}}$ %\rot{Was passiert, falls $x^*\in\R^{n-1}$ gilt?}
% or $x^*=\infty$. Now write
%  $$\lambda-A_0(x_k,D)=\lambda-A_0(x^*,D)-\wt{A}^k(D),\qquad
%    \wt{A}^k(D):=A_0(x_k,D)-A_0(x^*,D)$$
% $\lambda-A_0(x^*,D)$ satisfies the assumptions of Lemma~\ref{2.1}. \rot{Sollten hier $\lambda_k$ statt $\lambda$ stehen?} Hence Lemma~\ref{2.4} together with the last formula in its proof \blau{shows} that
% $\|(\lambda-A_0(x_k,D))^{-1}\|$ is uniformly bounded for $|\lambda|\ge\lambda_0$  \rot{$\lambda_0$ doppelt belegt, nicht das von Lemma~\ref{2.4}.} and $k$ sufficiently large. \rot{Also auch uniform in $k$, versteht man das so?} This is a contradiction. \rot{Gilt $\lambda_k\to\infty$ für $k\to\infty$? Wenn nicht, wie funktioniert das sonst? Ich habe den Eindruck, dass man mit $\lambda_k$ im Beweis eigentlich nichts macht. Falls dem so ist, könnte man vielleicht das Supremum über $\lambda$ stehen lassen.}

% For $\|L_{x_0}(\lambda)\|$ we argue analogously, using
% Theorems~\ref{2.2} and \ref{2.5}.
% }
\end{proof}

\begin{remark}\label{rem_ext} In the following, we construct a finite covering of $\R^n_+$ consisting of balls and the complement of a ball centered in the origin. Afterwards, we need to extend the coefficients of the localized problems to $\R^n$, $\R^n_+$, and $\R^{n-1}$, respectively. To this end, we will use a general extension function.
We fix
$\chi\in C^\infty([0,\infty))$ with $0\le \chi\le 1$,   $\chi(z)=1$ for
$0\le z\le 1$ and $\chi(z)=0$ for $z\ge 2$ and define the function $ \chi_U\colon \R^n\to \R^n\cup\{\infty\}$ via
  \begin{align*}
        \chi_{U}(x):= \begin{cases}
        \frac{x}{|x|^{2}} +
   \chi\big(\frac{r'}{|x|}\big)(x-\frac{x}{|x|^{2}}) & \text{if there exists }  r'>0:\, U=\R^n\backslash B(0,r'),\\
        x' + \chi\big(\frac{|x -  x'|}{r'}\big)(x-x')
     & \text{if there exist } r'>0,\; x'\in\R^n: \,U=B(x',r'). \\
        \end{cases}
    \end{align*}
   % \rot{$\chi_U$ ist für $U=\R^n\backslash B(0,r')$ im Nullpunkt nicht definiert.}
   The function $\chi_U$ coincides with the identity on $U$ and is later compatible with the parameter-ellipticity of the local operators. Since we use reflection techniques for the construction of $\chi_U$, it is crucial that our covering consists of balls and the complement of a ball centered in the origin.
\end{remark}

For the localization, we
first apply Lemma~\ref{2.3} with $\epsilon:= \frac1{2\rho_{A,B}}$,
where $\rho_{A,B}$ is taken from Lemma~\ref{2.5a}. We fix \begin{equation}\label{2-7b}
   \delta_0 := \delta\Big(\frac1{2\rho_{A,B}}, s, \sigma\Big) >0
\end{equation}
as being defined in Lemma~\ref{2.3}. % \rot{Heißt dort $\delta$.}
Let $x_0:=\infty$ and $U_0:=\{ x\in\R^n: |x|>r_0\}$
where $r_0$ is sufficiently large such that
\begin{align}
\begin{split}
\label{2-8.0}
 \max_{|\alpha|=2m} &
 \| a_\alpha(\cdot)-a_\alpha(x_0)\|_{L^\infty(\wt U_0\cap \R^n_+)}+ \\
 &+\max_{\substack{j=1,\dots,m \\ |\beta|=m_j}}
 \| b_{j\beta}(\cdot) - b_{j\beta}(x_0) \|_{L^\infty(\wt U_0\cap\R^{n-1})}
<\delta_0
\end{split}
\end{align}
with $\wt U_0:=\left\{x\in \R^n: |x|>\frac{r_0}{2}\right\} $
(this is possible due to (S2) and (S4)).
As the coefficients of $A$ and $B$ are continuous
and $\overline{B(0,r_0)}\cap \R^{n-1}$ is compact, there exists a finite covering
$\R^{n-1} \subset \bigcup_{k=0}^{K_0} U_k $ with $U_k:= B(x_k,r_k)$ for $k=1,\dots,K_0$, where
$x_k\in\R^{n-1}$ and $r_k>0$ are chosen such that
\begin{align}
\begin{split}\label{2-8}
 \max_{|\alpha|=2m} &
  \| a_\alpha(\cdot)-a_\alpha(x_k)\|_{L^\infty(\wt U_k\cap \R^n_+)}+\\
&+ \max_{\substack{j=1,\dots,m \\ |\beta|=m_j}} \| b_{j\beta}(\cdot) - b_{j\beta}(x_k) \|_{L^\infty(\wt U_k\cap\R^{n-1})}<\delta_0
\end{split}
\end{align}
with $\wt U_k := B(x_k,2r_k)$ for $k=1,\ldots,K_0$.
We set
\begin{align*}
    \delta_{\max} := \sup\bigg\{\delta>0 \,:\, \R^{n-1}\times [0,\delta]\subset
    \bigcup_{k=0}^{K_{0}}U_{k} \bigg\}.
%    \bigcup_{k=0}^{K_{0}}U_{k} \right\}.
\end{align*}
Similarly, as
$\R^n_+\setminus\bigcup_{k=0}^{K_0} U_k$
%$\R^n_+\setminus\bigcup_{k=0}^{K_0} U_k$
is compact, we can choose
$x_k\in \R^n_+$ and $0<r_k<\frac{\delta_{\max}}{2}$ for $k=K_0+1,\dots,K$ such that $U_k := B(x_k,r_k)\subset\left\{z\in\R^{n}\,|\, z_{n}>\frac{\delta_{\max}}{2}\right\}$,
\begin{equation}
  \label{2-9}
  \max_{|\alpha|=2m} \| a_\alpha(\cdot)-a_\alpha(x_k)\|_{L^\infty(\wt U_k)} <
  \delta_0%\quad (k=K_0+1,\dots,K)
\end{equation}
with $\wt U_k:=B(x_k,2r_k)$ for $k=K_0+1,\ldots, K$ and
%\rot{$\overline{\R^n_+} \subset U_{0}\cup\ldots\cup U_{K}$ (Joerg)}.
$\overline{\R^n_+} \subset \bigcup_{k=0}^{K} U_k$.

% \rot{
% Ich wuerde die folgende Bemerkung etwas umschreiben. Ich finde die Terminologie "local" und "extension" nicht 100\%ig passend. Alte Version im tex-file. (Joerg) }

\begin{remark}[Local operators and extensions]
\label{2.6}
Let $x_0,\dots, x_K$ be chosen as above.
% Moreover, let
% $\chi\in C^\infty([0,\infty))$ with $0\le \chi\le 1$,   $\chi(z)=1$ for
% $0\le z\le 1$ and $\chi(z)=0$ for $z\ge 2$. Set
%   $$\chi_0(x):=\frac{x}{|x|^{2}} +
%   \chi\Big(\frac{r_{0}}{|x|}\Big)\Big(x-\frac{x}{|x|^{2}}\Big)$$
% and
%   $$\chi_k(x):=x_k + \chi\Big(\frac{|x -  x_k|}{r_k}\Big)(x-x_k),
%      \qquad k=1,\ldots, K.$$
Starting out from the coefficient functions $a_\alpha$ and $b_{j\beta}$ let us define
\begin{align*}
 \acf_\alpha^{k}(x)&:=a_\alpha(\chi_{U_k}(x))\qquad (x\in \overline{\R^n_+}),\\
 \bcf_{j\beta}^{k}(x)&:=b_{j\beta}(\chi_{U_k}(x))\qquad (x\in \R^{n-1})
\end{align*}
for $k=0,\ldots,K_0$ and, for $k=K_0+1,\ldots,K$,
\begin{align*}
 \acf_\alpha^{k}(x):=a_\alpha(\chi_{U_k}(x))\qquad (x\in\R^n).
\end{align*}
Here the function $\chi_{U_k}$ is defined as in Remark~\ref{rem_ext}.
 These new coefficients have the same smoothness as before.
 %\rot{Problem für $k=0$ und $x=0$. Falls man $a_\alpha$ ab einem bestimmten Radius konstant werden lässt, bräuchte man hier kein $\chi $ mehr, sondern würde $a_\alpha^0(x):= a_\alpha(\infty)$ für alle $x\in\overline{\R^n_+}$ setzen.}
$a^k_\alpha$ coincides with $a_\alpha$ on $U_k\cap \overline{\R^n_+}$
and $U_k\cap\R^n$, respectively,
$\bcf_{j\beta}^{k}$ coincides with $\bcf_{j\beta}^{k}$ on $U_k\cap\R^{n-1}$.
By \eqref{2-8.0}--\eqref{2-9}, we have
  \begin{alignat}{4}
   \big\|   \acf_\alpha^{k}(\cdot) -   \acf_\alpha(  x_k)
  \big\|_{L^\infty(\R^n_+)} & < \delta_0 && \text{ for } k=0,\dots,K_0,\notag\\
   \big\|   \acf_\alpha^{k}(\cdot) -   \acf_\alpha(  x_k)
  \big\|_{L^\infty(\R^n )} & < \delta_0 && \text{ for }
   k=K_0+1,\dots,K,\label{2-9ex}\\
  \big\|   \bcf_{j \beta}^{k}(\cdot) -   \bcf_{j \beta}(  x_k)
  \big\|_{L^\infty(\R^{n-1})} & < \delta_0 && \text{ for } k=0,\dots,K_0. \notag
  \end{alignat}
%\rot{Das sehe ich nicht, da nicht immer $\chi_k(x)\in U_k$.
%Muesste man dazu nicht \eqref{2-8} and \eqref{2-9} fuer $2U_k$ fordern
%(also Baelle mit Radius $2r_k$)?}
%\rot{Warum haut das jetzt bei $k=0$ hin, wenn man außerhalb von $\wt{U_k}$ ist? Hier ist man ja jetzt nicht konstant $a(\infty)$}
With the new coefficient functions we associate the operators
$A^k$ and $B^{k}=(  B_1^k, \ldots, $ $ B_m^k)$ via
\begin{align*}
 \A^k=A^k(x,D):= \sum_{|\alpha|\le 2m} a_\alpha^k(x) D^\alpha,\qquad
 \B_j^k=\B_j^k(x,D) & :=\sum_{|\beta|\le m_j}b_{j\beta}^k(x) \gamma_0 D^\beta.
  \end{align*}

We remark that the localization procedure contains a subtlety concerning the constants $\delta$ and $\lambda_0$ in
Lemmas~\ref{2.3}--\ref{2.4} and Theorem~\ref{2.5}. We defined the neighborhoods $U_k$ and the radii $r_k$ in dependence
of $\delta_0$ which depends only on $\rho_{A,B}$, $s$, and $\sigma$, see \eqref{2-7b}. For the new coefficients $a_\alpha^k,\,
b_{j\beta}^k$,
the $\|\cdot\|_\infty$-norm  still satisfies the desired smallness conditions, as seen in \eqref{2-9ex}.
However, as $\chi_{U_k}$
appears in the definition of the new coefficients, the $\BUC^r$-norm and $\BUC^{k_j}$-norm of the new coefficients, respectively, depend on $U_k$ and therefore on
the radius $r_k$.   Here, it is important that  $\delta_0$ does not depend on the $\BUC^r$-norm (in contrast to $\lambda_0$, see Lemma~\ref{2.3}). Due to this, the above modification of the coefficients might lead to a larger constant $\lambda_0$,
but we do not have to redefine the radii $r_k$, which prevents a circular reasoning in the definition of $U_k$.
\end{remark}

\begin{lemma}\label{2.7}
% Let $(\lambda-A,B)$ be parameter-elliptic in the sector $\Lambda$.
Let $s,\sigma \in\rz$ with $s>\max_j m_j+\frac{1}{p}$, and assume
\textnormal{(S1)--(S4)} to hold.
Then there exists a $\lambda_0>0$ such that the operators
\begin{align*}
  \begin{pmatrix}\lambda-\A^k\\ \B^k\end{pmatrix} & \colon
 \E^{s,\sigma}(\R^n_+)\to \F^{s,\sigma}(\R^n_+)
 \quad (k=0,\dots,K_0),\\
 \lambda-\A^k & \colon
 \E^{s,\sigma}(\R^n )\to \F^{s,\sigma}(\R^n )\quad (k=K_0+1,\dots,K)
\end{align*}
defined in Remark~\ref{2.6}
 are isomorphisms for every $\lambda\in\Lambda$ with $|\lambda|\ge \lambda_0$.
 We denote the inverse operators by $\Lhs _k(\lambda)$ for $k=0,\dots,K$.
\end{lemma}

\begin{proof}
We split the operators into a part with constant coefficients and a
perturbation, i.e., $A^k=A^k_0+\wt A^k$ with
\[A_0^k=\sum_{|\alpha|= 2m}a_\alpha(x_k) D^\alpha, \qquad
  \wt A^k=\sum_{|\alpha|= 2m}
    \big(a_\alpha^k(\cdot)-a_\alpha(x_k)\big ) D^\alpha+\sum_{|\alpha|< 2m}
    a_\alpha^k(\cdot) D^\alpha. \]
The $B^k$ can be decomposed in a similar way. Due to the smallness property
\eqref{2-9ex}, the considered operators thus fit into the setting of
Lemma~\ref{2.4} and Theorem~\ref{2.5}, respectively.
This yields the assertion.
\end{proof}

In the following, we will fix a smooth partition of unity
$\varphi_k\in C^\infty(\R^n)$, $k=0,\dots,K$, with
   $\supp\varphi_k\subset U_k$,  $0\le\varphi_k\le 1$, and
%\rot{$\varphi_1+\ldots+\varphi_K=1$ (Joerg)}
   $\sum_{k=0}^K \varphi_k = 1$ on $\overline{\R^n_+}$.
In addition, we fix functions $\psi_k\in C^\infty(\R^n)$ with
$0\le \psi_k\le 1$,  $\supp\psi_k\subset U_k$ and $\psi_k=1$ on
$\supp\varphi_k$.
We can solve  \eqref{2-1}  locally in $U_k$, using the extended local operators in the half-space
and in the whole space and their inverses $\Lhs _k(\lambda)$.
However, the solution operators $\Lhs _k(\lambda)$ are not local, so we have to multiply the
half-space solution
by~$\psi_k$. In this way, commutators appear, which are estimated in the following
lemma. We write $[\cdot,\cdot]$ for the standard commutator and use the notation
$ \psi_k$ also for the operator of multiplication by $ \psi_k$.
For the boundary operators, the commutator $[\B^k,\psi_k]$ is defined as
\[[\B^k,\psi_k] u = \B^k ( \psi_k   u)
- (\gamma_0 \psi_k) \B^k    u.\]

\begin{lemma}
  \label{2.8}
Let $s,\sigma \in\rz$ with $s>\max_j m_j+\frac{1}{p}$, and assume
\textnormal{(S1)--(S4)} to hold.
% \forget{
% \begin{alignat*}{6}
%   \binom{
%     C_k(\lambda)&:=[\A^k,\psi_k]}{[\B^k,\psi_k]}
%     \Lhs _k(\lambda)&&\in L(\mathbb F^{s,\sigma}(\R^n_+),\mathbb F^{ s+1,\sigma}
%   (\R^n_+))&\quad & (k=0,\dots,K_0),\\
%      C_k(\lambda)&:=[\A^k,\psi_k]\Lhs _k(\lambda)
%     &&\in L(\mathbb F^{s,\sigma}(\R^n ),\mathbb F^{ s+1,\sigma}
%   (\R^n ))&& (k=K_0+1,\dots, K).
%   \end{alignat*}
% }
Let $R_0(\lambda)$ be defined on $\F^{s,\sigma}(\R^n_+)$ by
% Let $(f,g)\in \mathbb F^{s,\sigma}(\R^n_+)$, and set $v := R_0(\lambda)(f,g)$
%with $R_0(\lambda)$ being defined by
\begin{equation}\label{2-10}
 R_0(\lambda)\binom f g := \sum_{k=0}^{K_0}  \psi_k
 L_k(\lambda)\binom{ \varphi_k f}{(\gamma_0 \varphi_k) g}
  + \sum_{k=K_0+1}^{K}  \psi_k \Lhs _k(\lambda)( \varphi_k f).
\end{equation}
Then
  \begin{equation}
    \label{2-11}
    \binom {\lambda -A}{B}R_0(\lambda) = 1+C(\lambda)
%    \binom {\lambda -A}{B}v = (1+C(\lambda))\binom f g,
  \end{equation}
  where $C(\lambda)\in L(\F^{s,\sigma}(\R^n_+),
  \F^{ s+1,\sigma } (\R^n_+))$, and there exists a $\lambda_0>0$ such that  $1+C(\lambda)\in L(\F^{s,\sigma}(\R^n_+))$ is invertible for all $\lambda \in \Lambda$ with $|\lambda|\ge\lambda_0$.

\end{lemma}

\begin{proof}
As first step of the proof we show the commutator estimates
  \begin{align*}
  C_k(\lambda)&:=\binom{
    -[\A^k,\psi_k]}{[\B^k,\psi_k]}\Lhs _k(\lambda)
    \in L(\F^{s,\sigma}(\R^n_+),\F^{ s+1,\sigma}(\R^n_+))
     \qquad (k=0,\dots,K_0),\\[1ex]
    C_k(\lambda)&:=-[\A^k,\psi_k]\Lhs _k(\lambda)
    \in L(\F^{s,\sigma}(\R^n ),\F^{ s+1,\sigma}(\R^n ))\qquad (k=K_0+1,\dots, K).
  \end{align*}
We shall only consider the case $k=0,\dots,K_0$, since the proof for $k=K_0+1,\dots,K$ is analogous (and simpler).

The operator $[\A^k,\psi_k]$ is a differential operator of order not greater than $2m-1$. Therefore, it is a bounded operator
\[ [\A^k,\psi_k]\colon
\E^{s,\sigma}(\R^n_+)=H_{p,\lambda}^{s,\sigma}(\R^n_+)\to
H_{p,\lambda}^{s-2m+1,\sigma}(\R^n_+).\]
For the boundary operators, we have for $j=1,\dots,m$
\begin{align*}
  \B_j^k(\psi_k  u) & =\sum_{|\beta|\le m_j}
 b^k_{j\beta} \gamma_0 D^\beta(\psi_k u) \\
 & = \sum_{|\beta|\le m_j}
 b^k_{j\beta} \gamma_0 \Big(  \psi_k D^\beta  u +
 \sum_{\gamma\le \beta, \gamma \neq \beta}
  c_{j,k,\beta,\gamma}(\cdot) D^\gamma  u \Big) \\
 & = (\gamma_0  \psi_k) B_j^k u + \sum_{\beta,\gamma}
b^k_{j\beta}(\gamma_0  c_{j,k,\beta,\gamma})
 \gamma_0 D^\gamma u,
\end{align*}
where the coefficients $c_{j,k,\beta,\gamma}$ depend on $ \psi_k$. %\rot{Die 2. Summe läuft eigentlich über weniger Multiindices, oder? Ist das Absicht?}
Consequently,
the operator $[\B^k_j,\psi_k]$ is a boundary operator of order
not greater than $m_j-1$. In the case $m_j=0$, this operator is zero.
Therefore, $[\B_j^k,\psi_k]$ is
 continuous as an operator
\[  [\B_j^k,\psi_k]\colon \E^{s,\sigma}(\R^n_+)\to
 B_{pp,\lambda}^{s+\sigma-m_j+1-1/p}(\R^{n-1}).\]
Hence the commutator estimates are true, since
$\Lhs _k(\lambda)\in L(\F^{s,\sigma}(\R^n_+), \E^{s,\sigma}(\R^n_+))$ by Lemma~\ref{2.7}.

Now let $v:= R_0(\lambda)(f,g)$. We write
\[  \binom{\lambda -A}{B} v   = \sum_{k=1}^{K_0}
  \binom{\lambda -A}{B}  \psi_k
  L_k(\lambda)\binom{\varphi_k f}{(\gamma_0\varphi_k) g}
  + \sum_{k=K_0+1}^{K} \binom{\lambda -A}{B}  \psi_k
  \Lhs _k(\lambda)(\varphi_k f)\]
and treat each term separately. For $k=1,\dots,K_0$, we obtain
\begin{align*}
\binom{\lambda -A}{B} &   \psi_k    \Lhs _k(\lambda)
  \binom{ \phi_k  f}{
  (\gamma_0 \phi_k ) g}  =  \binom{\lambda -\A^k}{\B^k}
  \psi_k \Lhs _k(\lambda)
  \binom{ \phi_k  f}{
  (\gamma_0 \phi_k ) g} \\
  & =  \Bigg[ \binom{\psi_k(\lambda -\A^k)}{(\gamma_0 \psi_k)\B^k}
  \Lhs _k(\lambda) + \binom{
   -[\A^k,\psi_k]}{[\B^k,\psi_k]}
    \Lhs _k(\lambda) \Bigg]
    \binom{ \phi_k  f}{
  (\gamma_0 \phi_k ) g} \\
  & =  \binom{\psi_k \varphi_k f}{(\gamma_0 \psi_k)(\gamma_0\varphi_k) g} +   C_k(\lambda)
  \binom{\varphi_k f}{(\gamma_0\varphi_k) g}\\
& = \binom{\varphi_k f}{(\gamma_0\varphi_k) g} +   C_k (\lambda)
  \binom{\varphi_k f}{(\gamma_0\varphi_k) g}.
\end{align*}
For $k=K_0+1,\dots,K$, we obtain in the same way
\[ \binom {\lambda -A}{B} \psi_k \Lhs _k(\lambda)(\varphi_k f) =
\binom{\varphi_k f + C_k(\lambda)
 ( \varphi_k f)}{0}.\]
Summing up over $k$ yields
\[\binom {\lambda -A}{B} v = (1+C(\lambda))\binom f g\]
with
\[ C(\lambda)\binom f g := \sum_{k=0}^{K_0}  C_k(\lambda)
  \binom{\varphi_k f}{(\gamma_0\varphi_k) g} + \sum_{k=K_0+1}^K
   \binom{C_k(\lambda)
  (\varphi_k f)}{0}.\]
  Note that for sake of readability we have dropped the extensions and restrictions from our notation, here. More precisely, the  upper entry in the last term above would be $r_{\R^n_+} C_k(\lambda)
  e^0_{\R^n_+}\varphi_k f$.

From the above commutator estimates and the fact that multiplication by $\varphi_k$  preserves the smoothness, we obtain $C(\lambda)\in L(\F^{s,\sigma}(\R^n_+),\F^{s+1,\sigma}  (\R^n_+))$.

Proceeding as in the proof of Lemma~\ref{2.3}~a) for the lower order terms and using the Neumann series as in Lemma~\ref{2.4}, we obtain that for sufficiently large $\lambda$, the
  operator $1+C(\lambda) \in
  L(\F^{s,\sigma}(\R^n_+))$ is invertible, and the norm of the inverse is not
  greater than $2$.
\end{proof}

The last result provides a solution operator for the boundary value problem~\eqref{2-1}. To show uniqueness, the following observation will be useful.

\begin{lemma}
  \label{2.9}
  Let $E,F$ be Banach spaces, and let $T\in L(E,F)$ be a retraction, i.e., there
  exists $R\in L(F,E)$ with $TR = \id_F$. Let $E_0$ be a dense subset of $E$. %\rot{Ist das wirklich alles an Bedingung? Braucht man nicht wenigstens linearer Unterraum oder so?}
  If $T|_{E_0}
  \colon E_0\to F$ is injective, then $T$ is injective.
\end{lemma}

\begin{proof}
Let $f\in F$ and $u\in E$ with $Tu=f$. Choose a sequence $(u_n)_{n\in\N}\subset E_0$
  with $u_n\to u\;(n\to\infty)$ in $E$. As $T|_{E_0}$ is injective, we have
  $u_n = Rf_n$, where $f_n:= Tu_n$. With the continuity of $T$, we see
  $f_n = Tu_n \to Tu = f$ in $F$, and from the continuity of $R$ we get
  $u_n = Rf_n \to Rf$ in $E$. As the limit is unique, this yields $u=Rf$, which shows
  the injectivity of $T$.
\end{proof}

The following theorem is the key result of this section.

\begin{theorem}
  \label{2.10}
Let $p \in (1, \infty)$ and $s,\sigma \in\rz$ with $s>\max_j m_j+\frac{1}{p}$. Let $(\lambda-A,B)$ be a boundary value problem in $\R^n_+$
of the form \eqref{2-6}--\eqref{2-7} which is parameter-elliptic in $\Lambda$
for all $x\in\overline{\R^n_+}\cup\{\infty\}$, and assume
\textnormal{(S1)--(S4)} to hold. Then, there exists a $\lambda_0>0$ such that for
every $\lambda\in\Lambda$ with $|\lambda|\ge \lambda_0$, the operator
\begin{equation}\label{2-12}
  \begin{pmatrix}\lambda-  A \\   B \end{pmatrix}   \colon
 \E^{s,\sigma}(\R^n_+)\to \F^{s,\sigma}(\R^n_+)
\end{equation}
is an isomorphism. Its inverse is given by
  \[  R(\lambda)= R_0(\lambda)(1+C(\lambda))^{-1}\in L(
  \F^{s,\sigma}(\R^n_+),\E^{s,\sigma}(\R^n_+)),\]
where $R_0(\lambda)$ and $C(\lambda)$ are defined in Lemma~\ref{2.8}.
\end{theorem}

\begin{proof}
Let  $\lambda_0$ be as in Lemma~\ref{2.8}.
For $R(\lambda) = R_0(\lambda) (1+C(\lambda))^{-1} $, we have
$R(\lambda)\in L(\F^{s,\sigma}(\R^n_+), \E^{s,\sigma}(\R^n_+))$
by Lemma~\ref{2.7} and Lemma~\ref{2.8}. From \eqref{2-11} we obtain
\[ \binom{\lambda-A}B R(\lambda) = (1+C(\lambda))(1+C(\lambda))^{-1}  =
\id\nolimits_{\F^{s,\sigma}(\R^n_+)}.\]
In particular, the operator in \eqref{2-12} is surjective.

To show injectivity (i.e., uniqueness of the solution), we
remark that $ \F^{2m,0}(\R^n_+)$ and $\E^{2m,0}(\R^n_+)$ are
classical spaces, and therefore we obtain unique solvability in these spaces
(see, e.g., \cite{Agranovich-Faierman-Denk97}, Theorem~2.1).
In particular, the restriction of the operator \eqref{2-12} to $\mathscr S(
{\R^n_+})$ is injective. Now we can apply Lemma~\ref{2.9}
with $T = \binom{\lambda-A}B$ and $R=R(\lambda)$ in the
spaces $E = \E^{s,\sigma}(\R^n_+)$,
$F= \F^{s,\sigma}(\R^n_+)$, and $E_0 = \mathscr S({\R^n_+})$.
%(see Remark~\ref{1.3}~c) and Remark~\ref{1.6}~c) for the density).
\end{proof}

\begin{corollary}\label{2.11}
 In the situation of Theorem \ref{2.10}, let additionally $\sigma\in(-\infty,0]$. Then, there exists a $\lambda_0>0$ such that for
every $\lambda\in\Lambda$ with $|\lambda|\ge \lambda_0$ and
\[ (f,g)\in H_{p,\lambda}^{s-2m}(\R^n_+)\times\prod_{j=1}^m B_{pp,\lambda}^{s+\sigma-m_j-1/p}
(\R^{n-1}) \]
the boundary value
problem \eqref{2-1} has a unique solution $u\in H_{p,\lambda}^{s,\sigma}(\R^n_+)$. In particular, we have
$u\in H_{p,\lambda}^{s+\sigma}(\R^n_+)$ and
\[ \|u\|_{H_{p,\lambda}^{s+\sigma}(\R^n_+)} \le C \Big( \|f\|_{H_{p,\lambda}^{s-2m}(\R^n_+)}
+ \sum_{j=1}^m \|g_j\|_{B_{pp,\lambda}^{s+\sigma-m_j-1/p}
(\R^{n-1}) }\Big)\]
with a constant $C$ independent of $\lambda$.
\end{corollary}

\begin{proof}
This follows immediately from Theorem~\ref{2.10} and  the continuous embeddings
$H_{p,\lambda}^{s-2m}(\R^n_+)\subset H_{p,\lambda}^{s-2m,\sigma}(\R^n_+)$ and
$H_{p,\lambda}^{s,\sigma}(\R^n_+)\subset H_{p,\lambda}^{s+\sigma}(\R^n_+)$. %\rot{Diese Inklusionen wurden glaube ich nirgends $\lambda$-abhängig gezeigt.}
\end{proof}

In Theorem~\ref{2.10}, we considered the half-space case. For an operator $A$ acting in the whole space,
the analog results hold, where the proofs are similar but much simpler, due to the absence of boundary operators.
We obtain the following result.

\begin{lemma}
 \label{2.13}
 Let $A=A(x,D)$ be an operator of the form \eqref{2-6} with coefficients $a_\alpha\colon\R^n\to\C$, and assume that
 $\lambda-A$ is parameter-elliptic in $\Lambda$. Let $s,\sigma\in\R$, and assume  \textnormal{(S1)} and \textnormal{(S2)} to hold.
 Then, there exists a $\lambda_0>0$ such that for every $\lambda\in\Lambda$ with $|\lambda|\ge\lambda_0$, the operator
 \[ \lambda-A \colon \E^{s,\sigma}(\R^n) \to \F^{s,\sigma}(\R^n)\]
 is an isomorphism.
 \end{lemma}

%%%%%%%%%%%%%%%%%%%%%%%%%%%%%%%%%%%%%%%%%%%%%%%%%
%%%%%%%%%%%%%%%%%%%%%%%%%%%%%%%%%%%%%%%%%%%%%%%%%

\section{Boundary value problems in domains}
\label{sec4}

We now consider \eqref{2-1} in a bounded or exterior domain. Throughout this section, we assume
$\Omega$ to be a domain with compact boundary $\Gamma$, and $(\lambda-A,B)$ to be a boundary value problem
which is parameter-elliptic in some sector $\Lambda\subset\C$. Moreover, we assume
 (S1)--(S3) and (S5) to hold.

We define $C^\infty(\overline\Omega)$ as the restriction of all $u\in
C_0^\infty(\R^n)$ to $\Omega$. As the definition of the spaces $H^{s,\sigma}_{p,\lambda}$
is non-canonical in domains, we will only consider standard Sobolev spaces
on $\Omega$. For the construction of the solution operators, we will use
local coordinates where the  space  $H_{p,\lambda}^{s,\sigma}(\R^n_+)$ is available. %\rot{Mit $\lambda$?}

We start with some remarks concerning the localization technique: Let $ x_{0}\in \Gamma $. Since the domain $ \Omega $ has a $ C^{2m+\lceil r'\rceil} $-boundary, there is an open set $ \wt U_{x_{0}} $ containing~$ x_{0} $, a radius $ r_{x_{0}}> 0$ and a $ C^{2m+\lceil r'\rceil} $-diffeomorphism $ \diffeo_{x_{0}}\colon \wt V_{x_{0}} \to \wt U_{x_{0}} $, where $ \wt V_{x_{0}} = B(0, 2r_{x_{0}})  $, such that $ \diffeo_{x_{0}}(\wt V_{x_{0}}\cap \R^{n}_{+})= \wt U_{x_{0}}\cap \Omega $ and $ \diffeo_{x_{0}}(0)=x_{0} $. We set $ \V_{x_{0}}\coloneqq B(0,r_{x_{0}}) $ and $U_{x_{0}}\coloneqq \diffeo_{x_{0}}(\V_{x_{0}})  $.
By compactness of $ \Gamma $, there are $ x_{1}, \ldots, x_{K_{0}}\in \Gamma $ and open sets $ U_{x_{1}}, \ldots, U_{x_{K_{0}}} $ as above such that
$  \Gamma \subset \bigcup_{k=1}^{K_{0}} U_{x_{k}} $.
%\rot{$\Gamma \subset U_{x_0}\cup\ldots\cup U_{x_{K_0}} $ (Joerg)}.
For the sake of simplicity, we shall use $ k $ instead of $ x_{k} $ as index.

 We proceed similarly as in the half-space case. Hence, we define
\begin{align*}
\delta_{\max}\coloneqq \sup\bigg\{\delta>0 \,\Big|\, \{x \in \Omega ~|~ \operatorname{dist}(x,\Gamma)\leq \delta\} \subset
%\rot{U_1\cup\ldots\cup U_{K_0} (Joerg)} \Big\}.
\bigcup\limits_{k=1}^{K_0} U_k \bigg\}.
\end{align*}
If $\Omega$ is bounded,
$\Omega \,\backslash \bigcup\nolimits_{k=1}^{K_0} U_{k}$
%\rot{$\Omega \,\backslash (U_1\cup\ldots\cup U_{K_0})$}
is compact, and we can choose $ x_{k} $ in $ \Omega $ and $ 0<r_{k}<\frac{\delta_{\max}}{2} $  such that
\begin{align}\label{3-3a}
 U_{k}\coloneqq B(x_{k},r_{k})\subset \Big\{x\in\Omega: \operatorname{dist}(x,\Gamma)>\frac{\delta_{\max}}{2}\Big\}
\end{align}
for $ k=K_{0}+1,\ldots, K $ and
$ \overline{\Omega}\subset \bigcup_{k=1}^{K}U_{k} $.
%\rot{$ \overline{\Omega}\subset U_{1}\cup\ldots\cup U_K $ (Joerg)}.

In the case of an exterior domain, this construction has to be slightly modified. We first define $U_{K_0+1} := \R^n\setminus\overline {B(0,r_{K_0+1})}$, where the radius $r_{K_0+1}$ is chosen such that
$\R^n\setminus\Omega\subset B(0,\frac{r_{K_0+1}}{2})$. Now
%\rot{$\Omega\setminus(U_{1}\cup\ldots\cup U_{K_0+1}) (Joerg)$} is compact, and
$\Omega\setminus\bigcup_{k=1}^{K_0+1}U_k$ is compact, and
we choose $x_k$ and~$r_k$ with \eqref{3-3a} for $k=K_0+2,\dots,K$ such that
again
%\rot{$\overline{\Omega}\subset U_1\cup\ldots\cup U_K$. (Joerg)}}
$\overline{\Omega}\subset \bigcup_{k=1}^{K}U_{k} $.

For formal reasons, %\rot{Klar, dass in beiden Fällen?},
we define $\V_k \coloneqq U_k$ and
$\diffeo_k\coloneqq \id_{\V_k} $ for $k=K_0+1,\dots,K$.

\begin{remark}[Local operators and extensions]
	Let $ x_{1}, \ldots, x_{K} $ be chosen as above. For $k\in\{1,\dots, K_0\}$, we define the local operator $\wt A^k$ as the pullback
	of the operator~$A$ by $\diffeo_k$. More precisely, for $\vhs \in C^\infty(\wt \V_k)$, we write
	\[  (\wt A^{k}\vhs)(\y) \coloneqq A(\vhs\circ \diffeo_{k}^{-1})(\diffeo_{k}(\y)) =: \sum_{|\alpha|\le 2m}
		\wt a_\alpha^{ k}(\y) D^\alpha \vhs(\y)\quad (\y
		\in \wt\V_{k}\cap \overline{\R^n_+}).  \]
The explicit description of the coefficients $\wt a_\alpha^{ k}$ (Fa\`{a} di Bruno-formula, see \cite{Fraenkel78}, Formula~B)  shows that
$\wt a_\alpha^{ k}$ contains the function $a_\alpha \circ \diffeo_k$ as well as derivatives of $\diffeo_k ^{-1}$ up to order
$2m+1-|\alpha|$ for $|\alpha|\ge 1$ (and no derivative for $|\alpha|=0$), concatenated with~$\vartheta_k$. Hence we always need at most $2m$ derivatives of $\vartheta_k^{-1}$, which ensure $\wt a_\alpha^{ k} \in \BUC^{\lceil r' \rceil}$. For $|\alpha|=2m$ at most one derivative of $\vartheta_k^{-1}$ appears and as $m \in \N$ we have $2m+\lceil r' \rceil-1 \geq  \lfloor r' \rfloor +1$, which shows $\wt a_\alpha^{ k} \in \BUC^{r}$ for $|\alpha|=2m$.  Consequently, condition (S5) implies that (S1) also holds
for $\wt a_\alpha^{ k}$. In the same way, we define the local operator $ \wt B^k=(\wt B_1^k,\ldots, \wt B_m^k) $ via
\[ 		(\wt B_j^k\vhs)(\y)   \coloneqq B_j(\vhs\circ \diffeo_{k}^{-1})(\diffeo_{k}(\y))=:
		\sum_{|\beta|\le m_j}
		\wt b_{j\beta}^{k}(\y) \gamma_0 D^\beta\vhs (\y) \quad
		(\y
		\in \wt\V_{k}\cap\R^{n-1}).
	\]
	A simple calculation shows that $2m+ \lceil r'\rceil \geq m_j+\lfloor k_j'\rfloor +1 =  m_j+ k_j$ and thus (S5) also implies that the transformed operators $\wt B_j^k$ satisfy (S3) for all  $|\beta| \leq m_j$. For $k\in\{K_0+1,\dots,K\}$, we set
	$\wt a_\alpha^{ k}(\y):= a_\alpha(\y)$ for $\y\in B(x_k,2r_k)$ with some obvious modifications in the case of an exterior domain for $k=K_0+1$.
	
	Again with the general extension function from Remark~\ref{rem_ext}, we extend the coefficients $ \wt a_\alpha^{ k} $ and $ \wt b_{j\beta}^{k} $
	to $ \R^{n}_{+} $, $ \R^{n} $ and $ \R^{n-1} $, respectively.
% 	We fix a function
% 	$\chi\in C^\infty([0,\infty))$ with $0\le \chi\le 1$,
% 	$\chi(z)=1$ for $0\le z\le 1$ and $\chi(z)=0$ for $z\ge 2$. If $\Omega$ is bounded,
	We set
	\begin{alignat*}{5}
	\acf_\alpha^{k}( \y) &\coloneqq \wt  a_\alpha^{ k} ( \chi_{V_k}
	(\y))&& ( \y\in
	\overline{\R^n_+})&& \text{ for } k=1,\dots,K_0,\\
	\acf_\alpha^{k}( \y) &\coloneqq  \wt a_\alpha^{ k}(  \chi_{V_k}(\y))&& ( \y\in
	{\R^n})&& \text{ for } k=K_0+1,\dots,K ,\\
	\bcf_{j\beta}^{k}( \y) &\coloneqq \wt  b_{j\beta}^{ k}( \chi_{V_k}
	(\y))&\quad& ( \y\in
	{\R^{n-1}})&& \text{ for } k=1,\dots,K_0.
	\end{alignat*}
	%\grun{Angepasst mit Bemerkung~\ref{rem_ext}.}
% 	\grun{In the case of an exterior domain, we replace the definition of $a_\alpha^k$ for $k=K_0+1$ by
% 	\[ \acf_\alpha^{K_0+1}( \y)  \coloneqq  \wt a_\alpha^{ K_0+1}\Big(   \frac{\y}{|\y|^2} + \chi\Big(\frac{r_{K_0+1}}{|\y|} \Big)\Big( \y - \frac{\y}{|\y|^2}\Big) \Big)\quad  ( \y\in
% 	{\R^n}).\]}
	 Finally, we define
	\begin{alignat*}{5}
		\A^{k}\vhs(\y) &\coloneqq \sum_{|\alpha|\le 2m}
		\acf_\alpha^{ k}(\y) D^\alpha \vhs(\y)&& (\y
		\in \overline{\R^{n}_{+}} ) && \textnormal{ for } k=1,\ldots, K_{0}, \\
		\A^{k}\vhs(\y) &\coloneqq \sum_{|\alpha|\le 2m}
		\acf_\alpha^{ k}(\y) D^\alpha \vhs(\y)&& (\y
		\in  \R^{n}) && \textnormal{ for } k=K_{0}+1,\ldots, K, \\
		\B_j^k\vhs(\y) & \coloneqq \sum_{|\beta|\le m_j}
		\bcf_{j\beta}^{ k}(\y) \gamma_0 D^\beta\vhs (\y) &\quad&
		(\y
		\in  \R^{n-1}) && \textnormal{ for } k=1,\ldots, K_{0}.
	\end{alignat*}
\end{remark}

The extended local operators $A^k$ and $B^k$ satisfy the above smoothness and ellipticity assumptions, so we can apply the results from
Section~\ref{sec3}. However, as we do not have the spaces $H^{s,\sigma}_{p,\lambda}$ %\rot{Mit $\lambda$}
in domains, we use the standard
Sobolev spaces as in Corollary~\ref{2.11}.

Therefore, we additionally  consider the spaces
\begin{align}
  \EE^{s,\sigma}(\Omega) & := H_{p,\lambda}^{s+\sigma}(\Omega),\label{DefEE}\\
  \FF^{s,\sigma}(\Omega) & := H_{p,\lambda}^{s-2m}(\Omega) \times \prod\limits_{j=1}^m B_{pp,
    \lambda}^{s+\sigma-m_j-1/p}(\Gamma) \label{DefFF}
\end{align}
and the analog spaces with $\Omega$ being replaced by $\R^n_+$. We also set $\EE^{s,\sigma}(\R^n) := H_{p,\lambda}^{s+\sigma}(\R^n)$ and $\FF^{s,\sigma}(\R^n):= H_{p,\lambda}^{s-2m}(\R^n)$. %\rot{$\FF^{s,\sigma}(\R^n)$ hängt nicht von $\sigma$ ab.}
Note that for $\sigma \leq 0$ we have the continuous embeddings
\begin{align}\label{embeddings_EF}
    \FF^{s,\sigma}\subset \F^{s,\sigma} \quad \textnormal{and} \quad \E^{s,\sigma}\subset \EE^{s,\sigma}.
\end{align}
% $\FF^{s,\sigma}\subset \F^{s,\sigma}$ and $\E^{s,\sigma}\subset \EE^{s,\sigma}$ (cf. \eqref{DefE} and \eqref{DefF}).

\begin{lemma}\label{3.8}
Let $s,\sigma \in\rz$ with $s>\max_j m_j+\frac{1}{p}$,
and let $\A^k, \B^k$
denote the extended local operators.
Then there exists a $\lambda_0>0$ such that the operator families
\begin{equation}\label{3-2E}
\begin{alignedat}{4}
 \begin{pmatrix}\lambda-\A^k\\ \B^k\end{pmatrix}&\colon
 \E^{s,\sigma}(\R^n_+) \to\F^{s,\sigma}(\R^n_+)&\quad & (k=1,\dots,K_0),\\
 \lambda-\A^k& \colon
 \E^{s,\sigma}(\R^n )\to \F^{s,\sigma}(\R^n)&& (k=K_0+1,\dots,K)
\end{alignedat}
\end{equation}
for  $\lambda\in\Lambda$ with $|\lambda|\ge \lambda_0$ are isomorphisms.
We denote the inverse operator by $\Lhs_k(\lambda)$. For $\sigma\leq 0$,
the restrictions of $ \Lhs_k(\lambda)$ to $\FF^{s,\sigma}(\R^n_+)$ and $\FF^{s,\sigma}(\R^n)$, respectively,
yield bounded operator families
\begin{equation}\label{3-2EE}
\begin{alignedat}{4}
 \Lhs_k(\lambda) & \in L ( \FF^{s,\sigma}(\R^n_+),\EE^{s,\sigma}(\R^n_+))
  &\quad & (k=1,\dots,K_0),\\
 \Lhs_k(\lambda) & \in L ( \FF^{s,\sigma}(\R^n ),\EE^{s,\sigma}(\R^n ))&& (k=K_0+1,\dots,K).
\end{alignedat}
\end{equation}
\end{lemma}

\begin{proof}
We have seen above that $A^k, B^k$ satisfy conditions (S1) and (S3). Conditions  (S2)  and  (S4)  follow directly from the fact that the extended coefficients are constant far away from the origin by construction. Hence the statement follows for $k\in\{1,\dots,K_0\}$ from
Theorem~\ref{2.10} and Corollary~\ref{2.11} and for $k\in\{K_0+1,\dots,K\}$ from Lemma~\ref{2.13} and the embeddings~\eqref{embeddings_EF}.
% (applied with $(s,\sigma)$ for \eqref{3-2E} and with $(s,0)$ for \eqref{3-2EE}). \rot{Wendet man das wirklich auf $(s,0)$ an?}
\end{proof}

To solve \eqref{2-1} in  $\Omega$, we first construct an approximate solution operator $R_0(\lambda)$, using the local
solution operators $L_k(\lambda)$ from Lemma~\ref{3.8} and the local coordinate maps~$\diffeo_k$ for $k=1,\dots,K$. Setting
   $\Theta_k \vhs := \vhs\circ\diffeo_k^{-1}$,  the $C^{2m+\lceil r' \rceil }$-diffeomorphism~$\diffeo_k$  induces  isomorphisms
\begin{equation}\label{3-10}
\begin{alignedat}{6}
  \Theta_k  & \colon   H^s_{p,\lambda}(\V_k\cap\R^n_+)  &&\to H^s_{p,\lambda}(U_k\cap\Omega) &\quad & (k=1,\dots,K_0),\\
  \Theta_k  & \colon   H^s_{p,\lambda}(\V_k)  &&\to H^s_{p,\lambda}(U_k) && (k=K_0+1,\dots,K)
\end{alignedat}
\end{equation}
for  $s\in [0,2m+\lceil r' \rceil]$. Since we have $\Theta_k(\dot H_{p,\lambda}^{s}(V_k\cap \R^n_+)) = \dot H_{p,\lambda}^s(U_k \cap \Omega) $, we even get~\eqref{3-10} for all $|s|\leq 2m+ \lceil r' \rceil $ via duality.
%\rot{Warum ist das so? Braucht man dazu etwas in der Art $|\det D\vartheta_k^{-1}|=1 $ oder Ähnliches? Wie bekommt man das $\lambda$-abhängig? Bekommt man das auch für negative $s$?} By complex interpolation, this  holds for all $s\in[0,\blau{\lceil r' \rceil}+2m]$.
Moreover, by the definition of the Besov space on the closed $C^{2m+\lceil r' \rceil}$-manifold
 $\Gamma$, the restriction $\diffeo_k|_{\V_k\cap\R^{n-1}}\colon
  \V_k\cap\R^{n-1}\to U_k\cap \Gamma$ also induces isomorphisms
 \[\Theta_k \colon B_{pp,\lambda}^s(\V_k\cap \R^{n-1}) \to B_{pp,\lambda}^s (U_k\cap \Gamma)\]
 for $k=1,\dots,K_0$ and all $|s|\leq 2m +\lceil r' \rceil$.
 We fix a smooth partition of unity
$\phiom_k\in C^\infty(\R^n)$, $k=1,\dots,K$, with
   $\supp\phiom_k\subset U_k$,  $0\le\phiom_k\le 1$, and
  $\sum_{k=1}^K \phiom_k = 1$
%\rot{$\varphi_1^\Omega+\ldots+\varphi_K^\Omega=1$ (Joerg)}
on $\overline\Omega$.
   Additionally, let $\psiom_k\in C^\infty(\R^n)$
  with $0\le \psiom_k\le 1$,  $\supp\psiom_k\subset U_k$ and $\psiom_k=1$ on $\supp\phiom_k$. We set $\psihs_k := \Theta_k^{-1} \psiom_k= \psiom_k\circ\diffeo_k,$
  % \begin{align*}
  %     \phihs_k & := \Theta_k^{-1} \phiom_k = \phiom_k\circ\diffeo_k, \rot{\text{ Taucht nie auf.}}\\
  %     \psihs_k & := \Theta_k^{-1} \psiom_k= \psiom_k\circ\diffeo_k,
  % \end{align*}
  where here and in the following, we  identify functions with compact support and their
trivial extensions for sake of readability. Without this identification, we have, e.g.,
$\psihs_k = e^0_{V_k}\Theta_k^{-1}(r_{U_k}\psiom_k)$ for $k=K_0 +1,\dots,K$, %\rot{Stimmt glaube ich nur für $k=K_0+1,\ldots, K$}
where again $r_{U_k}$ stands for the
restriction to $U_k$ and $e^0_{V_k}$ for the trivial extension to $\R^n$ by zero.

In the following let $\lambda_0>0$ be given as in Lemma \ref{3.8}. The approximate solution operator $R_0(\lambda)$ is now for $\lambda\in\Lambda$, $|\lambda|\ge\lambda_0$
formally defined as
\begin{equation}\label{3-12}
 R_0(\lambda)\binom{f}{g} := \sum_{k=1}^{K_0} \Lom_k(\lambda)\binom{\phiom_k f}{(\gamma_0\phiom_k) g}
  + \sum_{k=K_0+1}^{K} \Lom_k(\lambda)(\phiom_k f).
\end{equation}
  Here, $\Lom_k(\lambda)$ is defined by
  \begin{equation}\label{3-8}
  \begin{alignedat}{4}
       \Lom_k(\lambda)\binom{f}{g} & :=  \Theta_k \psihs_k \Lhs_k(\lambda)  \Theta_k^{-1}\binom{f}{g}
 &\quad& (k=1,\dots,K_0),\\
 \Lom_k(\lambda)f  & := \Theta_k \psihs_k \Lhs_k(\lambda)  \Theta_k^{-1}f
  && (k=K_0+1,\dots, K).
  \end{alignedat}
  \end{equation}

\begin{lemma}
  \label{3.9}
Let $s,\sigma\in\R$ with $s>\max_j m_j+\frac1p$ and $\sigma\le 0$. Then the operator $R_0(\lambda)$ in \eqref{3-12}
is well-defined on $\FF^{s,\sigma}(\Omega)$ for $\lambda\in\Lambda$ with $|\lambda|\ge \lambda_0$
and yields a bounded operator family
  \begin{equation}\label{3-9}
   R_0(\lambda)\in L( \FF^{s,\sigma}(\Omega), \EE^{s,\sigma}(\Omega)).
   \end{equation}
\end{lemma}

\begin{proof}
% We note that for $k=1,\dots,K_0$
% the following operators are all
% well-defined and continuous:
% \begin{alignat*}{6}
%   \FF^{s,\sigma}(\Omega) &\to \FF^{s,\sigma}(\Omega), & \binom f g &\mapsto
%   \binom{\phiom_k f}{(\gamma_0\phiom_k)g},&&\\
%   \FF^{s,\sigma}(\Omega) &\to \FF^{s,\sigma}(\R^n_+),\quad &   \binom{f }{g } &\mapsto
%   \binom{e_{\R^n_+}^0 \Theta_k^{-1} r_{U_k\cap\Omega}f}{e_{\R^{n-1}}^0 \Theta_k^{-1} r_{U_k\cap\Gamma}
%   g}\quad && \text{ by \eqref{3-10}},\\
%     \FF^{s,\sigma}(\R^n_+)&\to \EE^{s,\sigma}(\R^n_+),  & \binom{f}
%     {g}&\mapsto \Lhs_k(\lambda) \binom{f}
%     {g}&& \text{ by Lemma~\ref{3.8},}\\
%       \EE^{s,\sigma}(\R^n_+) &\to \EE^{s,\sigma}(\R^n_+) ,\quad& \vhs  &\mapsto
%     \psihs_k \vhs ,&&\\
%     \EE^{s,\sigma}(\R^n_+) &\to \EE^{s,\sigma} (\Omega),&  \vhs & \mapsto
%     e_\Omega^0\Theta_k r_{V_k\cap\R^n_+} \vhs  &&
%     \text{ by \eqref{3-10}}.
% \end{alignat*}
% \rot{Die Notation mit der Fortsetzung passt so nicht, da hier als Index steht, worauf es fortgesetzt wird und nicht, von wo es fortgesetzt wird.} This shows that the first sum in \eqref{3-12}  is well-defined  on $\FF^{s,\sigma}(\Omega)$ and continuous.
% Exactly the same arguments work
%  in the simpler situation $k\in\{K_0+1,\dots,K\}$, so we obtain that $R_0(\lambda)$ is well
%  defined and $R_0(\lambda)\in L( \FF^{s,\sigma}(\Omega), \EE^{s,\sigma}(\Omega))$.

%  \rot{Gekürzte Variante:}
 The continuity of $R_0(\lambda)$ in the corresponding spaces follows from \eqref{3-10} and Lemma~\ref{3.8}.
\end{proof}

In the following, we will modify $R_0(\lambda)$ to get a solution. For this, we compute $(\lambda-A,B)R_0(\lambda)(f,g)$,
where we may choose $(f,g)$ sufficiently smooth such that the classical theory can be applied.
% More precisely, we
% choose $(f,g)\in\FF^{s',0}(\Omega)$ with $s'\ge \max\{2m,s\}$, so we are in the setting of classical Sobolev spaces, that allow for classical traces. \rot{Das mit der klassischen Theorie steht hier doppelt, man könnte diesen Satz streichen.}
Therefore, we introduce $s'$, and assume from now on that $s,\sigma,s'\in\R$ satisfy
\begin{equation}
  \label{3-7}
  s > \max_{j=1,\dots,m} m_j + \tfrac 1p ,\quad -1 < \sigma\le 0,\quad s'\geq \max\{2m,s\}.
\end{equation}
Moreover, we assume (S1)--(S3) and (S5) %\rot{Braucht man hier (S2)?}
for $(s,\sigma)$ (as before) and also for $(s',0)$.
The conditions with respect to $(s',0)$ collapse to $r'=s'-2m$ and $k_j'=s'-m_j-\frac{1}{p}$. In the end we take the maximum, respectively.

In contrast to the half-space situation, we have a restriction on $\sigma$ in \eqref{3-7}. This is essentially due to
the commutator estimates and the fact that we only consider standard Sobolev spaces in $\Omega$.
% We write $[\cdot,\cdot]$ for the commutator and use the notation $\psihs_k$
% also for the multiplication operator with $\psihs_k$. For the boundary operators,
% the commutator $[\B^k,\psihs_k]$ is defined by
% \[[\B^k,\psihs_k]\vhs = \B^k (\psihs_k \vhs)
% - (\gamma_0\psihs_k) \B^k  \vhs.\]
\begin{lemma}
  \label{3.10}
Let $s,\sigma$ and $s'$ satisfy \eqref{3-7}. Let $0<\eps <\min\{1+\sigma, \frac1p\}$.
\begin{enumerate}[a)]
\item
%a)   Assume \eqref{3-7}, and let $k\in\{1,\dots, K\}$.
For $\lambda\in\Lambda$ with $|\lambda|\ge\lambda_0$, define
  the operator $C_k(\lambda)$ by
  \begin{alignat*}{4}
  C_k(\lambda)  & := \binom{
    -[\A^k,\psihs_k]}{[\B^k,\psihs_k]}
    \Lhs _k(\lambda)&\quad& (k=1,\dots,K_0),\\
    C_k(\lambda)  & := -[\A^k,\psihs_k]\Lhs _k(\lambda)
    && (k=K_0+1,\dots, K).
  \end{alignat*}
  Then
  \begin{alignat*}{4}
  C_k(\lambda)\in L(\FF ^{s,\sigma}(\R^n_+),\FF ^{s+\eps,\sigma}
  (\R^n_+)) &\quad& (k=1,\dots,K_0),\\
  C_k(\lambda)\in L(\FF ^{s,\sigma}(\R^n),\FF ^{s+\eps,\sigma}
  (\R^n))  && (k=K_0+1,\dots, K).
  \end{alignat*}
  \item Let $(f,g)\in \FF ^{s',0}(\Omega)$, and set $v := R_0(\lambda)(f,g)$
with $R_0(\lambda)$ being defined in \eqref{3-12}. Then
  \begin{equation}
    \label{3-11}
    \binom {\lambda -A}{B} v = (1+C(\lambda))\binom f g
  \end{equation}
  holds with an operator $C(\lambda)\in L(\FF ^{s,\sigma}(\Omega),
  \FF ^{s+\eps,\sigma} (\Omega))$
%  for all $(s,\sigma)$ satisfying \eqref{3-7}, where $\bar s := \sigma+1 >0$.
  and there exists a $\lambda_1\geq\lambda_0$ such that
  $1+C(\lambda)\in L (\FF ^{s,\sigma}(\Omega))$ is invertible
    for all $\lambda\in\Lambda$ with $|\lambda|\ge\lambda_1$.
 \end{enumerate}
\end{lemma}

\begin{proof}
a) The operator $[\A^k,\psihs_k]$ is a differential operator of order not greater
than $2m-1$.
%As $\sigma>-1$, we have $\sigma+1<\eps$. \rot{Passt so nicht.}
Therefore, the mapping
\[ [\A^k,\psihs_k]\colon
\EE^{s,\sigma}(\R^n_+)=H_{p,\lambda}^{s+\sigma}(\R^n_+)\to H_{p,\lambda}^{s-2m+1+\sigma}(\R^n_+)\subset
H_{p,\lambda}^{s+\eps-2m}(\R^n_+)\]
is bounded. %\rot{Das letzte $=$ sollte ein $\subset$ sein und gilt fuer $\eps\le \sigma+1$? (Joerg)}

% For the boundary operators, we have (cf. the proof of Lemma \ref{2.8})
% \begin{align*}
%   \B^k_j(\psihs_k \vhs)
%  & =\sum_{|\beta|\le m_j}
%    b_{j\beta}^k \gamma_0 D^\beta(\psihs_k\vhs) \\
% & = \sum_{|\beta|\le m_j}
%   b_{j\beta}^k \gamma_0 \Big(\psihs_k D^\beta \vhs +
% \sum_{|\gamma|\le m_j-1}
% C_{j,\beta,\gamma}(\cdot) D^\gamma \vhs \Big) \\
%  & = (\gamma_0\psihs_k) B_j^k\vhs + \blau{\sum_{|\beta|\leq m_j}\sum_{|\gamma| \leq m_{j}-1}}
%     b_{j\beta}^k(\gamma_0 c_{j,\beta,\gamma})
%  \gamma_0 D^\gamma\vhs,
% \end{align*}
% for $j=1,\dots,m$,
% where the coefficients $c_{j,\beta,\gamma}$ depend on $\psihs_k$. Therefore,
% the operator $[\B^k,\psihs_k]$ is a boundary operator of order
% not greater than $m_j-1$.
In analogy to the proof of Lemma \ref{2.8} we obtain that the operator $[\B^k_j,\psihs_k]$ is a boundary operator of order not greater than $m_j-1$.
In the case $m_j=0$, this operator is zero.
As $\EE^{s,\sigma}(\R^n_+)= H_{p,\lambda}^{s+\sigma}(\R^n_+)$
and $\sigma>-1$, the boundary operator $[\B_j^k,\psihs_k]$ is
defined in the classical sense and is continuous as an operator
\[  [\B_j^k,\psihs_k]\colon \EE ^{s,\sigma}(\R^n_+)\to
 B_{pp,\lambda}^{s+\sigma-m_j+1-1/p}(\R^{n-1})\subset B_{pp,\lambda}^{s+\eps+\sigma-m_j-1/p}(\R^{n-1}).\]
By Lemma~\ref{3.8}, we have
$\Lhs _k(\lambda)\in L(\FF ^{s,\sigma}(\R^n_+),
\EE ^{s,\sigma}(\R^n_+))$. This and the above mapping properties for the commutators
show $C_k(\lambda)\in L(\FF ^{s,\sigma}(\R^n_+),\FF ^{s+\eps,\sigma}
  (\R^n_+))$.

b) We first remark that $v\in \EE ^{s',0}(\Omega)$ holds due to Lemma~\ref{3.9}, and, as $s'\geq 2m$,
the boundary operators $B_j$ can be applied to $v$ in the classical sense. Using calculations similar to the ones in the proof of Lemma \ref{2.8} and the equality $ \Theta_k\psi_k\Theta_k^{-1}=\psiom_k$, we obtain

\[\binom {\lambda -A}{B} v = (1+C(\lambda))\binom f g\] with
\[ C(\lambda)\binom f g := \sum_{k=1}^{K_0} \Theta_k C_k(\lambda) \Theta_k^{-1}
  \binom{\phiom_k f}{(\gamma_0\phiom_k) g} + \sum_{k=K_0+1}^K
 \binom{  \Theta_k C_k(\lambda) \Theta_k^{-1}
(\phiom_k f)}{0}.\]

From a) and the fact that multiplication by $\phiom_k$ and the coordinate transformations
$\Theta_k$, $\Theta_k^{-1}$ preserve the smoothness as $\eps<\frac1p$, we obtain
\begin{align*}
    C(\lambda)\in L(\FF ^{s,\sigma}(\Omega),\FF ^{s+\eps,\sigma}
  (\Omega)).
\end{align*}
Proceeding as in the proof of Lemma \ref{2.3}~a), and using the Neumann series as in Lemma~\ref{2.4}, we obtain that for sufficiently large $\lambda$, the
  operator $1+C(\lambda) \in
  L(\FF^{s,\sigma}(\Omega))$ is invertible, and the norm of the inverse is not
  greater than $2$.
\end{proof}

The following theorem is the key result of this section and gives an a priori-estimate for the solution operator of \eqref{2-1} in spaces of rough regularity. Note that we first consider smooth functions, where the boundary operators
are defined in a classical way and where we know unique solvability by classical results. However, the a priori-estimate gives a continuous extension of the solution operator
to larger spaces.

\begin{theorem}
  \label{3.11}
  Let $(\lambda-A,B)$ be parameter-elliptic in the sector $\Lambda$, and let $s,\sigma,s'\in\R$ satisfy \eqref{3-7}.
  Assume \textnormal{(S1)--(S3), (S5)} to hold for $(s,\sigma)$ and $(s',0)$. Then taking $\lambda_1 \geq \lambda_0$ as in Lemma \ref{3.10}~b), for all
  $\lambda\in\Lambda$ with $|\lambda|\ge\lambda_1$, and every
    $(f,g)\in \FF ^{s',0}(\Omega)$,
  the unique solution
  $u\in\EE ^{s',0}(\Omega)$ of \eqref{2-1} is given by
  \[ u = R(\lambda)\binom f g := R_0(\lambda)(1+C(\lambda))^{-1}\binom f g,\]
and we have the a priori-estimate
  \begin{equation}\label{3-13}
   \Big\| R(\lambda)\binom f g\Big\|_{\EE ^{s,\sigma}(\Omega)}\le C
   \Big\|\binom f g\Big\|_{\FF ^{s,\sigma}(\Omega)}
   \end{equation}
   with a constant $C$ not depending on $f,g$ or $\lambda$. In particular, the solution operator $R(\lambda)$ extends uniquely
   to a continuous operator family
   \[ R(\lambda) \in L( \FF ^{s,\sigma}(\Omega), \EE ^{s,\sigma}(\Omega)).\]
\end{theorem}

\begin{proof}
First, we remark that $\FF ^{s',0}(\Omega)$ and $\EE ^{s',0}(\Omega)$ are
classical spaces, and therefore we obtain unique solvability in these spaces
(see, e.g., \cite{Agranovich-Faierman-Denk97}, Theorem~2.1). By Lemma~\ref{3.10}
b) with $(s,\sigma)=(s',0)$, the operator $1+C(\lambda)$ is invertible in
$L(\FF ^{s',0}(\Omega))$,
and from Lemma~\ref{3.9} we get $R_0(\lambda) \in L (\FF ^{s',0}(\Omega),
\EE ^{s',0}(\Omega))$ because of $s'\geq s$. Therefore,
$u:= R(\lambda)\binom fg\in \EE ^{s',0}(\Omega)$. As
\[ \binom{\lambda-A}B u = (1+C(\lambda))(1+C(\lambda))^{-1}\binom f g = \binom f g\]
by Lemma~\ref{3.10} b), $u$ is the unique solution of \eqref{2-1}. Finally,
the a priori-estimate \eqref{3-13} follows from Lemma~\ref{3.9} and  $(1+C(\lambda))^{-1}
\in L(\FF ^{s,\sigma}(\Omega))$.
\end{proof}
The existence of continuous solution operators given by Theorem~\ref{3.11} is the main part of the analysis
of \eqref{2-1}. To formulate the uniqueness of the solution, we have to consider a function space over $\Omega$
where the
boundary operators are well-defined. For this, we apply the theory of Roitberg (\cite{Roitberg96}, \cite{Roitberg99}), which leads to the space $H^s_{p,A,s_0}(\Omega)$ as defined below. Note that for the construction of the solution operator, the results by Roitberg were not used. We still assume \eqref{3-7} to hold.

% ==========================

% Neuer Versuch (Robert): Wir haben eh schon im ganzen Abschnitt \eqref{3-7} vorausgesetzt, es wird also alles auf
% diese Situation spezialisiert, ebenso auf $s_0\ge s-2m$, was wir in Theorem~\ref{3.12} auch voraussetzen.

% =========================

\begin{definition} \label{3.1.1}
Let $s_0\in\R$ with $s_0\ge s-2m$. Then we define $H^s_{p,A,s_0}(\Omega)$ as the completion
of $C^\infty(\overline\Omega)$ with respect
  to the norm
  \[ \|u\|_{H^s_{p,A,s_0}(\Omega)} := \|u\|_{H_p^s(\Omega)}  + \|Au\|_{H^{s_0}_p(\Omega)}.\]
\end{definition}

\begin{remark}
  \label{3.2.1}
By the continuity of $A\colon H^{s_0+2m}_p(\Omega)\to H^{s_0}_p(\Omega)$ and the condition on~$s_0$, we find that
 \[\|u\|_{H^s_p(\Omega)}\le \|u\|_{H^s_{p,A,s_0}(\Omega)}
   \le C\|u\|_{H^{s_0+2m}_p(\Omega)}\]
for every $u\in C^\infty(\overline\Omega)$. It follows that
 \[H_p^{s_0+2m}(\Omega)\subset H^s_{p,A,s_0}(\Omega)
   \subset H^s_p(\Omega)\]
with dense embeddings.
\end{remark}

\begin{lemma}
 \label{3.3.1}
 Let $s_0\in\R$ with $s_0>-1+\frac1p$ and $s_0\ge s-2m$. Then
 \begin{equation}\label{3-6.2}
 \binom AB \colon H^{s+\sigma}_{p,A,s_0}(\Omega)\to H^{s_0}_p(\Omega) \times
  \prod_{j=1}^m  B_{pp}^{s+\sigma-m_j-1/p}(\Gamma)
 \end{equation}
is well-defined and continuous.
\end{lemma}

\begin{proof}
  For smooth $u$, we write $B_j u$ in the form
    \begin{equation}\label{3-6.1}
    B_j u   = \gamma_0
  \sum_{|\beta|\le m_j} b_{j\beta}(\cdot) D^\beta u  = \sum_{l=0}^{m_j}
  M_{jl}(x,D')  \gamma_l u ,
  \end{equation}
  where $\gamma_l\colon u\mapsto (\partial_\nu^l u)|_{\Gamma}$ is the
  classical trace and  $M_{jl}(x,D')$ is a differential operator of order
  not greater than $m_j-l$ which contains only derivatives in tangential direction.
We first show that  for all $u\in C^\infty(\overline\Omega)$  and $l=0,\dots, \max_{j} m_j$ we have
  \begin{equation}\label{3-5.1}
   \|\gamma_l u\|_{B_{pp}^{s+\sigma-l-1/p}(\Gamma)} \le C \|u\|_{H^{s+\sigma}_{p,A,s_0}(\Omega)}.
  \end{equation}
  Indeed, if $s+\sigma>\max_j m_j +\frac1p$, this follows from classical trace results, where we can even replace
  the norm on the right-hand side of \eqref{3-5.1} by $\|u\|_{H_p^{s+\sigma}(\Omega)}$.
  If $s+\sigma\le \max_j m_j +\frac1 p$, we first note that we have $s+\sigma>-1+\frac1p$ by \eqref{3-7}.
  Therefore, we can apply  \cite{Roitberg96}, Theorem~6.1.1 and (6.1.29) (see also the text after
  \cite{Roitberg96}, Definition~6.2.1) and obtain
  \begin{equation}\label{3-4.1}
   \|\gamma_l u\|_{B_{pp}^{s+\sigma-l-1/p}(\Gamma)} \le C \big( \|u\|_{H^{s+\sigma}_p(\Omega)}
  + \|Au\|_{\dot H^{s+\sigma-2m}_{p }(\overline\Omega) }\big).
  \end{equation}
% \forget{
% As $s+\sigma-2m\le s-2m\le s_0$, we see that \blau{$\dot H^{s_0}_{p }(\overline\Omega) \subset \dot H_p^{s+\sigma-2m}(\overline\Omega)$}. %\rot{Andersrum?!}
% Moreover, the condition $s_0>-1+\frac1p$ implies $\dot H_p^{s_0}(\overline\Omega)\subset H_p^{s_0}(\Omega)$ (with
% equality for $s_0\in (-1+\frac1p, \frac1p)$). Therefore, we may replace $ \|Au\|_{\dot H^{s+\sigma-2m}_{p }(\overline\Omega) }$
% in \eqref{3-4.1} by $\|Au\|_{H_p^{s_0}(\Omega)}$ which yields \eqref{3-5.1}. \rot{Braucht man hierfür nicht die Gleichheit? Wenn ja, warum gilt die?}
% }
Now choose $0<\eps<1$ such that  $$s+\sigma-2m\le \max_j m_j +\frac1 p-2m< -1+\frac1p+\eps<s_0.$$
Then
 $$H^{s_0}_p(\Omega)\subset H^{-1+\frac1p+\eps}_p(\Omega)=
   \dot H^{-1+\frac1p+\eps}_p(\overline\Omega)\subset \dot H^{s+\sigma-2m}_{p }(\overline\Omega), $$
where the equality can be found, e.g., in \cite{Triebel78}, Theorem~4.3.2/1.
Hence we may replace $ \|Au\|_{\dot H^{s+\sigma-2m}_{p }(\overline\Omega) }$
in \eqref{3-4.1} by $\|Au\|_{H_p^{s_0}(\Omega)}$ which yields \eqref{3-5.1}.

From the continuity of $M_{jl}(x,D')\colon B_{pp}^{s+\sigma-l-1/p}(\Gamma)\to B_{pp}^{s+\sigma-m_j-1/p}(\Gamma)$, the estimate in
\eqref{3-5.1}, and the definition of the norm in $H_{p,A,s_0}^{s+\sigma}(\Omega)$ we obtain
\[ \|Au\|_{H_p^{s_0}(\Omega)} + \sum_{j=1}^m \|B_j u \|_{B_{pp}^{s+\sigma-m_j-1/p}(\Gamma)} \le C \|u\|_{H_{p,A,s_0}^{s+\sigma}(\Omega)}\]
for all $u\in C^\infty(\overline\Omega)$. As $C^\infty(\overline\Omega) $ is dense in $H_{p,A,s_0}^{s+\sigma}(\Omega)$,
the operator \eqref{3-6.2} is well-defined by unique  extension and continuous.
\end{proof}

Now we are able to formulate our main result. At first we will cover the general
situation and then, as a corollary, the simpler setting $f\in L^p(\Omega)$.

\begin{theorem}
  \label{3.12}
  Let $(\lambda-A,B)$ be parameter-elliptic in the sector $\Lambda$, and let $s,\sigma\in\R$ with
   $s>\max_j m_j+\frac1p$ and $\sigma\in (-1,0]$.
   Let $-1+\frac1p < s_0 \leq s+\sigma$, and $s_0 \geq s-2m$. We fix $s':=\max\{2m,s_0+2m\}$.  Assume conditions \textnormal{(S1)--(S3)}, \textnormal{(S5)} to
   hold with respect to  $(s,\sigma)$ as well as $(s',0)$. Then there exists a $\lambda_1>0$ such that
  for all $\lambda\in\Lambda$ with $|\lambda|\ge \lambda_1$ and all
  \begin{equation}\label{3-14}
  (f,g)\in H_p^{s_0}(\Omega) \times  \prod\limits_{j=1}^m B_{pp}^{s+\sigma-m_j-1/p}(\Gamma),
  \end{equation}
  the boundary value problem \eqref{2-1} has a unique solution $u\in H^{s+\sigma}_{p,A,s_0}(\Omega)$. This
  solution is given by $u=R(\lambda)\binom fg$ and
  satisfies the a priori-estimate \eqref{3-13}.
\end{theorem}

\begin{proof}
By the assumptions on $s$ and $\sigma$,  \eqref{3-7} holds, and for sufficiently large $|\lambda|$,
we can define $u:= R(\lambda)\binom fg\in \EE ^{s,\sigma}(\Omega)$.  By Theorem~\ref{3.11},
$u$ satisfies the a priori-estimate \eqref{3-13}. For the rest of the proof let $\lambda$ be arbitrary but fixed.

We want to show that  $u\in H^{s+\sigma}_{p,A,s_0}(\Omega)$. For this we denote the space in \eqref{3-14} by
$\FF^{s,\sigma,s_0}(\Omega)$ and first note that
$\FF^{s',0}(\Omega)$ is dense in the space $\FF^{s,\sigma,s_0}(\Omega)$, as even   smooth functions
are dense. Therefore, we can
 choose a sequence
$\binom{f_k}{g_k}\in \FF^{s',0}(\Omega)$ with $\| \binom{f_k}{g_k}-\binom fg\|_{\FF^{s,\sigma,s_0}(\Omega)}\to 0$ for $k\to\infty$ and set $u_k := R(\lambda)\binom{f_k}{g_k}$.
By Theorem~\ref{3.11} and Remark~\ref{3.2.1}, we know $u_k\in \EE ^{s',0}(\Omega) = H_p^{s'}(\Omega)\subset
H^{s+\sigma}_{p,A,s_0}(\Omega)$. Moreover, $Bu_k$ is defined
in the classical sense, and we have $\binom{\lambda-A}{B} u_k =
\binom{f_k}{g_k}$ by Theorem~\ref{3.11}. In particular, $Au_k = \lambda u_k - f_k$.
 This  and \eqref{3-13} yield
\begin{align*}
 \|u_k-&u_\ell\|_{H^{s+\sigma}_{p,A,s_0}(\Omega)}= \|u_k-u_\ell\|_{H^{s+\sigma}_{p}(\Omega)}
 + \|Au_k - Au_\ell\|_{H_p^{s_0}(\Omega)} \\
& \le C_\lambda \Big\| \binom{f_k}{g_k}-\binom {f_\ell}{g_\ell}\Big\|_{\FF^{s,\sigma}(\Omega)} +\|f_k-f_\ell \|_{H_p^{s_0}(\Omega)} + |\lambda|\|u_k-u_\ell\|_{H_p^{s_0}(\Omega)}\\
& \le C_\lambda \Big\| \binom{f_k}{g_k}-\binom {f_\ell}{g_\ell}\Big\|_{\FF^{s,\sigma,s_0}(\Omega)} + |\lambda|\,\|u_k-u_\ell\|_{H_p^{s+\sigma}(\Omega)}
 \to 0 \quad (k,\ell\to\infty).
\end{align*}
Recall that here we have used parameter-independent norms and the a priori-estimate
\eqref{3-13}  for fixed
$\lambda$. We also use the condition $s_0 \leq s+\sigma$. We have seen that $(u_k)_{k\in\N}$ is a Cauchy sequence in
$H^{s+\sigma}_{p,A,s_0}(\Omega)$ and therefore convergent to some element $v\in
H^{s+\sigma}_{p,A,s_0}(\Omega)$. By \eqref{3-13}, we see $u_k\to u$ in $H_p^{s+\sigma}(\Omega)$,
and therefore $u=v\in H^{s+\sigma}_{p,A,s_0}(\Omega)$.

As $u\in H^{s+\sigma}_{p,A,s_0}(\Omega)$ and $s_0> -1+\frac1p$, the expression $Bu$ is well-defined in the sense of
Lemma~\ref{3.3.1}, which also yields the continuity of the operator in the respective spaces. Hence, the above approximation shows that $\binom{\lambda-A}B u = \binom fg$,
therefore $u$ is a solution of \eqref{2-1}.
For the uniqueness we make use of Lemma~\ref{2.9} once more, where we take
$E=H^{s+\sigma}_{p,A,s_0}(\Omega)$, $F=\FF^{s,\sigma,s_0}(\Omega)$, $E_0=\EE^{s',0}(\Omega)$, and $T=\binom{\lambda-A}{B}.$
%Finally, to show uniqueness, let $u'\in H^{s+\sigma}_{p,A,s_1}(\Omega)$ for any $s_1\leq s_0$ satisfying , $s_1>-1+\frac1p, s_1>s-2m$ with $(\lambda - A)u'=f$
%and $Bu'=g$. Let $(u_k)_{k\in\N}\subset C^\infty(\overline\Omega)$ with $u_k\to u'$
%in $H^{s+\sigma}_{p,A,s_1}(\Omega)$. Then $\binom{f_k}{g_k}:=\binom{(\lambda-A)u_k}{Bu_k}\in
%\mathbb F^{(s',0)}(\Omega)$, and the unique solvability of \eqref{2-1} with $(s,\sigma)=(s',0)$
%yields $u_k = R(\lambda)\binom{f_k}{g_k}$. On the other hand, by the definition of the
%norm in $H^{s+\sigma}_{p,A,s_1}(\Omega)$ and the continuity of the trace due to Lemma \ref{3.3}
%(here we use $s_1 > -1+\frac1p$) we obtain
%$\binom{f_k}{g_k}\to \binom f g$ in $\tilde{\mathbb F}^{(s,\sigma,s_1)}(\Omega)\subset \mathbb{F}^{(s,\sigma)}$. Applying the solution
%operator $R(\lambda)$, we obtain
%\[ u_k = R(\lambda)\binom{f_k}{g_k} \to u:= R(\lambda)\binom fg\quad\text{ in }\mathbb
%E^{(s,\sigma)}(\Omega) = H_p^{s+\sigma}(\Omega).\]
%As $u_k\to u'$ in $H^{s+\sigma}_{p,A,s_1}(\Omega)$, this implies $u'=u$, which finishes the proof.
\end{proof}

The conditions on the parameters $s,\sigma$ and on the smoothness are much simpler in the case $f\in L^p(\Omega)$. We obtain the following corollary, which shows that  boundary spaces of order close to $-1$ may appear.

\begin{corollary}
 \label{3.13}
Let $(\lambda-A,B)$ be parameter-elliptic in the sector $\Lambda$, and let  $\tau\in\R$ with $\max_j m_j +\frac1p-1 < \tau \le 2m$
and $\tau\ge 0$ (the last condition is automatically satisfied except for $2m=2$ and $m_1 =0$).
Assume \textnormal{(S1)} and \textnormal{(S3)} to hold for  $r':=2m-\tau$ and $k_j':=2m - m_j-\frac1p$.  Let  $\Omega$ be of class $C^{2m+\lceil r' \rceil}$, and assume \textnormal{(S2)} if $\Omega$ is unbounded.

% \blau{Assume
% \begin{alignat*}{4}
%     a_\alpha & \in {\BUC}^r(\Omega)&\quad& \text{ with }r:=
%     \begin{cases} \lfloor |2m-\tau| \rfloor+1 &\text{ for } |\alpha|=2m \\ \lceil |2m-\tau| \rceil &\text{ else}. \end{cases} \\
%     b_{j\beta}&\in {\BUC}^{r_j}(\Gamma)&& \text{ with }r_j:=\begin{cases}\lfloor |\tau+m_j-\frac1p |'\rfloor+1 &\text{ for } |\beta|=m_j \\ \lceil |\tau-m_j-\frac1p| \rceil & \text{ else}. \end{cases}
% \end{alignat*}}
% \rot{Soll es jeweils $\tau-m_j-1/p$ in der Definition von $r_j$ sein? Die Betraege in $r$ und $r_j$ kann man weglassen? Heißt sonst $k_j$, nicht $r_j$, nicht dasselbe $r$ bzw. $ k_j$ wie sonst.} Let $\Omega$ be of class $C^{2m+\lceil 2m-\tau \rceil}$, and assume \textnormal{(S2)} if $\Omega$ is unbounded.

Then there exists a $\lambda_1>0$ such that
  for all $\lambda\in\Lambda$ with $|\lambda|\ge \lambda_1$ and all
  \[
  (f,g)\in L^p(\Omega) \times \prod\limits_{j=1}^m B_{pp}^{\tau-m_j-1/p}(\Gamma),
  \]
  the boundary
  value problem \eqref{2-1} has a unique solution $u\in H^{\tau}_{p,A,0}(\Omega)$, which satisfies the a priori-estimate \eqref{3-13}.
\end{corollary}

\begin{proof}
  We apply Theorem~\ref{3.12} with $s_0=0$.
  % and set $\tau:=s+\sigma$ \blau{such that $\max_j m_j+\frac{1}{p}<s\leq 2m.$} \rot{Warum kann man $s,\sigma$ so wählen, dass das erfüllt ist?} Then the conditions on $\tau$ and the definitions of $r$ and $r_j$ follow by simple calculations from the conditions in Theorem~\ref{3.12} and (S1), (S3).
  If $\tau >\max_j m_j + \frac1p$, we choose $s:=\tau$ and $\sigma := 0$. In the case $\tau \leq \max_j m_j + \frac1p$, we set $s := \max_j m_j + \frac1p + \eps$ and $\sigma := \tau - s$ for $\eps>0$ sufficiently small. Note that for this choice of $(s,\sigma)$ the conditions in Theorem~\ref{3.12} are fulfilled.
\end{proof}

\pagebreak[2]
\section{Boundary value problems with dynamic boundary conditions}
\label{sec5}

As an application of the above results, we consider a boundary value
problem with dynamic boundary conditions, which is related to the linearized Cahn--Hilliard equation
and was discussed in detail in \cite{Pruess-Racke-Zheng06}.
We show that the corresponding operator
generates a holomorphic semigroup in $L^p$. For simplicity, we restrict ourselves to the
model problem situation and do not consider a general domain.
The related model problem
in the half-space has the form (see \cite{Pruess-Racke-Zheng06}, Equation (2.1))
\begin{alignat*}{4}
  (\partial_t+\Delta^2) u & = f && \text{ in } (0,\infty)\times \R^n_+,\\
  \partial_t u + \partial_\nu u -\Delta' u & = g && \text{ on } (0,\infty)\times\R^{n-1},\\
  \partial_\nu \Delta u & = 0 && \text{ on }(0,\infty)\times\R^{n-1}
\end{alignat*}
(plus initial condition), where $\Delta'$ stands for the tangential Laplacian.
Here we have set the constants to $1$ and omitted lower-order terms.
Following a standard approach for boundary value problems with dynamical boundary conditions, we decouple $u=:u_1$ and
$\gamma_0 u=:u_2$ and consider the Cauchy problem in a product space, where now the compatibility condition $u_1 = \gamma_0 u_2$ has to be added.
The corresponding resolvent problem is given by
\begin{equation}\label{4-4}
\begin{alignedat}{4}
  (\lambda+\Delta^2) u_1 & = f && \text{ in } \R^n_+,\\
  \partial_\nu u_1 + (\lambda-\Delta') u_2 & = g && \text{ on }\R^{n-1},\\
  \partial_\nu \Delta u_1 & = 0 && \text{ on }\R^{n-1},\\
  \gamma_0 u_1 - u_2 & = 0 && \text{ on }\R^{n-1}.
\end{alignedat}
\end{equation}
It was shown in \cite{Pruess-Racke-Zheng06}, Remark~2.2, that the related operator generates an analytic $C_0$-semigroup
in the basic space $L^p(\R^n_+)\times B_{pp}^r(\R^{n-1})$ for all $r\in [2-\frac 1p, 3-\frac 1p]$.
In the present paper, we consider the basic space $X:= L^p(\R^n_+)\times L^p(\R^{n-1})$.
  In order to define a suitable operator $A$ representing \eqref{4-4}, we have to verify the parameter-ellipticity of the auxiliary problem below.

  In the following, let $\arg(\cdot)$ denote the argument of a complex number with values in $(-\pi,\pi]$. Furthermore let $\sqrt{\cdot}$ denote the principal branch of the complex
  square root which is holomorphic in $\C\setminus(-\infty,0]$ and for which we
  have $\Re\sqrt z>0$ for all $z\in\C\setminus(-\infty,0]$.
\begin{lemma}
  \label{4.2}
  Let $\theta\in (0,\pi)$. Then the boundary value problem $(\lambda+\Delta^2, \gamma_0,
  \partial_\nu\Delta)$ in $\R^n_+$ is parameter-elliptic in the sector $\Lambda:=\overline{\Sigma_\theta}$, where
  \begin{align*}
      \Sigma_\theta:= \{ z\in\C\backslash\{0\}: |\arg z|<\theta\}.
  \end{align*}
\end{lemma}
\begin{proof}
  Obviously, the operator $\lambda+\Delta^2$ is parameter-elliptic in $\Lambda$. To see that the
  Shapiro--Lopatinskii condition holds, we first assume $\lambda\in\Lambda\setminus\{0\}$. Then
  every stable solution of the ODE
  \begin{equation}\label{4-5}
   \big[ \lambda + (\partial_n^2 - |\xi'|^2)^2 \big] w(x_n)  = 0\quad (x_n>0)
  \end{equation}
  is of the form $w(x_n) = c_1 e^{-\tau_1 x_n} + c_2 e^{-\tau_2 x_n}$ with the roots
  $\tau_{1,2} = \tau_{1,2}(|\xi'|,\lambda)$, where
  \begin{equation}\label{4-6}
   \tau_{1,2}(|\xi'|,\lambda) :=  \sqrt{ |\xi'|^2 \pm i\sqrt{\lambda}}.
  \end{equation}
   % However, due to the term $-\lambda$, we choose another branch for  the inner square root in order to obtain a function which is holomorphic in $\Lambda_\theta$: In \eqref{4-6}, we define $\sqrt{\ldots}'$ as the branch of the complex root which maps $z=r e^{i\phi}$ with $\phi\in (0,2\pi)$ to $\sqrt{z}' = \sqrt r e^{i\phi/2}$.Note that this function is holomorphic on $\C\setminus [0,\infty)$ and that we have $\Im \sqrt{z}' >0$ for all $z\in\C\setminus[0,\infty)$.
  Note that we have $\Re\tau_j(|\xi'|,\lambda)>0$ for $j=1,2$ and all $\xi'\in\R^{n-1}$.

  The first boundary condition $w(0)=0$ yields $c_1=-c_2$, and from the second boundary condition we
  obtain
  \begin{equation}\label{4-7}
  \begin{aligned}
    0 &= -\partial_n (\partial_n^2- |\xi'|^2) w(0) = \tau_1 (\tau_1^2-|\xi'|^2) c_1
    +\tau_2 (\tau_2^2-|\xi'|^2) c_2 \\
    & = i\sqrt{\lambda} (\tau_1 c_1 - \tau_2 c_2).
  \end{aligned}
  \end{equation}
  Therefore, $0 = \tau_1 c_1 -\tau_2 c_2 = (\tau_1+\tau_2) c_1$. As
  $\Re(\tau_1+\tau_2)>0$, we obtain $c_1=c_2=0$.

  In the case $\lambda=0$, every stable solution of \eqref{4-5} is of the form $w(x_n) = (c_1+c_2x_n)
  e^{-|\xi'|x_n} $. From $w(0)=0$ we obtain $c_1=0$, and
  \[ 0  = -\partial_n (\partial_n^2- |\xi'|^2) w(0) = - 2c_2 |\xi'|^2 \]
  shows $c_1=c_2=0$ for every $\xi'\in\R^{n-1}\setminus\{0\}$.
\end{proof}
\begin{definition}\label{DefA}
Now we are able to define the operator $A\colon X\supset D(A)\to X$ with $X:=L^p(\R^n_+)\times L^p(\R^{n-1})$ by
\begin{align*}
     A(D) u & := \begin{pmatrix}
-\Delta^2 & 0 \\
-\partial_\nu & \Delta'
\end{pmatrix}\quad (u\in D(A)),
\end{align*}
where
\begin{align*}
  D(A) & := \{ u=(u_1,u_2)\in X: \; A(D)u\in X,\; \partial_\nu \Delta u_1 =0,\;
\gamma_0 u_1 -u_2 =0 \textnormal{ on $\R^{n-1}$}\}.
\end{align*}
%\rot{Definition von $X$ nochmal wiederholen?}
\end{definition}
\begin{remark}
  \label{4.1}
  a) Due to the parameter-ellipticity of the boundary value problem in Lemma~\ref{4.2} we may use Theorem \ref{2.2} to solve the system
  \begin{equation}\label{4-4'}
\begin{alignedat}{4}
  (\lambda+\Delta^2) u_1 & = f && \text{ in } \R^n_+,\\
  \gamma_0 u_1 & = u_2 && \text{ on }\R^{n-1},\\
   \partial_\nu \Delta u_1 & = 0 && \text{ on }\R^{n-1}.
  \end{alignedat}
\end{equation} Now, the existence of all traces is clear, as we are in the half-space situation and have the spaces $H^{s,\sigma}_p(\R^n_+)$ at our disposal. Using the embeddings $L^p(\R^n_+) \subset H^{0,-4}_p(\R^n_+)$ and $L^p(\R^{n-1})\subset B_{pp}^{-1/p}(\R^{n-1})$ we obtain $u_1 \in H^{4,-4}_p(\R^n_+),$ which shows that $A$ is well-defined,
  % Taking large  $u\in D(A)$ implies $\blau{u_1}\in H^0_{p,\Delta^2,0}(\R^n_+)$ (see Definition~\ref{3.1.1}) \rot{Formal wurde der Raum da nicht definiert, da in Kapitel 4 generell kompakter Rand vorausgesetzt wird ...},
  % and therefore the boundary operators are well-defined due to Lemma~\ref{3.2.1} and Remark~\ref{3.3.1}.
 as $\partial_\nu u_1 \in B_{pp}^{-1-1/p}(\R^{n-1})$ and $\partial_\nu\Delta u_1
  \in B_{pp}^{-3-1/p}(\R^{n-1})$.
  % \rot{Wendet man hier nicht eher Kapitel 3 an? Dort gibt es die Spuren doch immer und man hat hier $ \sigma =-4+1/(2p)$?}

b) Next, we observe that the operator $A$ is densely defined. For this, let $(f,g)\in X = L^p(\R^n_+)\times L^p(\R^{n-1})$,
and let $\epsilon>0$.
We first choose $\phi_2\in C_0^\infty(\R^{n-1})$ with $\|\phi_2-g\|_{L^p(\R^{n-1})}<\epsilon$ and then define $\phi_1\in H_p^4(\R^n_+)$ as
the unique solution of
\begin{alignat*}{4}
(1+\Delta^2) \phi_1 & = 0 &&\text{ in }\R^n_+,\\
\gamma_0 \phi_1 & = \phi_2 &&\text{ on }\R^{n-1},\\
\partial_\nu\Delta \phi_1 & = 0 &&\text{ on }\R^{n-1}.
\end{alignat*}
By definition,
we obtain $(\phi_1,\phi_2)\in D(A)$. In a second step, we choose $\phi_1'\in C_0^\infty(\R^n_+)$ with $\|\phi_1'+\phi_1-f\|_{L^p(\R^n_+)}
< \epsilon$. Then $(\phi_1',0)\in D(A)$, which implies that $u := (\phi_1+\phi_1',\phi_2)\in D(A)$. By construction, we know
$\|u-(f,g)\|_X<2\epsilon$.

  c) Finally, the operator $A$ is closed. To see this, let $(u^k)_{k\in\N}\subset D(A)$ be a sequence with
  $u^k=(u^k_1,u^k_2)\to u=(u_1,u_2)$ in $X$ and $Au^k\to v=(v_1,v_2)$ in $X$.
  Then we have $\Delta^2 u_1^k\to \Delta^2 u_1$ in $H_p^{-4}(\R^n_+)$ due to the continuity
  of the operator $\Delta^2\colon L^p(\R^n_+)\to H_p^{-4}(\R^n_+)$ as well as $-\Delta^2 u_1^k
  \to v_1$ in $L^p(\R^n_+)$ and therefore also in $H_p^{-4}(\R^n_+)$. By uniqueness of the limit,
  we see that $-\Delta^2 u_1 = v_1\in L^p(\R^n_+)$. Similarly, using the spaces from a), one shows $-\partial_\nu u_1
  + \Delta' u_2=v_2 \in L^p(\R^{n-1})$ and $\partial_\nu\Delta u_1 = \gamma_0 u_1-u_2=0$. Therefore,
  $u\in D(A)$ and $Au=v$.
\end{remark}
Now we want to show that the operator $A$ generates a holomorphic semigroup in $L^p(\R^n_+)\times L^p(\R^{n-1})$.
The key step in the proof consists in the analysis of the solution operator of \eqref{4-4} with $f=0$ and $\lambda \in \Sigma_\theta$.
% \grun{Ab hier doppelt mit Thm 5.6 ii) (nicht mehr)}
For this, we take the partial Fourier transform $(\mathscr F' u_1)(\xi',x_n)
 =: w(\xi',x_n)=:w(x_n)$ and obtain the ODE \eqref{4-5} as well
 as \eqref{4-7} from the boundary condition $\partial_\nu \Delta u_1=0$. From the proof
 of Lemma~\ref{4.2} we know that $w(x_n)=c_1e^{-\tau_1 x_n} + c_2e^{-\tau_2 x_n}$
 and $\tau_1c_1=\tau_2c_2$, where $\tau_{1,2} =\tau_{1,2}(|\xi'|,\lambda)$ are defined in
 \eqref{4-6}. Inserting this into the second line of \eqref{4-4}, we get
 \[ \tau_1 c_1 + \tau_2 c_2 + (\lambda+|\xi'|^2) (c_1+c_2) = \hat g := \mathscr F'g.\]
 With $c_2 = \frac{\tau_1}{\tau_2} c_1$, this yields
 \[ c_1 = \frac{\tau_2}{(\tau_1+\tau_2)(\lambda+|\xi'|^2)+2\tau_1\tau_2}\; \hat g .\]
 For $\hat u_2 (\xi')= (\mathscr F' u_1)(\xi',0)$, we obtain
 \begin{equation} \label{4-1'}
 \hat u_2 = c_1+c_2 = \frac{\tau_1+\tau_2}{(\tau_1+\tau_2)(\lambda+|\xi'|^2)+2\tau_1\tau_2}
 \hat g =:
 S(|\xi'|,\lambda) \hat g.
 \end{equation}
 Therefore, we have to analyze the symbol $S(|\xi'|,\lambda)$. As we will use the
 bounded $H^\infty$-calculus, we will extend this symbol with respect to the first
 variable to a small sector $\Sigma_\epsilon$.
 % \rot{Wahrscheinlich sollte man dann schreiben, dass $S$ nur von $|\xi'|$ abhängt, sonst unklar, was mit der zweiten Variable bzw. dem Sektor gemeint ist. Außerdem würde ich $\lambda$ und $\xi'$ tauschen, da es bei den anderen auftauchenden Funktionen immer andersrum ist.}
 We start with a technical result
 on the zeros $\tau_1$ and $\tau_2$.
\begin{lemma}\label{4.3}
Let $\theta\in (\frac\pi2,\pi)$ and $\epsilon \in (0,\frac{\pi-\theta}4)$. %\rot{enumerate Umgebung? Uneinheitlich.}
\begin{enumerate}[a)]
    \item Let $\tau_{1,2}\colon \Sigma_\eps\times \Sigma_\theta\to\C$
be defined by
 \[ \tau_{1,2}(z,\lambda) := \sqrt{z^2\pm i\sqrt{\lambda}}\quad ((z,\lambda) \in
 \Sigma_\epsilon\times \Sigma_\theta).\]
Then $\tau_{1,2} $ is holomorphic in $\Sigma_\eps\times \Sigma_\theta$ and satisfies
\begin{alignat}{4}
    C (|z|+|\lambda|^{1/4}) & \le  |\tau_j(z,\lambda)| && \le C' (|z|+|\lambda|^{1/4})\quad (j=1,2), \label{4-13}\\
    C (|z|+|\lambda|^{1/4}) & \le |\tau_1(z,\lambda)+\tau_2(z,\lambda)| && \le C' (|z|+|\lambda|^{1/4})\label{4-14}
\end{alignat} for suitable constants $C,C'>0$ and all
$(z,\lambda) \in  \Sigma_\epsilon\times \Sigma_\theta$.
\item For all $(z,\lambda) \in  \Sigma_\epsilon\times \Sigma_\theta$
%with $\arg\lambda\in (\frac\pi2,\theta)$, \rot{Passt so nicht, da man unten auch $\arg\lambda\in (-\theta, -\frac\pi2)$ betrachtet. Wahrscheinlich sollte hier $\Re \lambda <0$ stehen.}
we have
\[
\arg\Big(\frac{\tau_1(z,\lambda)\tau_2(z,\lambda)}{(\tau_1(z,\lambda)+\tau_2(z,\lambda))}\Big)   \in
\begin{cases} (-\epsilon,\tfrac{\theta+\pi}4)
&\text{if } \arg\lambda\in (\frac\pi2,\theta),\\
(-\frac{3\pi}{8},\frac{3\pi}{8})
&\text{if } \arg\lambda\in [-\frac\pi2,\frac\pi2],\\
(-\tfrac{\theta+\pi}4, \epsilon)
&\text{if } \arg\lambda\in (-\theta, -\frac\pi2).
\end{cases}
\]

\end{enumerate}

\end{lemma}

\begin{proof}
  a) As $\pm i\sqrt\lambda\in\Sigma_{(\theta+\pi)/2}$, the condition on $\epsilon$ implies $z^2\pm i\sqrt\lambda\in \Sigma_{(\theta+\pi)/2}$.
  This shows that $ \tau_j$ is well-defined and holomorphic in $\Sigma_\epsilon\times\Sigma_\theta$ with
  values in $\Sigma_{(\theta+\pi)/4}$. The function
  $\phi(z,\lambda) := |\tau_j(z,\lambda)| (|z|+|\lambda|^{1/4})^{-1}$ is smooth and quasi-homogeneous of degree $0$ in the sense that
  \[ \phi(\rho z,\rho^4\lambda) = \phi(z,\lambda)\quad (\rho>0,\, z\in\Sigma_\epsilon,\,\lambda\in\Sigma_\theta).\]
  Therefore, its minimum and maximum are attained on the compact set
  \[ M := \big\{ (z,\lambda)\in \overline{\Sigma_\epsilon}\times \overline{\Sigma_\theta} : |z|+|\lambda|^{1/4} = 1\big\}  \]
  (here we note that $\tau_j$ can be extended continuously to $M$). As $\tau_j\not=0$ for all $(z,\lambda)\in M$, we
  obtain $0< C\le \phi(z,\lambda) \le C'<\infty$, which yields \eqref{4-13}.

  Because of   $\tau_j\in \Sigma_{(\theta+\pi)/4}$, there exists a constant $C_\theta>0$ with
  $\Re \tau_j\ge C_\theta |\tau_j|$. Consequently,
  \[ |\tau_1+\tau_2|\ge \Re (\tau_1+\tau_2) \ge \Re\tau_1 \ge C_\theta |\tau_1| \ge C C_\theta (|z|+|\lambda|^{1/4}) .\]
  As the other inequality in \eqref{4-14} is obvious, we obtain a).

b) For $\Re\lambda\geq 0$ we have $\tau_1,\tau_2 \in \Sigma_{3\pi/8}$. Consequently, the same holds for $\frac{\tau_1\tau_2}{\tau_1+\tau_2}=\frac{1}{\tau_1^{-1}+\tau_2^{-1}}$.
Now, let $\arg\lambda\in(\frac\pi2,\theta)$. Analogously, we get $\frac{\tau_1\tau_2}{\tau_1+\tau_2}\in \Sigma_{(\theta+\pi)/4}$.
To see $\arg\left(\frac{\tau_1\tau_2}{\tau_1+\tau_2}\right)>-\epsilon$, it is sufficient to prove $\Im \left(\frac{|\tau_1+\tau_2|^2\tau_1\tau_2}{z|z|^2(\tau_1+\tau_2)} \right)\geq 0$. We set $c:=\sqrt{\lambda/z^4}$ and obtain
\begin{align*}
 \frac{|\tau_1+\tau_2|^2\tau_1\tau_2}{z|z|^2(\tau_1+\tau_2)}  =  \frac{|\tau_1|^2\tau_2 + |\tau_2|^2\tau_1}{z|z|^2}=
     |1+ic|\sqrt{1-ic}+|1-ic|\sqrt{1+ic}.
 \end{align*}
By the condition on $\lambda$ and $z$, we know that $c=a+ib$ for some $a,b> 0$.
In a first step, we show
\begin{align}\label{Im_sum_tau}
    \Im ( \sqrt{1+ic} + \sqrt{1-ic} )\geq 0.
\end{align}
Using the formula %$\left(\sqrt{\frac{|x\pm iy|+x}{2}} \pm i \sqrt{\frac{|x\pm iy|-x}{2}}\right)^{2}=x\pm iy$
$\Im\sqrt{x\pm iy} = \pm \sqrt{\frac{|x\pm iy|-x}{2}} $ for all $x\in\R$ and $y>0$, we get
\begin{align*}
    \Im ( \sqrt{1+ic} + \sqrt{1-ic})
   = \sqrt{\frac{|1-b+ia|-(1-b)}{2}} - \sqrt{\frac{|1+b-ia|-(1+b)}{2}},
\end{align*}
such that \eqref{Im_sum_tau} is equivalent to
\begin{align*}
    \sqrt{(1+b)^2+a^2}- \sqrt{(1-b)^2+a^2}\leq 2b,
\end{align*}
which holds by the reverse triangle inequality in $\R^2$ applied to the points $(1+b,a)$ and $(1-b,a)$.
With the inequalities $|1-ic|\geq |1+ic|$ and $\Im \sqrt{1+ic}\geq 0$ as well as~\eqref{Im_sum_tau}, we finally get
\begin{align*}
    \Im (|1+ic|\sqrt{1-ic}+|1-ic|\sqrt{1+ic}) \geq |1+ic| \Im ( \sqrt{1-ic} + \sqrt{1+ic})\geq 0.
\end{align*}
Consequently, the statement in b) holds for $\arg\lambda\in (\frac\pi2,\theta)$. The statement for
$\arg\lambda\in (-\theta,-\frac\pi2)$ follows from $\overline{\tau_{1,2}(z,\lambda)} = \tau_{2,1}(\overline z,\overline\lambda)$.
\end{proof}

\begin{lemma}\label{4.4}
 Let $\theta\in (\frac\pi2,\pi)$ and $\epsilon \in (0,\frac{\pi-\theta}4)$.
 For  $\lambda\in \Sigma_\theta$ and $z\in\Sigma_\epsilon$, define
 \begin{equation}\label{4-10}
  m(z,\lambda) := (\lambda+z^2)S(z,\lambda)=\frac{(\lambda+z^2)(\tau_1+\tau_2)}{(\lambda+z^2)(\tau_1+\tau_2)+2 \tau_1
 \tau_2}.
 \end{equation}
Then $m\colon \Sigma_\epsilon\times\Sigma_\theta\to \C$ is holomorphic and bounded.
% \rot{Erwähnen, dass $\lambda S = m $ gilt, damit man weiß, warum dieses $m$ untersucht wird. Wollen wir nicht auch zeigen, dass $(\lambda+z^2) S(z,\lambda)$ beschränkt ist, dann bekommen wir statt einer Lösung in $L^p$ eine in $H^2$.}
\end{lemma}
\begin{proof}
   For the boundedness we notice that
  \begin{align*}
      m(z,\lambda)=\frac{1}{1+ \frac{2\tau_1\tau_2}{(\lambda+ z^2)(\tau_1+\tau_2)}  }
  \end{align*}
  and show $\frac{\tau_1\tau_2}{(\lambda+ z^2)(\tau_1+\tau_2)}  \in \Sigma_{\varphi}$ for some $\phi\in(0,\pi)$.
  In the case $\Re \lambda \geq 0$ we have  $  \frac{\tau_1\tau_2}{\tau_1 + \tau_2}  \in \Sigma_{3\pi/8}$ due to Lemma~\ref{4.3}~b), which yields
  \begin{align*}
      \frac{\tau_1\tau_2}{(\lambda+ z^2)(\tau_1+\tau_2)} \in \Sigma_{7\pi/8}.
  \end{align*}
   If $\arg \lambda \in (\frac{\pi}{2},\theta)$ we use again Lemma~\ref{4.3} b) and obtain $\arg \left(\frac{\tau_1\tau_2}{\tau_1 + \tau_2} \right)\in(-\eps, \frac{\theta + \pi}{4}) $. By the condition on $\lambda$ and $z$, we get $\arg(\lambda+z^2)^{-1} \in (-\theta,2\eps)$ and therefore
  \begin{align*}
      \frac{\tau_1\tau_2}{(\lambda+ z^2)(\tau_1+\tau_2)} \in \Sigma_{\theta+\eps}.
  \end{align*}
  For $\arg \lambda \in (-\theta, -\frac{\pi}{2})$ we argue in the same way to see $\frac{\tau_1\tau_2}{(\lambda+ z^2)(\tau_1+\tau_2)} \in \Sigma_{\theta+\eps}$.
  Obviously, $m$ is holomorphic in $\Sigma_\epsilon\times
  \Sigma_\theta$.
\end{proof}

The last two lemmas allow us to prove the main result of this section. We recall that the operator $A$ is described in Definition \ref{DefA}.

\begin{theorem}
  \label{4.5}
  For every $\lambda_0>0$, the operator $A-\lambda_0$ generates a bounded holomorphic $C_0$-semigroup
  of  angle $\frac\pi 2$ in $X:= L^p(\R^n_+)\times L^p(\R^{n-1})$.
  In particular, $A$ generates a holomorphic $C_0$-semigroup of angle $\frac\pi 2$ in $X$.
 Furthermore we obtain $H^2$-regularity of the solution. More precisely for any $\eps >0$ we have
  \[ D(A) \subset H^{2+1/p-\eps}_p(\R^n_+)\times H^2_p(\R^{n-1}). \]
  We may choose $\eps=0$ if $p\geq 2.$
\end{theorem}

\begin{proof}
  Let $\theta\in (\frac\pi2,\pi)$, and let $\lambda_0>0$. Then there is some $\lambda_0'>0$ with
\begin{equation}\label{4-19}
\lambda_0+\Sigma_\theta \subset \{ \lambda\in\Sigma_\theta: |\lambda|\ge \lambda_0'\}.
\end{equation}
  We  show that for every $\lambda\in\Sigma_\theta$ with $|\lambda|\ge \lambda_0'$ and every
  $(f,g)\in X$, equation~\eqref{4-4} has a unique solution
  $u=(u_1,u_2)\in D(A)$ and $|\lambda|\,\|u\|_X \le C \|(f,g)\|_X$ with a constant
  not depending on $\lambda$.

  (i) Let $(f,g)\in X$. We construct the unique solution $u=(u_1,u_2)\in D(A)$ of equation~\eqref{4-4} by solving two different boundary value problems. First, we consider the boundary value problem
  \begin{equation}
    \label{4-8}
\begin{alignedat}{4}
  (\lambda+\Delta^2) u_1^0 & = f && \text{ in } \R^n_+,\\
  \gamma_0 u_1^0 & = 0 && \text{ on }\R^{n-1},\\
  \partial_\nu \Delta u_1^0 & = 0 && \text{ on }\R^{n-1}.
\end{alignedat}
  \end{equation}
  By Lemma~\ref{4.2}, this problem is parameter-elliptic, and by classical results
  (see \cite{Grubb95}, Theorem~1.9, or apply Theorem~\ref{2.2} with $s=4$ and $\sigma=0$),
   there exists a unique solution $u_1^0\in H_p^4(\R^n_+)$ of \eqref{4-8}
  and
  \[ |\lambda| \|u_1^0\|_{L^p(\R^n_+)} \le C \|f\|_{L^p(\R^n_+)} ,\]
  where the constant $C$ depends on $\theta$ and $\lambda_0'$ but not on $f$ or $\lambda$.

(ii)  Next, we solve
  \begin{equation}
    \label{4-9}
\begin{alignedat}{4}
  (\lambda+\Delta^2) u_1' & = 0 && \text{ in } \R^n_+,\\
  \partial_\nu u_1' + (\lambda-\Delta') u_2 & = g' && \text{ on }\R^{n-1},\\
  \partial_\nu \Delta u_1' & = 0 && \text{ on }\R^{n-1},\\
   \gamma_0 u_1' & = u_2 && \text{ on }\R^{n-1}
\end{alignedat}
\end{equation}
such that the solution of~\eqref{4-4} is given by $u=(u_1,u_2)$ with $u_1=u_1'+u_1^0$.
Here, we have set $g' := g - \partial_\nu u_1^0$. Note that
\[ \|\partial_\nu u_1^0\|_{L^p(\R^{n-1})}
\le C \|\partial_\nu u_1^0\|_{B_{pp}^{3-1/p}(\R^{n-1})} \le
C \|u_1^0\|_{H_p^4(\R^n_+)} \le C \|f\|_{L^p(\R^n_+)}\]
and therefore
\[  \|g'\|_{L^p(\R^{n-1})}\le C\big(\|g\|_{L^p(\R^{n-1})} + \|f\|_{L^p(\R^n_+)}\big)\le C\|(f,g)\|_X.\]
%\blau{For \grun{the} sake of simplicity \grun{we write again $u_1$ and $g$ instead of $u_1'$ and $g'$}. \rot{$g'$ taucht unten aber wieder auf. Gilt $u_2'=u_2$?}
With the same calculations as those leading up to \eqref{4-1'}, we observe that the boundary value problem~\eqref{4-9} possesses a unique solution $(u_1',u_2)$ satisfying $\hat{u}_2=S(|\xi'|,\lambda) \hat g'$ and therefore $u_2 = S(|D'|,\lambda) g'$.
% To solve \eqref{4-9}, we take partial Fourier transform $(\mathscr F' u_1)(\xi',x_n)
%  =: w(\xi',x_n)$ and obtain the ODE \eqref{4-5} for $w(x_n) = w(\xi',x_n)$ as well
%  as \eqref{4-7} from the boundary condition $\partial_\nu \Delta u_1=0$. From the proof
%  of Lemma~\ref{4.2} we know that $w(x_n)=c_1e^{-\tau_1 x_n} + c_2e^{-\tau_2 x_n}$
%  and $\tau_1c_1=\tau_2c_2$, where $\tau_{1,2} =\tau_{1,2}(\xi',\lambda)$ are defined in
%  \eqref{4-6}. Inserting this into the second line of \eqref{4-9}, we get
%  \[ \tau_1 c_1 + \tau_2 c_2 + (\lambda+|\xi'|^2) (c_1+c_2) = \hat g := \mathscr F'g.\]
%  With $c_2 = \frac{\tau_1}{\tau_2} c_1$, this yields
%  \[ c_1 = \frac{\tau_2}{(\tau_1+\tau_2)(\lambda+|\xi'|^2)+2\tau_1\tau_2}\; \hat g .\]
%  For $\hat u_2 (\xi')= (\mathscr F' u_1)(\xi',0)$, we obtain
%  \[ \hat u_2 = c_1+c_2 = \frac{\tau_1+\tau_2}{(\tau_1+\tau_2)(\lambda+|\xi'|^2)+2\tau_1\tau_2}
%  \hat g =:
%  S(\lambda,\xi') \hat g.\]
%Therefore, $u_2 = S(|D'|,\lambda) g'$.
Since $m$ is bounded due to Lemma~\ref{4.4} and
 \begin{align*}
     (-\Delta')^{1/2}\colon L^p(\R^{n-1})\supset D((-\Delta')^{1/2})= W_p^1(\R^{n-1}) \to L^p(\R^{n-1})
 \end{align*}
  has a bounded $H^\infty$-calculus (see, e.g., \cite{Denk-Saal-Seiler08}, Corollary~2.10),
  the operator
  \begin{align*}
      m((-\Delta')^{1/2},\lambda)= (\lambda+\Delta') S(|D'|,\lambda)
  \end{align*}
  is
  well-defined and a bounded operator in $L^p(\R^{n-1})$. The operator norm can be
  estimated by a constant independent of %$\lambda$ for
 $\lambda\in\Sigma_\theta$.
% \rot{Ab hier doppelt}
% \blau{ We show that the norm of the operator $(\lambda+\Delta')S(|D'|,\lambda)$ in $L^p(\R^{n-1})$ is bounded. For this, we fix $\epsilon\in (0,\frac{\pi-\theta}4)$ \rot{$\eps$ doppelt belegt.} and write $(\lambda+|\xi'|^2)S(|\xi'|,\lambda) = m(|\xi'|,\lambda)$ %\rot{Hier sollte wahrscheinlich $(\lambda+|\xi'|^2)S(|\xi'|,\lambda) = m(|\xi'|,\lambda)$ stehen.}
% with $m$ being defined in \eqref{4-10}. By Lemma~\ref{4.4},
%   the function $m\colon \Sigma_\epsilon \times\Sigma_\theta\to\C$
%   is a bounded holomorphic function. It is known that
%   \[  (-\Delta')^{1/2}\colon L^p(\R^{n-1})\supset \blau{D((-\Delta')^{1/2})}= W_p^1(\R^{n-1}) \to L^p(\R^{n-1})\]
%   has a bounded $H^\infty$-calculus (see, e.g., \cite{Denk-Saal-Seiler08}, Corollary~2.10). Therefore
%   the operator $m((-\Delta')^{1/2},\lambda)= (\lambda+\Delta') S(|D'|,\lambda)$ is
%   well-defined and a bounded operator in $L^p(\R^{n-1})$ with operator norm which can be
%   estimated by a constant independent of $\lambda$ for
% all $\lambda\in\grun{\Sigma_\theta}$. %\rot{Gilt das tatsächlich im abgeschlossenen Sektor?}
% \rot{Doppelt Ende}
This shows that
\begin{align}\label{S_bounded}
    S(|D'|,\lambda)\colon L^p(\R^{n-1})\rightarrow H^2_p(\R^{n-1})
\end{align}
is continuous, and as $\frac{\lambda}{\lambda+z^2}=\frac{1}{1+\frac{z^2}{\lambda}}$ is a bounded holomorphic function as well, we also obtain
the boundedness of $\lambda S(|D'|, \lambda)$ on $L^p(\R^{n-1}).$

% \rot{Man könnte (ii) und (iii) zusammenfassen.}
%  We show that the norm of the operator $\lambda S(D',\lambda)$ in $L^p(\R^{n-1})$ is bounded.
%  For this, we fix $\epsilon\in (0,\frac{\pi-\theta}4)$ and write $\lambda S(\xi',\lambda) =
%  m(|\xi'|,\lambda)$ with $m$ being defined in \eqref{4-10}. By Lemma~\ref{4.4},
%   the function $m\colon \Sigma_\epsilon \times\Sigma_\theta\to\C$
%   is a bounded holomorphic function. It is known that
%   \[  (-\Delta')^{1/2}\colon L^p(\R^{n-1})\supset \blau{D((-\Delta')^{1/2})}= W_p^1(\R^{n-1}) \to L^p(\R^{n-1})\]
%   has a bounded $H^\infty$-calculus (see, e.g., \cite{Denk-Saal-Seiler08}, Corollary~2.10). Therefore
%   the operator $m((-\Delta')^{1/2},\lambda)= \lambda S(D',\lambda)$ is
%   well-defined and a bounded operator in $L^p(\R^{n-1})$, with operator norm which can be
%   estimated by a constant independent of $\lambda$ for
% all $\lambda\in\overline{\Sigma_\theta}$.

(iii) %From (ii) we see that \grun{$u_2 = \gamma_0 u_1'$ is uniquely determined for $\lambda\in\Sigma_\theta$}, and we have
With (ii) and $u_2=S(|D'|,\lambda)g'$ we get
\[ |\lambda|\, \|u_2\|_{L^p(\R^{n-1})} \le C \|g'\|_{L^p(\R^{n-1})}\le
%C\big(\|g\|_{L^p(\R^{n-1})} + \|f\|_{L^p(\R^n_+)}\big).
C\|(f,g)\|_X.\]
The function $u_1'$ in particular solves the problem
  \begin{equation}
    \label{4-12}
\begin{alignedat}{4}
  (\lambda+\Delta^2) u_1' & = 0 && \text{ in } \R^n_+,\\
   \gamma_0 u_1'  & = u_2 && \text{ on }\R^{n-1},\\
  \partial_\nu \Delta u_1' & = 0 && \text{ on }\R^{n-1}.
\end{alignedat}
\end{equation}
As this boundary value problem is parameter-elliptic due to Lemma~\ref{4.2}, we can apply
Theorem~\ref{2.2} with $s:=4, \, \sigma:= -4 + 1/(2p)$.
We use the embeddings
$L^p(\R^{n-1}) = H_{p,\lambda}^0 (\R^{n-1}) \subset B_{pp,\lambda}^{-1/(2p)}(\R^{n-1})$ and
$H_{p,\lambda}^{4, -4+1/(2p)} (\R^n_+)\subset H_{p,\lambda}^{1/(2p)}(\R^n_+)$ and obtain from
Theorem~\ref{2.2} that $u_1' \in H^{1/(2p)}_{p,\lambda}(\R^n_+)$ satisfies the estimate
\begin{align*}
|\lambda|^{1/(8p)} \|u_1'\|_{L^p(\R^n_+)} & \le C \|u_1'\|_{H_{p,\lambda}^{1/(2p)}(\R^n_+)}
\le C \|u_1'\|_{H_{p,\lambda}^{4, -4+1/(2p)}(\R^n_+)} \\
& \le C \|u_2\|_{B_{pp,\lambda}^{-1/(2p)}(\R^{n-1})} \le C
\|u_2\|_{L^p(\R^{n-1})}
\end{align*}
for $\lambda \in \Sigma_\theta$ with $|\lambda|\geq \lambda_0'$.
Altogether, $u=(u_1,u_2)$ with $u_1=u_1'+u_1^0$ is the unique solution of \eqref{4-4} and fulfills the uniform estimate $|\lambda| \,\|u\|_X\le C \|(f,g)\|_X$ for $\lambda \in \Sigma_\theta$ with $|\lambda|\geq \lambda_0'$.
% we obtain unique
% solvability of \eqref{4-4} as well as the uniform estimate $|\lambda| \,\|u\|_X\le C \|(f,g)^\top\|_X$
% for the solution $u=(u_1,u_2)^\top$.
Writing $\lambda-A= (\lambda-\lambda_0)-(A-\lambda_0)$ and recalling
\eqref{4-19}, we see that $A-\lambda_0$ generates a bounded analytic $C_0$-semigroup of angle $\frac\pi 2$
in $X$, and therefore $A$ generates an analytic $C_0$-semigroup of angle $\frac\pi2$ in $X$.
%\rot{Man könnte hier einen neuen Abschnitt im Beweis anfangen.}

(iv) From \eqref{S_bounded} we even know that $u_2=S(|D'|,\lambda)g'$ lies in $H^2_p(\R^{n-1})$.
Consequently, we can also apply Theorem~\ref{2.2} to \eqref{4-12} with $s=4$ and $\sigma=-2+\frac1p-\eps$. Hence, taking a fixed $\lambda\in \Sigma_\theta$, we obtain the desired higher regularity due to
% \blau{For the higher regularity of $u_1$ we proceed similarly as above. We solve \eqref{4-12} once more, but with $u_2 \in H^2_p(\R^{n-1})$. \rot{Es klingt ein bisschen so, als würde man die Gleichung für ein neues, besseres $u_2$ nochmal lösen, aber man wendet ja nur das Theorem für andere $s,\sigma$ auf das gleiche $u_2$ an, was geht, da wir wissen, dass $u_2\in H^2_p(\R^{n-1})$ gilt.} We choose $s=4$, $\sigma=-2+1/p-\eps$ and obtain for a fixed $\lambda$:
\begin{align*}
 \|u_1\|_{H_{p}^{2+1/p-\eps}(\R^n_+)}
&\le C_\lambda \|u_1\|_{H_{p,\lambda}^{4, -2+1/p-\eps}(\R^n_+)} \\
& \le  C_\lambda \|u_2\|_{B_{pp,\lambda}^{2-\eps}(\R^{n-1})} \le  C_\lambda
\|u_2\|_{H^2_p(\R^{n-1})}.
\end{align*}
For $p \geq 2$, the last embedding also holds for $\eps=0$.

\end{proof}

\begin{remark}\label{4.6}
In the above estimates we could show that \[ |\lambda|\, \|u_2\|_{L^p(\R^{n-1})} \le C  \|g'\|_{L^p(\R^{n-1})}\]%\big(\|g\|_{L^p(\R^{n-1})} + \|f\|_{L^p(\R^n_+)}\big)\]
holds for all $\lambda\in\Sigma_\theta$. The condition $|\lambda|\ge\lambda_0$ with arbitrary small $\lambda_0>0$ was
only used for the uniform estimate of $\|u_1\|_{L^p(\R^n_+)}$.
\end{remark}

\begin{example}
With exactly the same methods as for  \eqref{4-4}, one can also treat the more simple boundary
value problem with dynamics boundary condition given as
\begin{equation}
  \label{4-2}
  \begin{aligned}
    \lambda u_1 - \Delta u_1 & = f \quad\text{ in }  \R^n_+,\\
    \lambda u_2 + \partial_\nu u_1 & = g \quad\text{ on } \R^{n-1},\\
     \gamma_0  u_1  - u_2 & = 0 \quad\text{ on } \R^{n-1}.
  \end{aligned}
\end{equation}
The operator $A$ related to \eqref{4-2}
acts in the space $X := L^p(\R^n_+)\times L^p(\R^{n-1})$ and is defined by
$D(A):= \{ u=(u_1,u_2) \in X: A(D) u \in X,\, \gamma_0  u_1  - u_2  = 0  \}$, where
\[ A(D) u := \begin{pmatrix}
  \Delta  & 0 \\ -\partial_\nu & 0
\end{pmatrix} u \quad (u\in D(A)).\]
In the same way as above,
but with much simpler resolvent estimates, one sees that $A-\lambda_0$
generates for every $\lambda_0>0$ a bounded holomorphic $C_0$-semigroup in $X$. The symbol which we have
to estimate now has the form
\[ m(z,\lambda) := \frac{\lambda}{\lambda + \sqrt{\lambda+z^2}}\]
for $(z,\lambda)\in\Sigma_\epsilon\times\Sigma_\theta$.
\end{example}

%\bibliography{dpsr}
%\bibliographystyle{abbrv}
%\rot{In Referenz [38] ist Sickel alphabetisch der erste Autor }

\end{document}